\documentclass[12pt]{amsart}

\usepackage{amsmath,amssymb}
\usepackage{bm}
\usepackage{graphicx}
\usepackage{ascmac}
\usepackage{amsthm}
\usepackage{amsfonts}
\usepackage{mathrsfs}
\usepackage{mathtools}
\usepackage[all]{xy}
\usepackage{comment}
\usepackage{DotArrow}

\theoremstyle{plain}
\newtheorem{thm}{Theorem}[section]
\newtheorem{cor}[thm]{Corollary}
\newtheorem{prop}[thm]{Proposition}
\newtheorem{lem}[thm]{Lemma}

\theoremstyle{definition}
\newtheorem{dfn}[thm]{Definition}

\newtheorem{rem}[thm]{Remark}

\newtheorem*{nota}{Notations}
\newtheorem*{claim}{Claim}

\newtheorem*{oftp}{Organization of this paper}
\newtheorem*{ack}{Acknowledgments}
\renewcommand{\qed}{\hfill$\blacksquare$\par}

\numberwithin{thm}{section}
\numberwithin{equation}{section}

\newcommand{\Zpn}{\mathbb{Z}_{>0}}
\newcommand{\Znn}{\mathbb{Z}_{\geq 0}}
\newcommand{\Z}{\mathbb{Z}}

\newcommand{\Q}{\mathbb{Q}}

\newcommand{\A}{\mathbb{A}}

\newcommand{\G}{\mathbb{G}}

\newcommand{\sS}{\mathscr{S}}

\renewcommand{\P}{\mathbb{P}}
\newcommand{\bK}{\mathbf{K}}

\newcommand{\bd}{\mathbf{d}}

\newcommand{\cH}{\mathcal{H}}
\newcommand{\cI}{\mathcal{I}}

\newcommand{\Hhat}{\widehat{H}}

\DeclareMathOperator{\Ker}{Ker}
\DeclareMathOperator{\Ima}{Im}
\DeclareMathOperator{\Coker}{Coker}
\DeclareMathOperator{\Hom}{Hom}

\DeclareMathOperator{\id}{id}

\DeclareMathOperator{\Gal}{Gal}

\DeclareMathOperator{\N}{N}

\DeclareMathOperator{\red}{red}

\DeclareMathOperator{\Res}{Res}
\DeclareMathOperator{\Cor}{Cor}

\DeclareMathOperator{\rk}{rank}

\DeclareMathOperator{\Pic}{Pic}

\DeclareMathOperator{\sep}{sep}

\DeclareMathOperator{\Ext}{Ext}
\DeclareMathOperator{\Ind}{Ind}
\DeclareMathOperator{\pr}{pr}

\DeclareMathOperator{\tor}{tor}

\DeclareMathOperator{\fl}{fl}

\DeclareMathOperator{\nor}{nor}
\DeclareMathOperator{\set}{set}
\DeclareMathOperator{\srd}{srd}

\usepackage[OT2,T1]{fontenc}
\DeclareSymbolFont{cyrletters}{OT2}{wncyr}{m}{n}
\DeclareMathSymbol{\Sha}{\mathalpha}{cyrletters}{"58}

\usepackage{geometry}
\usepackage{xcolor}

\title[Rationality problem]{The rationality problem for multinorm one tori}

\author[S.~Hasegawa]{Sumito Hasegawa}
\address{(Hasegawa)}
\email{sumito.hasegawa@gmail.com}

\author[K.~Kanai]{Kazuki Kanai}
\address{(Kanai) Department of Computer and Science, Ibaraki University, 4-12-1, Nakanarusawa-cho, Hitachi, Ibaraki, 316-8511, Japan}
\email{kazuki.kanai.du62@vc.ibaraki.ac.jp}

\author[Y.~Oki]{Yasuhiro Oki}
\address{(Oki) Department of Mathematics, College of Science, Rikkyo University, 3-34-1, Nishi-Ikebukuro, Toshima-ku, Tokyo, 171-8501, Japan. }
\email{oki@rikkyo.ac.jp}


\begin{document}

\begin{abstract}
In this paper, we study the rationality problem for multinorm one tori, a natural generalization of norm one tori. 
For multinorm one tori that split over finite Galois extensions with nilpotent Galois group, we prove that stable rationality and retract rationality are equivalent, and give a criterion for the validity of the above two conditions. 
This generalizes the result of Endo (2011) on the rationality problem for norm one tori. 
To accomplish it, 
we introduce a generalization of character groups of multinorm one tori.
Moreover, we establish systematic reduction methods originating in work of Endo (2001) for an investigation of the rationality problem for arbitrary multinorm one tori.
In addition, we provide a new example for which the multinorm principle holds.
\end{abstract}

\maketitle

\tableofcontents

\section{Introduction}

Let $k$ be a field and $k^{\sep}$ a fixed separable closure of $k$.
In algebraic geometry, a fundamental problem is to determine whether a given algebraic variety over $k$ is rational, that is, birationally equivalent to a projective space over $k$. 
It is also important to determine stable rationality, retract rationality, and unirationality, which are weaker notions of rationality. 
These properties satisfy:
\begin{center}
rational $\Rightarrow$ stably rational $\Rightarrow$ 
retract rational $\Rightarrow$ unirational. 
\end{center}
We remark that the direction of the implication cannot be reversed.
See Section \ref{sect:rttr} for more details.

The multinorm one tori primarily studied in this paper are algebraic tori. 
We recall that an algebraic torus over $k$ is a group $k$-scheme $T$ that satisfies
$T\otimes_k k^{\sep}\cong (\G_{m,k^{\sep}})^n$ for some non-negative integer $n$.
Note that an algebraic $k$-torus $T$ is always unirational 
(see \cite[p.~40, Example 21]{Voskresenskii1998})
and Voskresenskii conjectured that stably rational tori are rational (see \cite[p.~68]{Voskresenskii1998}). 
In this paper, we focus on the study of stable and retract rationality. 

The rationality problem is well-understood for tori of small dimensions. 
It is known by Voskresenskii \cite{Voskresenskii1967} that all tori of dimension $2$ are rational. Moreover, Kunyavskii \cite{Kunyavskii1990} solved the rationality problem for $3$-dimensional algebraic $k$-tori. 
After that, Hoshi--Yamasaki \cite{Hoshi2017} classified algebraic $k$-tori of dimensions $4$ and $5$ that are stably rational (resp.~retract rational). 

In this paper, we restrict our attention to the stable and retract rationality of \emph{multinorm one tori}. 
Let $\bK$ be a finite {\'e}tale algebra over $k$, that is, a finite product of finite separable field extensions of $k$ which are contained in $k^{\sep}$. Then, we set
\begin{equation*}
    T_{\bK/k}:=\Ker(\N_{\bK/k}\colon \Res_{\bK/k}\G_{m}\rightarrow \G_{m}),
\end{equation*}
where $\Res_{\bK/k}$ is the Weil restriction. 
We call it the multinorm one torus associated to $\bK/k$. If $\bK$ is a field, then $T_{\bK/k}$ is called the \emph{norm one torus}. Note that $T_{\bK/k}$ has dimension $\dim_{k}(\bK)-1$, and splits over the Galois closure of the composite field of all factors of $\bK$. This means that there is an isomorphism of $L$-tori $T_{\bK/k}\otimes_{k}L\cong \G_{m,L}^{\dim_{k}(\bK)-1}$, where $\bK$ is the product of finite separable field extensions $K_{1},\ldots,K_{r}$ of $k$, and $L$ is the Galois closure of the composite field of $K_{1},\ldots,K_{r}$ over $k$. 

The rationality problem for norm one tori has been extensively investigated in \cite{Endo1975}, \cite{ColliotThelene1977}, \cite{Hurlimann1984}, \cite{ColliotThelene1987}, \cite{LeBruyn95}, \cite{CortellaKunyavskii2000}, \cite{LemireLorenz2000}, \cite{Florence}, \cite{Endo2011},
\cite{Hoshi2017}, \cite{Hasegawa2020}, \cite{Hoshi2021}, 
\cite{Hoshi2024} and \cite{Hoshi}.
However, less is known about the rationality problem for multinorm one tori than about norm one tori. 
The problem has been studied in \cite{Hurlimann1984}, \cite{Endo2001}, and also treated in \cite{ColliotThelene1977}, \cite{CortellaKunyavskii2000}, \cite{Endo2011}.

As a motivation for studying the rationality problem for multinorm one tori, 
it is expected that this problem has applications to the rationality problem for norm one tori.
Consider the norm one torus $T_{K/k}$ associated to a finite separable field extension $K/k$. 
Let $K'/k$ be a finite separable field extension. 
By definition, there is an isomorphism of $K'$-tori
\begin{equation*}
T_{K/k}\otimes_{k}K'\cong T_{(K\otimes_{k}K')/K'}. 
\end{equation*}
Here $K\otimes_{k}K'$ may not be a finite separable field extension of $K'$; however, it is a finite {\'e}tale algebra over $K'$. 
This means that multinorm one tori appear by taking base change of norm one tori. 
Moreover, if $T_{(K\otimes_{k}K')/K'}$ is \emph{not} rational (resp.~stably rational; retract rational) over $K'$, then one can prove that $T_{K/k}$ is so over $k$. In fact, this approach is used in the proof of \cite[{$\lbrack$I$\rbrack$}]{Endo2011} to obtain the non-retract rationality of some norm one tori. 

The main theorem of Endo (\cite{Endo2001}) can be stated in our notation as follows:

\begin{thm}[{\cite[Theorem 2]{Endo2001}}]\label{Endo01Theorem2}
Let $p$ be a prime, $k$ a field, $\bK=\prod_{i=1}^{r}K_{i}^{h_i}$ a finite {\'e}tale $k$-algebra with $r\geq 1$ and $h_i\geq 1$, and $L$ the Galois closure of the composite field of $K_{1},\ldots,K_{r}$ over $k$. 
Assume
\begin{itemize}
\item $\Gal(L/k)$ is an elementary $p$-abelian group;
\item $K_{i}\neq K_{j}$ for any $i,j \in \{1,\ldots,r\}$ with $i\neq j$; and
\item $[K_i:k]=p$ for any $i\in \{1,\ldots,r\}$.
\end{itemize}
\begin{itemize}
    \item[(1)] If $p \neq 2$, then the following are equivalent:
    \begin{enumerate}
        \item $T_{\bK/k}$ is stably rational over $k$; 
        \item $T_{\bK/k}$ is retract rational over $k$;
        \item $r=1$.
    \end{enumerate}
    \item[(2)] If $p=2$, then the following are equivalent:
     \begin{enumerate}
        \item $T_{\bK/k}$ is stably rational over $k$; 
        \item $T_{\bK/k}$ is retract rational over $k$;
        \item $r=1$ or $2$.
    \end{enumerate} 
\end{itemize}
\end{thm}

Theorem \ref{Endo01Theorem2} states that the stable rationality is determined solely by the number $r$ of direct factors, which is both simple and interesting.
Moreover, it naturally raises the question of how this extends to more general cases.

\subsection{Main theorems}

Our first main theorems generalize Theorem \ref{Endo01Theorem2} to the case where $\Gal(L/k)$ is a $p$-group.

\begin{thm}\label{mth1}
Let $p$ be an odd prime number, $k$ a field, $\bK=\prod_{i=1}^{r}K_{i}$ a finite {\'e}tale $k$-algebra with $r \geq 1$, 
and $L$ the Galois closure of the composite field of $K_{1}, \ldots, K_{r}$ over $k$.
Assume
\begin{itemize}
    \item $K_{i}\not\subset K_{j}$ for any $i,j \in \{1,\ldots,r\}$ with $i\neq j$; and
    \item $[L:k]$ is a power of $p$. 
\end{itemize}
Then the following are equivalent: 
\begin{enumerate}
\item $T_{\bK/k}$ is stably rational over $k$; 
\item $T_{\bK/k}$ is retract rational over $k$; 
\item $r=1$ and $L$ is cyclic over $k$. 
\end{enumerate}
\end{thm}

We denote by $D_{n}$ the dihedral group of order $2n$, that is, 
\begin{equation*}
    D_{n}:=\langle \sigma_{n},\tau_{n}\mid \sigma_{n}^{n}=\tau_{n}^{2}=1,\tau_{n}\sigma_{n}\tau_{n}^{-1}=\sigma_{n}^{-1}\rangle. 
\end{equation*}
Note that there is an isomorphism $D_{2}\cong (C_{2})^{2}$. 

\begin{thm}\label{mth2}
Let $k$ be a field, $\bK=\prod_{i=1}^{r}K_{i}$ a finite {\'e}tale $k$-algebra with $r \geq 1$, 
and $L$ the Galois closure of the composite field of $K_{1}, \ldots, K_{r}$ over $k$.
Assume
\begin{itemize}
    \item all conjugates of $K_{i}$ are not contained in $K_{j}$ for any $i,j \in \{1,\ldots,r\}$ with $i\neq j$; and 
    \item $[L:k]$ is a power of $2$.
\end{itemize}
Then the following are equivalent: 
\begin{enumerate}
\item $T_{\bK/k}$ is stably rational over $k$; 
\item $T_{\bK/k}$ is retract rational over $k$; 
\item $\bK$ satisfies \emph{(iii-a)} or \emph{(iii-b)} as follows:
\begin{enumerate}
\item[(iii-a)] $r=1$ and $L$ is cyclic over $k$; 
\item[(iii-b)] $r=2$ and there exist $\nu \in \Zpn$ and an isomorphism $\Gal(L/k)\cong D_{2^{\nu}}$ that induces $\Gal(L/K_{1})\cong \langle \sigma_{2^{\nu}}^{2m_{1}}\tau_{2^{\nu}}\rangle$ and $\Gal(L/K_{2})\cong \langle \sigma_{2^{\nu}}^{2m_{2}+1}\tau_{2^{\nu}}\rangle$ for some $m_{1},m_{2}\in \Z$.
\end{enumerate}
\end{enumerate}
\end{thm}

\begin{rem}
In the general setup we allow $\bK=\prod_{i=1}^{r} K_i^{h_i}$ with $h_i\in\mathbb{Z}_{>0}$, and $K_{j}$ may contain a conjugate of $K_{i}$ for some $i,j$ with $i\neq j$.
However, by Proposition \ref{prop:rdsr}, we may drop multiplicities and remove any factor that contains another factor up to conjugacy. 
Accordingly, we state Theorem \ref{mth2} in this reduced form.
\end{rem}

The main theorem in the nilpotent case, which will be presented as Theorem \ref{mth3}, follows from Theorem \ref{mth1} and Theorem \ref{mth2}. As we will explain later, the proof is non-trivial.

We prepare some notions to state Theorem \ref{mth3}. By a \emph{quasi-trivial torus} (or an \emph{induced torus}) over $k$, we mean a $k$-torus that is isomorphic to $\Res_{\bK/k}\G_{m}$ for some finite {\'e}tale algebra $\bK$ over $k$. The notion of quasi-trivial tori is introduced in \cite[\S2, p.~187]{ColliotThelene1977}. 
In addition, a \emph{subalgebra} of $\bK$ refers to a finite product $\prod_{j=1}^{s}K'_{i_{j}}$, where $1\leq i_{1}<\cdots<i_{s}\leq r$ are integers and $K'_{i_{j}}$ is an intermediate field of $K_{i_{j}}/k$ for each $j\in \{1,\ldots,s\}$.

\begin{thm}\label{mth3}
Let $k$ be a field, $\bK=\prod_{i=1}^{r}K_{i}$ a finite {\'e}tale $k$-algebra with $r\geq 1$, and $L$ the Galois closure of the composite field of $K_{1},\ldots,K_{r}$ over $k$. 
Assume
\begin{itemize}
\item $\Gal(L/k)$ is nilpotent.
\end{itemize}
Then the following are equivalent:
\begin{enumerate}
\item $T_{\bK/k}$ is stably rational over $k$; 
\item $T_{\bK/k}$ is retract rational over $k$; 
\item there exists an isomorphism of $k$-tori
\begin{equation*}
    T_{\bK/k}\times S\cong T_{\bK'/k}\times S',
\end{equation*}
where $S$ and $S'$ are quasi-trivial tori over $k$, and $\bK'=\prod_{i=1}^{r'}K'_{i}$ is a finite product of intermediate fields $K'_{i}$ of $L/k$ that satisfies \emph{(iii-a)} or \emph{(iii-b)} as follows:
\begin{enumerate}
    \item[(iii-a)] $r'=1$ and $L'/k$ is cyclic over $k$; 
    \item[(iii-b)] $r'=2$ and there exist $m\in \Zpn \setminus 2\Z$, $\nu \in \Zpn$ and an isomorphism $\Gal(L'/k)\cong C_{m}\times D_{2^{\nu}}$ that induces $\Gal(L'/K'_{1})\cong \langle \sigma_{2^{\nu}}^{2m_{1}}\tau_{2^{\nu}}\rangle$ and $\Gal(L'/K'_{2})\cong \langle \sigma_{2^{\nu}}^{2m_{2}+1}\tau_{2^{\nu}}\rangle$ for some $m_{1},m_{2}\in \Z$.
\end{enumerate}
Here, $L'$ is the Galois closure of the composite field of $K'_{1},\ldots,K'_{r'}$ over $k$. 
\end{enumerate}
\end{thm}

The condition (iii) implies that $T_{\bK/k}$ and $T_{\bK'/k}$ are stably birationally equivalent over $k$. 

\begin{rem}
We also give an explicit construction of $\bK'$ in Theorem \ref{mth3} (iii) using group theory. See Lemma \ref{lem:psdp}. 
\end{rem}

\begin{rem}
Theorem \ref{mth3} in the case $r=1$ is a consequence of the results of Endo--Miyata (\cite[Theorem 1.5, Theorem 2.3]{Endo1975}) and Endo (\cite[Theorem 2.1]{Endo2011}). 
For a norm one torus associated with a non-Galois extension $K/k$ whose Galois closure is nilpotent, it is not retract rational (\cite[Theorem 2.1]{Endo2011}).
On the other hand, by extending the scope to multinorm tori, we obtain a new stably rational family as in (iii-b) of Theorem \ref{mth3}.
\end{rem}

The following is a byproduct of our proof of Theorem \ref{mth2}; in particular, we obtain stably rational examples even when the Galois group is not nilpotent.

\begin{thm}\label{mth4}
Let $k$ be a field, and $\bK=K_{1}\times K_{2}$ a finite {\'e}tale algebra over $k$. Assume that 
\begin{itemize}
    \item $K_{1}K_{2}/k$ is Galois with Galois group $D_{2m}$ for some $m\in \Zpn$; 
    \item $[K_{1}K_{2}:K_{i}]=2$ and $K_{i}/k$ is non-Galois for each $i\in \{1,2\}$; and
    \item $K_{1}$ and $K_{2}$ are not conjugate to each other. 
\end{itemize}
Then the multinorm one torus $T_{\bK/k}$ is stably rational over $k$. 
\end{thm}

Our proofs of the main theorems are based on the study of \emph{character groups} of the corresponding tori. Let $T$ be a torus over $k$. Then the character group of $T$ is defined as follows: 
\begin{equation*}
X^{*}(T):=\Hom_{k^{\sep}\text{-groups}}(T\otimes_{k}k^{\sep},\G_{m,k^{\sep}}). 
\end{equation*}
Take a finite Galois extension $L/k$ with Galois group $G$ over which $T$ splits, which is possible in any case. Then $X^{*}(T)$ is a $G$-lattice, that is, a finitely generated free abelian group equipped with an action of $G$. On the other hand, there exist two notions for $G$-lattices: \emph{quasi-permutation} and \emph{quasi-invertible}. One can confirm the following: 
\begin{align*}
T\textrm{ is stably rational over }k
\Rightarrow T\textrm{ is retract rational over }k
\\
\Updownarrow\hspace*{150pt} \Updownarrow\hspace*{60pt}  \\
X^{*}(T)\textrm{ is quasi-permutation} 
\Rightarrow
X^{*}(T)\textrm{ is quasi-invertible. }
\end{align*}
Moreover, we can determine whether a $G$-lattice $M$ is quasi-permutation or quasi-invertible by using a \emph{flabby resolution} of $M$. This efficient technique was introduced by Endo--Miyata \cite{Endo1975} and Voskresenskii \cite{Voskresenskii1969}, and further developed by Colliot-Th{\'e}l{\`e}ne--Sansuc \cite{ColliotThelene1977}. 
The details of this will be discussed in Section \ref{sect:rttr}. In particular, our theorems are reduced to determining whether a $G$-lattice $J_{G/\cH}$ is quasi-permutation or quasi-invertible, in the case where $G$ is a finite nilpotent group. Here, $\cH$ is a multiset of subgroups of $G$, and $J_{G/\cH}$ is a $G$-lattice which fits into the exact sequence
\begin{equation*}
    0\rightarrow \Z \xrightarrow{(\varepsilon_{G/H}^{\circ})_{H\in \cH}} \bigoplus_{H\in \cH}\Z[G/H] \rightarrow J_{G/\cH} \rightarrow 0. 
\end{equation*}
See Section \ref{ssec:mnot} for details. 

We first give a sketch of our proof of Theorem \ref{mth1}, which is in the case that $G$ is a $p$-group with $p>2$. Then, one can construct an inductive argument by considering the restriction of $J_{G/\cH}$ to a certain maximal subgroup of $G$ (Proposition \ref{prop:rtmc}). On the other hand, for a proof of Theorem \ref{mth2}, which discusses the case that $G$ is a $2$-group, consideration of the above restriction argument alone is insufficient. To overcome this issue, we introduce a generalization of $J_{G/\cH}$, which is denoted by $J_{G/\cH}^{(\varphi)}$ in this paper (see Section \ref{ssec:ijdf}). We use this notion to
study some cases in which $\cH$ consists of normal subgroups of indices $2$ or $4$. See Section \ref{ssec:tgs2}. This provides the first step of induction, similar to the proof of Theorem \ref{mth1}. Finally, Theorem \ref{mth3} follows from Theorems \ref{mth1}, \ref{mth2} and some reduction methods on $J_{G/\cH}^{(\varphi)}$ that will be given in Section \ref{ssec:rdsm}. Note that one of these methods, which will be presented as Lemma \ref{lem:rded}, is a generalization of \cite[Proposition 1.3]{Endo2011}. 

\subsection{Application to multinorm principle}

As another motivation for studying the rationality problem for multinorm one tori, we discuss their applications to the multinorm principle.

Here we assume that $k$ is a global field. For a finite {\'e}tale algebra $\bK$ over $k$, let
\begin{equation*}
    \Sha(\bK/k):=(\N_{K/k}(\A_{\bK}^\times)\cap k^\times)/\N_{\bK/k}(\bK^\times). 
\end{equation*}
Here, $\A_{\bK}^{\times}$ is the product of the id{\`e}le groups of all the factors of $\bK$.
We say that the \emph{multinorm principle holds for $K/k$} if 
\begin{equation*}
    \Sha(\bK/k)=1. 
\end{equation*}
This question has also been the subject of extensive study, for example, \cite{Hurlimann1984}, \cite{Demarche2014}, \cite{BayerFluckiger2019}, \cite{Lee2022}, \cite{Huang2024}, and \cite{Liang2024}. 

The study of $\Sha(\bK/k)$ in the case that $\bK$ is a field is one of the classic problems in algebraic number theory. For example, $\Sha(K/k)=1$ holds for any finite cyclic extension $K/k$. This result is known as Hasse's norm theorem (\cite[Satz, p.~64]{Hasse1931}).

\begin{thm}\label{mth5}
Let $k$ be a global field, and $\bK=K_{1}\times K_{2}$ a finite {\'e}tale algebra over $k$. We further assume that
\begin{itemize}
\item[(a)] $K_{1}K_{2}$ is Galois over $k$; and
\item[(b)] $\Gal(K_{1}K_{2}/k)\cong C_{m}\times D_{2^{\nu}}$ with $m\in \Zpn \setminus 2\Z$ and $\nu\in \Zpn$; and
\item[(c)] the isomorphism in \emph{(b)} induces $\Gal(K_{1}K_{2}/K_{i})\cong \{1\}\times \langle \sigma_{2^{\nu}}^{2m_{i}+i}\tau_{2^{\nu}}\rangle$ with $m_{i}\in \Z$ for each $i\in \{1,2\}$. 
\end{itemize}
Then, we have $\Sha(\bK/k)=1$. 
\end{thm}

Theorem \ref{mth5} gives a new example of the validity of the multinorm principle. It follows from Theorem \ref{mth3} together with the standard argument on the multinorm principle. Note that our proof uses an extension of Ono's theorem (\cite{Ono1963}), that is, a connection between $\Sha(\bK/k)$ and the Tate--Shafarevich group of $T_{\bK/k}$. 
See Section \ref{sect:pfmn} for more details.

\begin{oftp}
In Section \ref{sect:rttr}, we prepare some basic definitions and known results about the rationality of algebraic tori. 
In particular, we discuss the relationship between algebraic tori and $G$-lattices.
In Section \ref{sect:mult}, we introduce the concept of multinorm tori, and provide a generalization of their corresponding $G$-lattices. 
Furthermore, we develop the technique of Endo \cite{Endo2011}, and construct some reduction methods to investigate the rationality problem for arbitrary multinorm one tori.
In Section \ref{sect:pgrp}, we review some properties of $p$-groups used in this paper.
In Section \ref{sect:pfod}, we give a proof of Theorem \ref{mth1}. 
In Section \ref{sect:flrs}, 
for certain $G$-lattices, we determine whether they are quasi-permutation or not quasi-invertible using the theory of flabby resolutions. 
These lattices play a crucial role in Section \ref{sect:pfev} and Section \ref{sect:nilp}.
In Section \ref{sect:pfev}, we give a proof of Theorem \ref{mth2} in four steps.
In Section \ref{sect:nilp}, we give a proof of Theorem \ref{mth3}. 
That is, we characterize stable rationality and retract rationality for multinorm one tori that split over finite Galois extensions with nilpotent Galois groups. 
Finally, Section \ref{sect:pfmn} gives a proof of Theorem \ref{mth5}, which gives an application of our study of the rationality problem for multinorm one tori. 
\end{oftp}

\begin{ack}
The authors would like to thank Seidai Yasuda for his helpful comments on this paper. The third-named author was carried out with the support of the JSPS Research Fellowship for Young Scientists and KAKENHI Grant Number JP22KJ0041.
\end{ack}

\begin{nota}
Let $G$ be a finite group. 
\begin{itemize}
\item For a subgroup $H$ of $G$, we define $N^{G}(H)$ and $N_{G}(H)$ as follows: 
\begin{equation*}
N^{G}(H):=\bigcap_{g\in G}gHg^{-1},\quad
N_{G}(H):=\{g\in G\mid gHg^{-1}=H\}. 
\end{equation*}
Note that $N^{G}(H)$ and $N_{G}(H)$ are called the \emph{normal core} of $H$ and the \emph{normalizer} of $H$ in $G$, respectively. 
\item A \emph{$G$-lattice} refers to a finitely generated free abelian group equipped with a left action of $G$. For a $G$-lattice $M$, the dual lattice of $M$ is denoted by
\begin{equation*}
    M^{\circ}:=\Hom_{\Z}(M,\Z). 
\end{equation*}
Here we define a left action of $G$ on $M^{\circ}$ as
\begin{equation*}
    G\times M^{\circ} \rightarrow M^{\circ};\,(g,f)\mapsto [x\mapsto f(g^{-1}x)]. 
\end{equation*}
\item Let $M$ be a $G$-lattice. For a normal subgroup $N$ of $G$, we define $G/N$-lattices $M^{N}$ and $M^{[N]}$ as follows: 
\begin{equation*}
M^{N}:=\{x\in M\mid n(x)=x\text{ for all }n\in N\},\quad 
M^{[N]}:=\left((M^{\circ})^{N}\right)^{\circ}. 
\end{equation*}
Note that $M^{[N]}$ is isomorphic to $M_{N}/M_{N,\tor}$, where $M_{N}$ is the $N$-coinvariant part of $M$, and $M_{N,\tor}$ is the torsion part of $M_{N}$. Note that $M^{N}$ and $M^{[N]}$ may not coincide in general. 
\end{itemize}
\end{nota}

\section{Basic facts on the rationality of tori}\label{sect:rttr}

Let $k$ be a field. For a non-negative integer $n$, we denote by $\P_{k}^{n}$ the projective space of dimension $n$ over $k$. Consider an algebraic variety $X$ over $k$. We say that $X$ is
\begin{itemize}
    \item \emph{rational} over $k$ if it is birationally equivalent to a projective space over $k$; 
    \item \emph{stably rational} over $k$ if $X\times_{k}\P_{k}^{m}$ is rational over $k$ for some $m\in \Znn$; 
    \item \emph{retract rational} over $k$ if there exist rational maps $f\colon \P_{k}^{n}\dashrightarrow X$ and $g\colon X\dashrightarrow \P_{k}^{n}$ with $n\in \Znn$ such that $f\circ g=\id_{X}$; 
    \item \emph{unirational} over $k$ if there is a dominant rational map from $\P_{k}^{n}$ to $X$ for some $n\in \Znn$. 
\end{itemize}

The notion of retract rationality was originally introduced by Saltman (\cite{Saltman1984}) in the case where $k$ is infinite (see also \cite{Kang2012}). It has been generalized for all varieties over arbitrary fields by Merkurjev (\cite{Merkurjev2017}). Note that one has implications
\begin{center}
rational $\Rightarrow$ stably rational $\Rightarrow$ 
retract rational $\Rightarrow$ unirational. 
\end{center}

In this paper, we concentrate on the case where $X$ is an algebraic torus over $k$. Fix a separable closure $k^{\sep}$ of $k$. Then, we can rephrase the stable and retract rationality of algebraic tori by means of $G$-lattices, where $G$ is a finite quotient of the Galois group of $k^{\sep}/k$. We follow the same terminology as the textbook \cite{Lorenz2005} and our principal reference \cite{Endo2011}. 

\begin{dfn}[{\cite[\S 1]{Endo2011}}]
Let $G$ be a finite group. We say that a $G$-lattice $M$ is
\begin{enumerate}
    \item \emph{permutation} if $M$ has a $\Z$-basis permuted by $G$, that is, $M\cong \bigoplus_{i=1}^{m}\Z[G/H_i]$ for some subgroups $H_1,H_2,\dots, H_m$;
    \item \emph{quasi-permutation} if there is an exact sequence of $G$-lattices
    \begin{equation*}
        0\rightarrow M \rightarrow R \rightarrow F \rightarrow 0, 
    \end{equation*}
    where $R$ and $F$ are permutation;
    \item \emph{quasi-invertible} if it is a direct summand of a quasi-permutation $G$-lattice. 
\end{enumerate}
\end{dfn}

It is not difficult to confirm that
\begin{center}
permutation $\Rightarrow$ quasi-permutation $\Rightarrow$ 
quasi-invertible. 
\end{center}

\begin{dfn}[{\cite[\S\S 2.3--2.5]{Lorenz2005}}]
Let $G$ be a finite group. We say that a $G$-lattice $M$ is
\begin{enumerate}
\item \emph{stably permutation} if $M\oplus R\cong R'$ for some permutation $G$-lattices $R$ and $R'$;
\item \emph{invertible} (or, \emph{permutation projective}) if it is a direct summand of a permutation $G$-lattice;
\item \emph{coflabby} if $H^{1}(H,M)=0$ for any subgroup $H$ of $G$;
\item \emph{flabby} if $M^{\circ}$ is coflabby. 
\end{enumerate}
\end{dfn}

It is known that the following hold: 
\begin{center}
    permutation $\Rightarrow$ stably permutation $\Rightarrow$ invertible $\Rightarrow$ flabby and coflabby. 
\end{center}
Here the rightmost implication is a consequence of \cite[(1.1) Proposition]{Lenstra1974}. 

\vspace{6pt}
Let $G$ be a finite group. We say that $G$-lattices $M_{1}$ and $M_{2}$ are \emph{similar} if there exist permutation $G$-lattices $R_{1}$ and $R_{2}$ such that $M_{1}\oplus R_{1}\cong M_{2}\oplus R_{2}$. We denote by $\sS(G)$ the set of similarity classes of $G$-lattices. For a $G$-lattice $M$, we write $[M]$ for the similarity class containing $M$. Then $\sS(G)$ is a commutative monoid with respect to the sum
\begin{equation*}
    [M_{1}]+[M_{2}]:=[M_{1}\oplus M_{2}]. 
\end{equation*}
By definition, for a $G$-lattice $M$, we have 
\begin{itemize}
    \item $[M]=0$ in $\sS(G)$ if and only if $M$ is stably permutation; and
    \item $[M]$ is invertible in $\sS(G)$ if and only if $M$ is invertible. 
\end{itemize}

\begin{dfn}[{\cite[\S 2.6]{Lorenz2005}}]
Let $G$ be a finite group, and $M$ a $G$-lattice. 
\begin{enumerate}
\item A \emph{coflabby resolution} of $M$ is an exact sequence of $G$-lattices
\begin{equation*}
0\rightarrow U \rightarrow R \rightarrow M \rightarrow 0, 
\end{equation*}
where $R$ is permutation and $U$ is coflabby. 
\item A \emph{flabby resolution} of $M$ is an exact sequence of $G$-lattices
\begin{equation*}
0\rightarrow M \rightarrow R \rightarrow F \rightarrow 0, 
\end{equation*}
where $R$ is permutation and $F$ is flabby. 
\end{enumerate}
\end{dfn}

There is a coflabby resolution for any $G$-lattice; see \cite[Lemma 1.1]{Endo1975}. 
This implies the existence of a flabby resolution of every $G$-lattice. Moreover, if 
\begin{equation*}
0\rightarrow M \rightarrow R \rightarrow F \rightarrow 0
\end{equation*}
is a flabby resolution of $M$, then the class $[F]$ in $\sS(G)$ depends only on $M$. In the sequel, we denote $[F]$ by $[M]^{\fl}$. It is known that the map
\begin{equation*}
    \sS(G)\rightarrow \sS(G);\,[M]\mapsto [M]^{\fl}
\end{equation*}
is an endomorphism of monoids, i.e.~$[M_1\oplus M_2]^{\fl}=[M_1]^{\fl}+[M_2]^{\fl}$.
In particular, we obtain implications as follows: 
\begin{center}
    stably permutation $\Rightarrow$ quasi-permutation,\quad invertible $\Rightarrow$ quasi-invertible. 
\end{center}

\begin{lem}[{\cite[Lemma 2.7.1 (a)]{Lorenz2005}; cf.~\cite[(1.2) Proposition]{Lenstra1974}}]\label{lem:ivmo}
Let $G$ be a finite group, and $F$ an invertible $G$-lattice. Then we have
\begin{equation*}
    [F]^{\fl}=-[F]. 
\end{equation*}
\end{lem}

The next proposition is not difficult; for the reader's convenience, we include a proof since a precise reference appears to be lacking.

\begin{prop}\label{prop:edlr}
Let $G$ be a finite group. 
\begin{enumerate}
\item A $G$-lattice $M$ is quasi-permutation if and only if $[M]^{\fl}=0$. 
\item A $G$-lattice $M$ is quasi-invertible if and only if $[M]^{\fl}$ is invertible. 
\end{enumerate}
\end{prop}

\begin{proof}
(i): It is clear that $[M]^{\fl}=0$ if $M$ is quasi-permutation. For the reverse implication, assume $[M]^{\fl}=0$. Take a flabby resolution of $M$: 
\begin{equation*}
    0\rightarrow M \rightarrow R \xrightarrow{\pi} F \rightarrow 0. 
\end{equation*}
By assumption, $F$ is stably permutation. Hence, there is a permutation $G$-lattice $R'$ such that $F\oplus R'$ is permutation. Moreover, the sequence
\begin{equation*}
    0\rightarrow M \xrightarrow{x\mapsto (\iota(x),0)} R\oplus R' \xrightarrow{\pi\oplus \id_{R'}} F\oplus R' \rightarrow 0 
\end{equation*}
is exact. This implies that $M$ is quasi-permutation as desired. 

(ii): We first prove that $[M]^{\fl}$ is invertible if $M$ is quasi-invertible. 
By assumption, there is a $G$-lattice $M'$ such that $M\oplus M'$ is quasi-permutation. Combining this result with (i), we obtain 
\begin{equation*}
[M]^{\fl}+[M']^{\fl}=[M\oplus M']^{\fl}=[0]. 
\end{equation*}
Hence $[M]^{\fl}$ is invertible. 

On the other hand, assume that $[M]^{\fl}$ is invertible. Take a flabby resolution
\begin{equation*}
    0\rightarrow M \rightarrow R \rightarrow F \rightarrow 0
\end{equation*}
of $M$, where $F$ is invertible by assumption. Since $F$ is invertible, we have $[F]^{\fl}=-[F]$ according to Lemma \ref{lem:ivmo}. In particular, we obtain an equality
\begin{equation*}
    [M\oplus F]^{\fl}=[0]. 
\end{equation*}
This is equivalent to the condition that $M\oplus F$ is quasi-permutation, which follows from (i). This completes the proof. 
\end{proof}

The following is a consequence of the definition. 

\begin{prop}\label{prop:rtrd}
Let $G$ be a finite group, and $H$ its subgroup. Consider a $G$-lattice $M$. If $M$ is a quasi-permutation (resp.~quasi-invertible) $G$-module, then $M$ is so as an $H$-lattice. 
\end{prop}

\begin{lem}[{\cite[p.~179, Lemme 2 (i), (ii), (iii)]{ColliotThelene1977}}]\label{lem:qtfx}
Let $G$ be a finite group, and $N$ its normal subgroup. Consider a $G$-lattice $M$. 
\begin{enumerate}
\item If $M$ is a permutation $G$-lattice, then $M^{N}$ is a permutation $G/N$-lattice. 
\item If $M$ is a coflabby $G$-lattice, then $M^{N}$ is a coflabby $G/N$-lattice. 
\item Let
\begin{equation*}
0\rightarrow U \rightarrow R \rightarrow M \rightarrow 0
\end{equation*}
be a coflabby resolution of $M$ in $G$-lattices. Then
\begin{equation*}
0\rightarrow U^{N} \rightarrow R^{N} \rightarrow M^{N} \rightarrow 0
\end{equation*}
is a coflabby resolution of $M^{N}$ in $G/N$-lattices. 
\end{enumerate}
\end{lem}

\begin{lem}[{\cite[p.~179, Lemme 2]{ColliotThelene1977}}]\label{lem:qtac}
Let $G$ be a finite group, and $N$ its normal subgroup. Consider a $G/N$-lattice $M$. 
\begin{enumerate}
\item The $G$-lattice $M$ is permutation (resp.~stably permutation; invertible; coflabby; flabby) if and only if it is so as a $G/N$-lattice. 
\item Let
\begin{equation*}
0\rightarrow U \rightarrow R \rightarrow M \rightarrow 0
\end{equation*}
be a coflabby resolution of $M$ in $G/N$-lattices. Then it is a coflabby resolution of $M$ in $G$-lattices. 
\end{enumerate}
\end{lem}

\begin{cor}\label{cor:qtac}
Let $G$ be a finite group, and $N$ its normal subgroup. Consider a $G$-lattice $M$. 
\begin{enumerate}
\item If the $G$-lattice $M$ is quasi-permutation (resp.~quasi-invertible), then the $G/N$-lattice $M^{[N]}$ is so. 
\item If $N$ acts on $M$ trivially, then $M$ is a quasi-permutation (resp.~quasi-invertible) $G$-lattice if and only if it is so as a $G/N$-lattice. 
\end{enumerate}
\end{cor}

\begin{proof}
(i) follows from Lemma \ref{lem:qtfx}. (ii) is a consequence of Lemma \ref{lem:qtfx} (iii) and Lemma \ref{lem:qtac}. 
\end{proof}

For a $G$-lattice $M$, we define
\begin{equation*}
\Sha_{\omega}^{2}(G,M):=\Ker\left(H^{2}(G,M)\rightarrow \bigoplus_{g\in G}H^{2}(\langle g\rangle,M)\right). 
\end{equation*}

\begin{prop}[{\cite[Proposition 2.9.2 (a)]{Lorenz2005}; cf.~\cite[Proposition 9.5 (ii)]{ColliotThelene1987}, \cite[Proposition 9.8]{Sansuc1981}}]\label{prop:rrs2}
Let $G$ be a finite group, and
\begin{equation*}
    0\rightarrow M \rightarrow R \rightarrow F \rightarrow 0
\end{equation*}
a flabby resolution of a $G$-lattice $M$. Then, there is an isomorphism
\begin{equation*}
    \Sha_{\omega}^{2}(G,M)\cong H^{1}(G,F). 
\end{equation*}
In particular, if $M$ is quasi-invertible, then we have $\Sha_{\omega}^{2}(G,M)=0$. 
\end{prop}

For a $G$-lattice $M$, we can summarize the properties mentioned above as follows:
\begin{align*}
\textrm{permutation}\ \ 
\Rightarrow\ \ 
&\textrm{stably permutation}\ \ 
\Rightarrow\ \ 
\hspace*{4mm}\textrm{invertible}\hspace*{5mm}\ \ 
\Rightarrow\ \ 
\textrm{flabby\ and\ coflabby}\\
&\hspace*{17mm}\Downarrow\hspace*{38mm} \Downarrow\hspace*{35mm} \\
&\textrm{quasi-permutation} 
\hspace*{4.5mm}\Rightarrow\hspace*{2mm} \textrm{quasi-invertible}
\hspace*{2mm}\Rightarrow\hspace*{2.5mm} \Sha^2_{\omega}(G,M)=0\\
&\hspace*{17mm}\Updownarrow\hspace*{38mm} \Updownarrow 
\ \ \\
&\hspace*{10mm}\textrm{$[M]^{\fl}=0$} 
\hspace*{11mm}\Rightarrow\hspace*{1mm} \textrm{$[M]^{\fl}$ is invertible}. \\
\end{align*}

For a torus $T$ over a field $k$, we define the cocharacter module $X_{*}(T)$ and the character module $X^{*}(T)$ as
\begin{equation*}
X_{*}(T):=\Hom_{k^{\sep}\text{-groups}}(\G_{m,k^{\sep}},T\otimes_{k}k^{\sep}),\quad 
X^{*}(T):=\Hom_{k^{\sep}\text{-groups}}(T\otimes_{k}k^{\sep},\G_{m,k^{\sep}}). 
\end{equation*}
These are finite free abelian groups equipped with continuous actions of $\Gal(k^{\sep}/k)$ (with respect to the discrete topology). 

\begin{prop}
\label{prop:rtiv}
Let $k$ be a field, and $T$ a torus over $k$ which splits over a finite Galois extension $L$ of $k$. Define $G:=\Gal(L/k)$. 
\begin{enumerate}
\item \emph{(\cite[Theorem 1.6]{Endo1973}, \cite[Proposition 9.5.3]{Lorenz2005})} The algebraic torus $T$ is stably rational over $k$ if and only if the $G$-lattice $X^{*}(T)$ is quasi-permutation. 
\item \emph{(\cite[Proposition 9.5.4]{Lorenz2005}, cf.~\cite[Theorem 3.14]{Saltman1984})} The algebraic torus $T$ is retract rational over $k$ if and only if the $G$-lattice $X^{*}(T)$ is quasi-invertible. 
\end{enumerate}
\end{prop}

Let $Y$ be a smooth algebraic variety over $k$. A smooth compactification of $Y$ over $k$ refers to a proper smooth algebraic variety $X$ over $k$ that admits an open immersion $Y\hookrightarrow X$. Note that a smooth compactification of $Y$ over $k$ always exists if $k$ has characteristic $0$, which is a consequence of Hironaka (\cite{Hironaka1964}). Moreover, Colliot-Th{\'e}l{\`e}ne, Harari and Skorobogatov (\cite{ColliotThelene2005}) gave the existence of smooth compactifications of all tori over arbitrary fields. 

\begin{prop}[{\cite[Section 4, p.~1213]{Voskresenskii1969}}]\label{prop:flpc}
Let $k$ be a field, and $T$ an algebraic torus over $k$ which splits over a finite Galois extension $L$ over $k$. Take a smooth compactification $X$ of $T$ over $k$, and define $\overline{X}:=X\otimes_{k}k^{\sep}$. Then there is an exact sequence of $\Gal(k^{\sep}/k)$-lattices
\begin{equation*}
    0\rightarrow X^{*}(T)\rightarrow R \rightarrow \Pic(\overline{X})\rightarrow 0, 
\end{equation*}
where $R$ is permutation and $\Pic(\overline{X})$ is flabby. In particular, we have $[X^{*}(T)]^{\fl}=[\Pic(\overline{X})]$. 
\end{prop}

Combining Proposition \ref{prop:flpc} with Proposition \ref{prop:rrs2}, we obtain the following. 

\begin{cor}[{cf.~\cite[Proposition 9.5 (ii)]{ColliotThelene1987}, \cite[Proposition 9.8]{Sansuc1981}}]\label{cor:crvs}
Let $k$ be a field, and $T$ an algebraic torus over $k$ which splits over a finite Galois extension $L$ of $k$. Take a smooth compactification $X$ of $T$ over $k$, and put $\overline{X}:=X\otimes_{k}k^{\sep}$. Then there is an isomorphism
\begin{equation*}
    H^{1}(k,\Pic(\overline{X}))\cong \Sha_{\omega}^{2}(G,X^{*}(T)),
\end{equation*}
where $G:=\Gal(L/k)$. 
\end{cor}

\section{Multinorm one tori and their character groups}\label{sect:mult}

\subsection{Multinorm one tori}\label{ssec:mnot}

Let $k$ be a field. Consider a finite {\'e}tale algebra $\bK=\prod_{i=1}^{r}K_{i}$ over $k$, that is, a finite product of finite separable subextensions of $k^{\sep}$. The \emph{multinorm one torus} associated to $\bK/k$ is defined as
\begin{equation*}
    T_{\bK/k}=\Ker(\N_{\bK/k}\colon \Res_{\bK/k}\G_{m}\rightarrow \G_{m}). 
\end{equation*}

Let $G$ be a finite group, and $H$ a subgroup of $G$. Then one has a surjection
\begin{equation*}
    \varepsilon_{G/H}\colon \Z[G/H]\rightarrow \Z;\,\sum_{g \in G/H} a_g g\mapsto \sum_{g \in G/H} a_g, 
\end{equation*}
which is called the \emph{augmentation map}. Moreover, the dual of $\varepsilon_{G/H}$ coincides with the homomorphism
\begin{equation*}
    \varepsilon_{G/H}^{\circ}\colon \Z\rightarrow \Z[G/H];\,1\mapsto \sum_{g\in G/H}g. 
\end{equation*}
On the other hand, consider a finite group $G$ and subgroups $H'\subset H$ of $G$. Then the homomorphisms of $H$-lattices
\begin{equation*}
   \varepsilon_{H/H'}\colon \Z[H/H']\rightarrow \Z,\quad \varepsilon_{H/H'}^{\circ}\colon \Z \rightarrow \Z[H/H']
\end{equation*}
induce homomorphisms of $G$-lattices
\begin{equation*}
   \Ind_{H}^{G}\varepsilon_{H/H'}\colon \Z[G/H']\rightarrow \Z[G/H],\quad \Ind_{H}^{G}\varepsilon_{H/H'}^{\circ}\colon \Z[G/H] \rightarrow \Z[G/H']. 
\end{equation*}

For a multiset $\cH$ of subgroups of $G$, we define a $G$-module $I_{G/\cH}$ by an exact sequence
\begin{equation*}
    0\rightarrow I_{G/\cH}\rightarrow \bigoplus_{H\in \cH}\Z[G/H]\xrightarrow{(\varepsilon_{G/H})_{H\in \cH}} \Z \rightarrow 0. 
\end{equation*}
Furthermore, we define $J_{G/\cH}:=I_{G/\cH}^{\circ}$. Then one has an exact sequence
\begin{equation*}
    0\rightarrow \Z \xrightarrow{(\varepsilon_{G/H}^{\circ})_{H\in \cH}} \bigoplus_{H\in \cH}\Z[G/H] \rightarrow J_{G/\cH} \rightarrow 0. 
\end{equation*}

\begin{prop}\label{prop:coch}
Let $k$ be a field, $\bK=\prod_{i=1}^{r}K_{i}$ a finite {\'e}tale algebra over $k$. Take a finite Galois extension $L$ of $k$ containing $K_{1},\ldots,K_{r}$. Put $G:=\Gal(L/k)$ and $\cH:=\{\Gal(L/K_{i})\mid i\in \{1,\ldots,r\}\}$. Then there are isomorphisms of $G$-modules
\begin{equation}\label{eq:tkjg}
    X_{*}(T_{\bK/k})\cong I_{G/\cH},\quad X^{*}(T_{\bK/k})\cong J_{G/\cH}. 
\end{equation}
In particular, $T_{\bK/k}$ is stably (resp.~retract) rational over $k$ if and only if the $G$-lattice $J_{G/\cH}$ is quasi-permutation (resp.~quasi-invertible). 
\end{prop}

\begin{proof}
The isomorphisms \eqref{eq:tkjg} follow from the construction of $I_{G/\cH}$ and $J_{G/\cH}$. The equivalence between $T_{\bK/k}$ being stably rational (resp.~retract rational) and $J_{G/\cH}$ being quasi-permutation (resp.~quasi-invertible) is a consequence of Proposition \ref{prop:rtiv}. 
\end{proof}

In this paper, we determine whether the $G$-lattice $J_{G/\cH}$ is quasi-permutation or quasi-invertible in order to classify the stable and retract rationality of multinorm one tori. 
In what follows, we discuss
\begin{itemize}
\item how to reduce the problem to a smaller $G$-lattice (Corollary \ref{cor:rdrd}, Proposition \ref{prop:rdsr},
Proposition \ref{prop:rdtr}, Corollary \ref{cor:rdmn}); 
\item behavior of $I_{G/\cH}$ and $J_{G/\cH}$ with respect to the restriction to subgroups of $G$ (Proposition \ref{prop:rtmc}); and
\item description of $I_{G/\cH}^{N}$ and $J_{G/\cH}^{[N]}$ for a normal subgroup $N$ of $G$ (Proposition \ref{prop:tkfx}). 
\end{itemize}
Accordingly, we extend the notions of character groups and cocharacter groups of multinorm one tori, and set up a framework for dealing with them.

\subsection{$G$-lattices $I_{G/\cH}^{(\varphi)}$ and $J_{G/\cH}^{(\varphi)}$}\label{ssec:ijdf}

Let $\cH$ be a multiset which is assumed to be finite and non-empty in the sequel. We use the following notation. 
\begin{itemize}
\item Denote by $\cH^{\set}$ the underlying set of $\cH$. 
\item For $H\in \cH^{\set}$, write $m_{\cH}(H)$ for the multiplicity of $H$ in $\cH$. Moreover, we set $m_{\cH}(H):=0$ if a subgroup $H$ of $G$ does not belong to $\cH^{\set}$. 
\end{itemize}

We define $\Delta$ as follows:
\begin{equation*}
    \Delta :=\coprod_{m \in \Zpn} \Delta_{m}. 
\end{equation*}
Here, $\Delta_{m}:=\{(d_{1},\ldots,d_{m})\in (\Znn)^{m}\mid d_{1}\leq \cdots \leq d_{m}\}$ for each $m\in \Zpn$. 
For $\bd \in \Delta_{m}$ and $i\in \{1,\ldots,m\}$, we denote by $\bd_{i}$ the $i$-th factor of $\bd$. 

\begin{dfn}
Let $\cH$ be a multiset. A \emph{weight function on $\cH$} is defined as a map
\begin{equation*}
\varphi \colon \cH^{\set}\rightarrow \Delta
\end{equation*}
such that $\varphi(H)\in \Delta_{m_{\cH}(H)}$ for any $H\in \cH^{\set}$. 
\end{dfn}

Let $a_{1},\ldots,a_{m}\in \Z$. If $a_{i}\neq 0$ for some $i$, we denote by $\gcd(a_{1},\ldots,a_{m})\in \Zpn$ the usual one, that is, the greatest common divisor of $a_{1},\ldots,a_{m}$. On the other hand, set
\begin{equation*}
\gcd(0,\ldots,0):=0. 
\end{equation*}

\begin{dfn}\label{dfn:dfdp}
Let $G$ be a finite group, $\cH$ a multiset of its subgroups, and $\varphi$ a weight function on $\cH$. 
\begin{enumerate}
\item We define a weight function $d_{\varphi}$ of $\cH^{\set}$ as follows: 
\begin{equation*}
d_{\varphi}\colon \cH^{\set} \rightarrow \Delta_{1};\,H \mapsto 
\gcd(\varphi(H)_{1},\ldots,\varphi(H)_{m_{\cH}(H)}). 
\end{equation*}
\item Put $d(\varphi):=\gcd(d_{\varphi}(H)\mid H\in \cH^{\set})$. We say that $\varphi$ is \emph{normalized} if $d(\varphi)=1$. 
\end{enumerate}
\end{dfn}

The following can be confirmed from the definition:

\begin{lem}\label{lem:pnor}
Let $G$ be a finite group, $\cH$ a multiset of its subgroups, and $\varphi$ a weight function on $\cH$ satisfying $d:=d(\varphi)\neq 0$. Then the map
\begin{equation*}
\varphi^{\nor} \colon \cH^{\set} \rightarrow \Delta;\,H \mapsto (d^{-1}\varphi(H)_{1},\ldots,d^{-1}\varphi(H)_{m_{\cH}(H)})
\end{equation*}
is a normalized weight function of $\cH$. 
\end{lem}

\begin{dfn}\label{dfn:gnmn}
Let $G$ be a finite group, and $\cH$ a multiset of its subgroups. Consider a weight function $\varphi$ of $\cH$. We define a $G$-lattice $I_{G/\cH}^{(\varphi)}$ by the exact sequence
\begin{equation}\label{eq:ipdf}
0\rightarrow I_{G/\cH}^{(\varphi)}\rightarrow \bigoplus_{H\in \cH^{\set}}\Z[G/H]^{\oplus m_{\cH}(H)}\xrightarrow{(\varphi(H)\cdot \varepsilon_{G/H})_{H}} \Z. 
\end{equation}
Here, $\varphi(H)\cdot \varepsilon_{G/H}\colon \Z[G/H]^{\oplus m_{\cH}(H)}\rightarrow \Z$ is defined as the sum of $\varphi(H)_{i}\cdot \varepsilon_{G/H}$ for all $i\in \{1,\ldots,m_{\cH}(H)\}$. Furthermore, set 
\begin{equation*}
J_{G/\cH}^{(\varphi)}:=(I_{G/\cH}^{(\varphi)})^{\circ}. 
\end{equation*}
\end{dfn}

The following follows from the definitions of $I_{G/\cH}^{(\varphi)}$ and $J_{G/\cH}^{(\varphi)}$. 

\begin{lem}\label{lem:zero}
Let $G$ be a finite group, $\cH$ a multiset of subgroups of $G$, and $\varphi$ a weight function of $\cH$. Assume that there exist $H_{0}\in \cH^{\set}$ and $i_{0}\in \{1,\ldots,m_{\cH}(H_{0})\}$ such that $\varphi(H_{0})_{i_{0}}=0$.  
\begin{enumerate}
\item If $\cH=\cH^{\set}=\{H_{0}\}$, then there exist isomorphisms
\begin{equation*}
    I_{G/\cH}^{(\varphi)}\cong \Z[G/H_{0}],\quad 
    J_{G/\cH}^{(\varphi)}\cong \Z[G/H_{0}]. 
\end{equation*}
\item Suppose that $\cH=\cH^{\set}=\{H_{0}\}$ fails. We define a multiset $\cH'$ of subgroups of $G$ as follows: 
\begin{equation*}
    m_{\cH'}(H')=
    \begin{cases}
        m_{\cH}(H_{0})-1&\text{if }H'=H_{0};\\
        m_{\cH}(H')&\text{if }H'\neq H_{0}. 
    \end{cases}
\end{equation*}
Moreover, denote by $\varphi'$ the weight function of $\cH'$ defined as
\begin{equation*}
    \varphi'(H')=
    \begin{cases}(\varphi(H_{0})_{1},\ldots,\varphi(H_{0})_{i_{0}-1},\varphi(H_{0})_{i_{0}+1},\ldots,\varphi(H_{0})_{m_{\cH}(H_{0})})&\text{if }H'=H_{0}\in (\cH')^{\set};\\
        \varphi(H')&\text{if }H'\neq H_{0}. 
    \end{cases}
\end{equation*}
Then there exist isomorphisms of $G$-lattices
\begin{equation*}
    I_{G/\cH}^{(\varphi)}\cong I_{G/\cH'}^{(\varphi')}\oplus \Z[G/H_{0}],\quad 
    J_{G/\cH}^{(\varphi)}\cong J_{G/\cH'}^{(\varphi')}\oplus \Z[G/H_{0}]. 
\end{equation*}
\end{enumerate}
\end{lem}

If $\varphi$ is normalized, then the rightmost homomorphism of \eqref{eq:ipdf} is surjective. Moreover, we have an exact sequence of $G$-lattices
\begin{equation*}
    0 \rightarrow \Z \xrightarrow{\left(\varphi(H)\cdot \varepsilon_{G/H}^{\circ}\right)_{H\in \cH^{\set}}}
    \bigoplus_{H\in \cH^{\set}}\Z[G/H]^{\oplus m_{\cH}(H)}\rightarrow 
    J_{G/\cH}^{(\varphi)}\rightarrow 0. 
\end{equation*}

\begin{rem}
\begin{enumerate}
\item If $\varphi(H)=\underbrace{(1,\ldots,1)}_{m_{\cH}(H)}$ for all $H\in \cH^{\set}$, then the $G$-lattices $I_{G/\cH}^{(\varphi)}$ and $J_{G/\cH}^{(\varphi)}$ coincide with $I_{G/\cH}$ and $J_{G/\cH}$ respectively. 
\item Assume that $G$ is an elementary $p$-abelian group, where $p$ is a prime number, and $\cH$ consists of subgroups of index $p$ and the whole group $G$. Let
\begin{equation*}
\varphi \colon \cH^{\set} \rightarrow \Delta;H\mapsto 
\begin{cases}
\underbrace{(1,\ldots,1)}_{m_{\cH}(H)}&\text{if }H\neq G;\\
\underbrace{(p,\ldots,p)}_{m_{\cH}(G)}&\text{if }H=G. 
\end{cases}
\end{equation*}
Then the $G$-lattice $I_{G/\cH}^{(\varphi)}$ is the same as $L=\Ker \Phi$ in \cite[\S 2, l9--14]{Endo2001}. 
\end{enumerate}
\end{rem}

\begin{lem}\label{lem:dvci}
Let $G$ be a finite group, $\cH$ a multiset of its subgroups, and $\varphi$ a weight function of $\cH$ satisfying $d(\varphi)\neq 0$. Then one has $I_{G/\cH}^{(\varphi^{\nor})}=I_{G/\cH}^{(\varphi)}$ and $J_{G/\cH}^{(\varphi^{\nor})}=J_{G/\cH}^{(\varphi)}$. 
\end{lem}

\begin{proof}
It suffices to prove $I_{G/\cH}^{(\varphi)}=I_{G/\cH}^{(\varphi^{\nor})}$. This follows from the fact that the multiplication by a non-zero integer on $\Z$ is injective. 
\end{proof}

\subsection{Reduction to smaller $G$-lattices}\label{ssec:rdsm}

We give two types of reduction to smaller $G$-lattices. To accomplish it, we first prepare some lemmas.  

\begin{lem}\label{lem:spke}
Let $A$ be a (not necessarily commutative) ring with unit. Consider a commutative diagram of left $A$-modules
\begin{equation*}
\xymatrix{
0\ar[r]& M_{0}\ar[r]\ar[d]^{h_{0}}& M_{1}\ar[r]\ar[d]^{h_{1}}& M_{2}\ar[r]&0\\
&N\ar@{=}[r]&N, &&
}
\end{equation*}
where the horizontal sequence is exact and the images of $h_{0}$ and $h_{1}$ coincide. Assume that there is a left splitting $f'\colon M_{1}\rightarrow M_{0}$ of the exact sequence satisfying $h_{0}\circ f'=h_{1}$. Then there is a split exact sequence
\begin{equation*}
0\rightarrow \Ker(h_{0})\rightarrow \Ker(h_{1})\rightarrow M_{2}\rightarrow 0. 
\end{equation*}
\end{lem}

\begin{proof}
    This follows from the snake lemma. 
\end{proof}

\begin{lem}\label{lem:ivfc}
Let $G$ be a finite group, and $H$ its subgroup. Consider a homomorphism of $G$-lattices
\begin{equation*}
f:=(c_{i}\varepsilon_{G/H})_{i}\colon \Z[G/H]^{\oplus m}\rightarrow \Z,
\end{equation*}
where $m\in \Zpn$ and $c_{1},\ldots,c_{m}$ are integers with the greatest common divisor $d$. Then there is an automorphism $\lambda$ of $G$-lattices $\Z[G/H]^{\oplus m}$ such that $f\circ \lambda$ coincides with the composite
\begin{equation*}
    \Z[G/H]^{\oplus m}\xrightarrow{\pr_{1}}\Z[G/H]\xrightarrow{d\varepsilon_{G/H}}\Z. 
\end{equation*}
\end{lem}

\begin{proof}
By Frobenius reciprocity, there is an isomorphism
   \begin{equation*}
       \Hom_{\Z[G]}(\Z[G/H]^{\oplus m},\Z)\cong \Hom_{\Z[H]}(\Z^{\oplus m},\Z)=\Hom_{\Z}(\Z^{\oplus m},\Z). 
   \end{equation*}
Hence, the assertion follows from the theory of invariant factors for finitely generated abelian groups. 
\end{proof}

The following is the first type of reduction to smaller $G$-lattices. 

\begin{lem}\label{lem:rded}
Let $G$ be a finite group, $\cH$ a multiset of its subgroups, and $\varphi$ a weight function of $\cH$. Assume that there exist $H_{0},H'_{0}\in \cH$, $i_{0}\in \{1,\ldots,m_{\cH}(H_{0})\}$ and $i'_{0}\in \{1,\ldots,m_{\cH}(H'_{0})\}$ such that
\begin{itemize}
\item $H_{0}\subset H'_{0}$; 
\item $i_{0}\neq i'_{0}$ if $H_{0}=H'_{0}$; and
\item ${\varphi}(H_{0})_{i_{0}}\in {\varphi}(H'_{0})_{i'_{0}}\Z$.
\end{itemize}
We define a multiset $\cH'$ of subgroups of $G$ as follows: 
\begin{equation*}
m_{\cH'}(H')=
\begin{cases}
    m_{\cH}(H_{0})-1&\text{if }H'=H_{0};\\
    m_{\cH}(H')&\text{if }H'\neq H_{0}. 
\end{cases}
\end{equation*}
Moreover, denote by $\varphi'$ the weight function of $\cH'$ as 
\begin{equation*}
\varphi'(H')=
\begin{cases}
    (\varphi(H_{0})_{1},\ldots,\varphi(H_{0})_{i_{0}-1},\varphi(H_{0})_{i_{0}+1},\ldots,\varphi(H_{0})_{m_{\cH}(H_{0})})&\text{if }H'=H_{0}\in (\cH')^{\set};\\
    \varphi(H')&\text{if }H'\neq H_0. 
\end{cases}
\end{equation*}
Then there exist isomorphisms of $G$-lattices
\begin{equation*}
I_{G/\cH}^{(\varphi)} \cong I_{G/\cH'}^{({\varphi'})} \oplus \Z[G/H_{0}],\quad
J_{G/\cH}^{(\varphi)} \cong J_{G/\cH'}^{({\varphi'})} \oplus \Z[G/H_{0}]. 
\end{equation*}
\end{lem}

\begin{rem}
If $\varphi(H)=(\underbrace{1,\ldots,1)}_{m_{\cH}(H)}$ for all $H\in \cH^{\set}$, then Lemma \ref{lem:rded} is essentially the same as \cite[Proposition 1.3]{Endo2011}. 
\end{rem}

\begin{proof}
If $\varphi(H_{0})_{i_{0}}=0$, then the assertion follows from Lemma \ref{lem:zero}. Hence, we may assume $\varphi(H_{0})_{i_{0}}>0$. It suffices to give an isomorphism
\begin{equation}\label{eq:mlt2}
    I_{G/\cH}^{(\varphi)}\cong 
    I_{G/\cH'}^{(\varphi')}\oplus \Z[G/H_{0}],
\end{equation}
which easily implies an isomorphism $J_{G/\cH}^{(\varphi)}\cong J_{G/\cH'}^{(\varphi')}\oplus \Z[G/H_{0}]$. Fix $H'_{0}\in \cH^{\set}$ containing $H_{0}$, $i_{0}\in \{1,\ldots,m_{\cH}(H_{0})\}$ and $i'_{0}\in \{1,\ldots,m_{\cH}(H'_{0})\}$ such that $\varphi(H_{0})_{i_{0}}\in \varphi(H'_{0})_{i'_{0}}\Z$. Then one has a commutative diagram
\begin{equation}\label{eq:gned}
\begin{gathered}
\xymatrix{
    0\ar[r]& \bigoplus_{H'\in (\cH')^{\set}} \Z[G/H']^{\oplus m_{\cH'}(H')} \ar[r] \ar[d]_{({\varphi'}(H')\cdot \varepsilon_{G/H'})_{H'}}& 
    \bigoplus_{H\in \cH^{\set}} \Z[G/H]^{\oplus m_{\cH}(H)} \ar[r] \ar[d]^{({\varphi}(H)\cdot \varepsilon_{G/H})_{H}}& 
    \Z[G/H_{0}] \ar[r] & 0\\
    & \Z \ar@{=}[r]& \Z, &&
    }
    \end{gathered}
\end{equation}
where the horizontal sequence is the canonical split exact sequence. Then the images of the vertical homomorphisms coincide. Now, we define a $G$-homomorphism $\Phi$ as the direct sum of the identity maps on $\Z[G/H]^{\oplus m_{\cH}(H)}$ for all $H\in \cH^{\set}\setminus \{H_{0},H'_{0}\}$ and $\Phi'$, where $\Phi'$ is given as follows: 
\begin{itemize}
\item if $H_{0}\neq H'_{0}$, then $\Phi'$ is the map
\begin{equation*}
    \Z[G/H_{0}]^{\oplus m_{\cH}(H_{0})}\oplus \Z[G/H'_{0}]^{\oplus m_{\cH}(H'_{0})}\rightarrow \Z[G/H_{0}]^{\oplus m_{\cH}(H_{0})-1}\oplus \Z[G/H'_{0}]^{\oplus m_{\cH}(H'_{0})}
\end{equation*}
defined by
\begin{equation*}
    ((x_{i})_{i},(y_{j})_{j})\mapsto \left((x_{i})_{i\neq i_{0}},\left(y_{j}+\delta_{j,i'_{0}}\frac{{\varphi}(H_{0})_{i_{0}}}{{\varphi}(H'_{0})_{i'_{0}}}\Ind_{H'_{0}}^{G}(\varepsilon_{H'_{0}/H_{0}})(x_{i_{0}})\right)_{j}\right);
\end{equation*}
\item if $H_{0}=H'_{0}$, then $\Phi'$ is the map
\begin{equation*}
    \Z[G/H_{0}]^{\oplus m_{\cH}(H_{0})}\rightarrow \Z[G/H_{0}]^{\oplus m_{\cH}(H_{0})-1}
\end{equation*}
defined by
\begin{equation*}
    (x_{i})_{i}\mapsto \left(x_{i}+\delta_{i,i'_{0}}\frac{{\varphi}(H_{0})_{i_{0}}}{{\varphi}(H'_{0})_{i'_{0}}}\Ind_{H'_{0}}^{G}(\varepsilon_{H'_{0}/H_{0}})(x_{i_{0}})\right)_{i\neq i_{0}}. 
\end{equation*}
\end{itemize}
Here $\delta_{i,j}$ is Kronecker's delta. Then the map $\Phi$ gives a left splitting of the exact sequence in \eqref{eq:gned}. Furthermore, by definition, the diagram
\begin{equation*}
    \xymatrix{
    \bigoplus_{H\in \cH^{\set}} \Z[G/H]^{\oplus m_{\cH}(H)} \ar[r]^{\Phi\hspace{12pt}} \ar[d]_{({\varphi}(H)\cdot \varepsilon_{G/H})_{H}}& 
    \bigoplus_{H'\in (\cH')^{\set}} \Z[G/H']^{\oplus m_{\cH'}(H')} \ar[d]^{({\varphi'}(H')\cdot \varepsilon_{G/H'})_{H'}}\\
    \Z \ar@{=}[r]& \Z
    }
\end{equation*}
is commutative. Therefore Lemma \ref{lem:spke} implies the existence of \eqref{eq:mlt2}. This completes the proof of Lemma \ref{lem:rded}. 
\end{proof}

\begin{dfn}\label{dfn:reds}
Let $G$ be a finite group, and $\cH$ a set of its subgroups (that is, all elements in $\cH$ have multiplicity $1$). We say that $\cH$ is \emph{reduced} if $H\not\subset H'$ for any $H,H'\in \cH$ with $H\neq H'$ as subgroups of $G$. 
\end{dfn}

For a multiset $\cH$ of subgroups of a finite group $G$, we denote by $\cH^{\red}$ the subset of $\cH^{\set}$ consisting of all elements that are maximal with respect to inclusion. Note that it is reduced in the sense of Definition \ref{dfn:reds}. 

\begin{cor}\label{cor:rdrd}
Let $G$ be a finite group, and $\cH$ a multiset of its subgroups. 
\begin{enumerate}
\item Let $\varphi$ be a normalized weight function on $\cH$, and $H'_{0}\in \cH^{\set}$. We define a multiset $\cH_{*}$ of subgroups of $G$ as
\begin{equation*}
    m_{\cH_{*}}(H')=
    \begin{cases}
        1&\text{if }H'=H'_{0};\\
        m_{\cH}(H')&\text{if }H'\neq H'_{0}. 
    \end{cases}
\end{equation*}
Moreover, denote by $\varphi_{*}$ the normalized weight function of $\cH_{*}$ defined as follows: 
\begin{equation*}
    \varphi_{*}(H')=
    \begin{cases}
        d_{\varphi}(H'_{0})&\text{if }H'=H'_{0};\\
        \varphi(H')&\text{if }H'\neq H'_{0}. 
    \end{cases}
\end{equation*}
Then there exist isomorphisms of $G$-lattices
\begin{align*}
I_{G/\cH}^{(\varphi)} \cong I_{G/\cH_{*}}^{(\varphi_{*})} \oplus \Z[G/H'_{0}]^{\oplus m_{\cH}(H'_{0})-1},\quad
J_{G/\cH}^{(\varphi)} \cong J_{G/\cH_{*}}^{(\varphi_{*})} \oplus \Z[G/H'_{0}]^{\oplus m_{\cH}(H'_{0})-1}. 
\end{align*}
In particular, the $G$-lattice $J_{G/\cH}^{(\varphi)}$ is quasi-permutation (resp.~quasi-invertible) if and only if $J_{G/\cH^{\set}}^{(d_{\varphi})}$ is so. 
\item There exist isomorphisms of $G$-lattices
\begin{align*}
I_{G/\cH} &\cong I_{G/\cH^{\red}} \oplus \left(\bigoplus_{H\in \cH^{\red}}\Z[G/H]^{\oplus m_{\cH}(H)-1}\right)\oplus \left(\bigoplus_{H\in \cH^{\set}\setminus \cH^{\red}}\Z[G/H]^{\oplus m_{\cH}(H)}\right);\\
J_{G/\cH} &\cong J_{G/\cH^{\red}}\oplus \left(\bigoplus_{H\in \cH^{\red}}\Z[G/H]^{\oplus m_{\cH}(H)-1}\right)\oplus \left(\bigoplus_{H\in \cH^{\set}\setminus \cH^{\red}}\Z[G/H]^{\oplus m_{\cH}(H)}\right). 
\end{align*}
In particular, the $G$-lattice $J_{G/\cH}$ is quasi-permutation (resp.~quasi-invertible) if and only if $J_{G/\cH^{\red}}$ is so. 
\end{enumerate}
\end{cor}

\begin{proof}
(i): It suffices to construct an isomorphism on $I_{G/\cH}^{(\varphi)}$. By Lemma \ref{lem:ivfc}, we have
\begin{equation*}
    I_{G/\cH}^{(\varphi)} \cong I_{G/\cH}^{(\widetilde{\varphi})},
\end{equation*}
where $\widetilde{\varphi}$ is defined as follows: 
\begin{equation*}
    \widetilde{\varphi}(H):=
    \begin{cases}
        (\underbrace{0,\ldots,0}_{m_{\cH}(H'_{0})-1},d_{\varphi}(H'_{0}))&\text{if }H=H'_{0};\\
        \varphi(H)&\text{if }H\neq H'_{0}. 
    \end{cases}
\end{equation*}
In this case, we have ${\widetilde{\varphi}}(H'_{0})_{i}\in {d}_{\varphi}(H'_{0})\Z$ for any $i\in \{1,\ldots,m_{\cH}(H'_{0})-1\}$. Then the assertion follows from Lemma \ref{lem:rded}. 

(ii): By (i), we may assume that $\cH$ is a set. Let $\varphi$ be the weight function on $\cH$ taking the value $1$. Then, we have $I_{G/\cH}^{(\varphi)}=I_{G/\cH}$. Moreover, for every $H\in \cH^{\set}\setminus \cH^{\red}$, there is $H'\in \cH^{\red}$ such that $H\subset H'$ and $\varphi(H)\in \varphi(H')\Z$. Hence, the assertion is a consequence of Lemma \ref{lem:rded}. 
\end{proof}

\begin{lem}\label{lem:cjpm}
Let $G$ be a finite group, and $H_{0}$ its subgroup. Then, for every $g\in G$, the map
\begin{equation*}
\alpha^{(g)}\colon \Z[G/H_{0}]\rightarrow \Z[G/gH_{0}g^{-1}];\,\sum_{g' \in G/H_{0}}a_{g'}g' \mapsto \sum_{g' \in G/H_{0}}a_{g'}g'g^{-1}
\end{equation*}
is an isomorphism of $G$-lattices. 
\end{lem}

\begin{lem}\label{lem:conj}
Let $G$ be a finite group, and $\cH$ a multiset of its subgroups. Take $H_{0}\in \cH$ and $g\in G\setminus N_{G}(H_{0})$. We define a multiset $\cH'$ of subgroups of $G$ as follows: 
\begin{equation*}
m_{\cH'}(H')=
\begin{cases}
m_{\cH}(gH_{0}g^{-1})+1&\text{if }H'=gH_{0}g^{-1}; \\
m_{\cH}(H_{0})-1&\text{if }H'=H_{0};\\
m_{\cH}(H')&\text{otherwise}. 
\end{cases}
\end{equation*}
Then there exist isomorphisms of $G$-lattices
\begin{equation*}
    I_{G/\cH}\cong I_{G/\cH'},\quad J_{G/\cH}\cong J_{G/\cH'}. 
\end{equation*}
\end{lem}

\begin{proof}
It suffices to prove the left isomorphism. The direct sum of the map 
\begin{equation*}
    \Z[G/H_{0}]^{\oplus m_{\cH}(H_{0})}\oplus \Z[G/gH_{0}g^{-1}]^{\oplus m_{\cH}(gH_{0}g^{-1})}\rightarrow 
    \Z[G/H_{0}]^{\oplus m_{\cH'}(H_{0})}\oplus \Z[G/gH_{0}g^{-1}]^{\oplus m_{\cH'}(gH_{0}g^{-1})}
\end{equation*}
defined by
\begin{equation*}
    ((x_{i})_{i},(y_{j})_{j})\mapsto ((x_{i})_{i\neq m_{\cH}(H_{0})},(y_{i})_{i},\alpha^{(g)}(x_{m_{\cH}(H_{0})}))
\end{equation*}
and the identity maps on $\Z[G/H']^{\oplus m_{\cH}(H')}$ for all $H'\in \cH^{\set}\setminus \{H_{0},gH_{0}g^{-1}\}$ induce a commutative diagram
\begin{equation*}
\xymatrix{
    0\ar[r]& I_{G/\cH} \ar[r] \ar[d]& 
    \bigoplus_{H\in \cH^{\set}} \Z[G/H]^{\oplus m_{\cH}(H)} \ar[d]^{\cong} \ar[rr]^{\hspace{65pt}(\varepsilon_{G/H})_{H}}&& \Z
     \ar[r]\ar@{=}[d] & 0\\
    0\ar[r]& I_{G/\cH'} \ar[r]& \bigoplus_{H'\in (\cH')^{\set}} \Z[G/H']^{\oplus m_{\cH'}(H')} \ar[rr]^{\hspace{65pt}(\varepsilon_{G/H'})_{H'}}&&\Z \ar[r]& 0. 
    }
\end{equation*}
Hence, the assertion follows from the snake lemma. 
\end{proof}

\begin{dfn}
Let $G$ be a finite group. We say that a set $\cH$ of subgroups of $G$ is \emph{strongly reduced} if $H \not\subset gH'g^{-1}$ for any $H,H'\in \cH$ with $H\neq H'$ as subgroups of $G$ and any $g\in G$. 
\end{dfn}

For a multiset $\cH$ of subgroups of a finite group $G$, let $\cH^{\srd}$ be a maximal strongly reduced subset of the set of all elements of $\cH^{\set}$ that are maximal with respect to inclusion up to $G$-conjugacy. 
Note that $\cH^{\srd}$ is not uniquely determined. 
However, we have $\cH^{\srd}\subset \cH^{\red}$ by definition.

\begin{prop}\label{prop:rdsr}
Let $G$ be a finite group, and $\cH$ a multiset of its subgroups. 
Then there exist isomorphisms of $G$-lattices
\begin{gather*}
    I_{G/\cH}\cong I_{G/\cH^{\srd}}\oplus \left( \bigoplus_{H\in \cH^{\srd}}\Z[G/H]^{\oplus m_{\cH}(H)-1}\right) \oplus \left( \bigoplus_{H\in \cH^{\set}\setminus \cH^{\srd}}\Z[G/H]^{\oplus m_{\cH}(H)}\right),\\ 
    J_{G/\cH}\cong J_{G/\cH^{\srd}}\oplus \left( \bigoplus_{H\in \cH^{\srd}}\Z[G/H]^{\oplus m_{\cH}(H)-1}\right) \oplus \left(\bigoplus_{H\in \cH^{\set}\setminus \cH^{\srd}}\Z[G/H]^{\oplus m_{\cH}(H)}\right). 
\end{gather*}
\end{prop}

\begin{proof}
We may assume that $\cH$ is reduced, which follows from Corollary \ref{cor:rdrd} (ii). 
It suffices to prove the isomorphism on $I_{G/\cH}$. 
Take $H_{0}\in \cH \setminus \cH^{\srd}$. Since $\cH$ is reduced, there exist $g\in G\setminus N_{G}(H_{0})$ and $H'_{0}\in \cH^{\srd}$ such that $gH_{0}g^{-1}\subset H'_{0}$. 
Let $\cH'$ be the multiset $\cH'$ of subgroups of $G$ as in Lemma \ref{lem:conj}. Then one has $(\cH')^{\red}=\cH \setminus \{H_{0}\}$. Moreover, Lemma \ref{lem:rded} gives an isomorphism of $G$-lattices
\begin{equation*}
I_{G/\cH}\cong I_{G/(\cH')^{\red}}\oplus \Z[G/gH_{0}g^{-1}]. 
\end{equation*}
Combining this with Lemma \ref{lem:cjpm}, we obtain an isomorphism
\begin{equation*}
I_{G/\cH}\cong I_{G/(\cH')^{\red}}\oplus \Z[G/H_{0}]. 
\end{equation*}
Repeating the above procedure for all $H_{0}\in \cH\setminus \cH^{\srd}$, we obtain the desired assertion. 
\end{proof}

The following is the second type of reduction to smaller $G$-lattices. 

\begin{prop}\label{prop:rdtr}
Let $G$ be a finite group, $\cH$ a multiset of its subgroups, and $\varphi$ a normalized weight function on $\cH$. Assume that there exist $H_{0},H'_{0}\in \cH$ and $i_{0}\in \{1,\ldots,m_{\cH}(H_{0})\}$ such that
\begin{itemize}
\item $H_{0}'\subsetneq H_{0}$; and
\item $\varphi(H_{0})_{i_{0}}\in (H_{0}:H'_{0})d_{\varphi}(H'_{0})\Z$. 
\end{itemize}
We define a multiset $\cH'$ of subgroups of $G$ as follows: 
\begin{equation*}
m_{\cH'}(H)=
\begin{cases}
    m_{\cH}(H_{0})-1&\text{if }H=H_{0};\\
    m_{\cH}(H)&\text{if }H\neq H_{0}.
\end{cases}
\end{equation*}
Furthermore, denote by $\varphi'$ the weight function of $\cH'$ defined as 
\begin{equation*}
    \varphi'(H'):=
    \begin{cases}
        (\varphi(H_{0})_{1},\ldots,\varphi(H_{0})_{i_{0}-1},\varphi(H_{0})_{i_{0}+1},\ldots,\varphi(H_{0})_{m_{\cH}(H_{0})})&\text{if }H'=H_{0}\in (\cH')^{\set};\\
        \varphi(H')&\text{if }H'\neq H_{0},
    \end{cases}
\end{equation*}
which is normalized.
Then there exist isomorphisms of $G$-lattices
\begin{equation}\label{eq:incd}
I_{G/\cH}^{(\varphi)} \cong I_{G/\cH'}^{(\varphi')}\oplus \Z[G/H_{0}],\quad 
J_{G/\cH}^{(\varphi)} \cong J_{G/\cH'}^{(\varphi')}\oplus \Z[G/H_{0}]. 
\end{equation}
In particular, the $G$-lattice $J_{G/\cH}^{(\varphi)}$ is quasi-permutation (resp.~quasi-invertible) if and only if $J_{G/\cH'}^{(\varphi')}$ is so. 
\end{prop}

\begin{rem}
Assume that
\begin{itemize}
\item $G\cong (C_{p})^{\nu}$ for some prime number $p$ and $\nu \in \Zpn$; 
\item $\cH$ consists of $G$ and some subgroups of index $p$ in $G$; and
\item $\varphi(H)=
\begin{cases}
    (\underbrace{1,\ldots,1}_{m_{\cH}(H)})&\text{if }H\neq G;\\
    (\underbrace{p,\ldots,p}_{m_{\cH}(H)})&\text{if }H=G. 
\end{cases}
$
\end{itemize}
Then Proposition \ref{prop:rdtr} implies \cite[p.~29, Lemma]{Endo2001}. 
\end{rem}

\begin{proof}
If $\varphi(H_{0})_{i_{0}}=0$, then the assertion follows from Lemma \ref{lem:zero}. Hence, we may assume $\varphi(H_{0})_{i_{0}}>0$. It suffices to prove the isomorphism of $I_{G/\cH}^{(\varphi)}$. Let $\cH_{*}$ and $\varphi_{*}$ be as in Corollary \ref{cor:rdrd} (i). Similarly, apply Corollary \ref{cor:rdrd} for $\cH'$, $\varphi'$ and $H'_{0}\in \cH^{\set}$, and denote by $\cH'_{*}$ and $\varphi'_{*}$ the resulting multiset and normalized weight function, respectively. 
Then we obtain isomorphisms of $G$-lattices
\begin{equation*}
    I_{G/\cH}^{(\varphi)}\cong I_{G/\cH_{*}}^{(\varphi_{*})}\oplus \Z[G/H'_{0}]^{\oplus m_{\cH}(H'_{0})-1},\quad
    I_{G/\cH'}^{(\varphi')}\cong I_{G/\cH'_{*}}^{(\varphi'_{*})}\oplus \Z[G/H'_{0}]^{\oplus m_{\cH'}(H'_{0})-1}. 
\end{equation*}
Hence, we may assume $m_{\cH}(H'_{0})=1$. Consider the commutative diagram
\begin{equation}\label{eq:cnsp}
\begin{gathered}
\xymatrix{
0\ar[r]& \bigoplus_{H'\in (\cH')^{\set}}\Z[G/H']^{\oplus m_{\cH'}(H')}\ar[r]\ar[d]_{(\varphi'(H')\cdot \varepsilon_{G/H'})_{H'}}& 
\bigoplus_{H\in \cH^{\set}}\Z[G/H]^{\oplus m_{\cH}(H)} \ar[r] \ar[d]^{(\varphi(H)\cdot \varepsilon_{G/H})_{H}}&
\Z[G/H_{0}]\ar[r]& 0\\
&\Z \ar@{=}[r]& \Z, &&
}
\end{gathered}
\end{equation}
where the horizontal sequence is the canonical split exact sequence. By the definitions of $\cH'$ and $\varphi'$, the images of the vertical maps coincide. Now, we define the map $\Psi$ as the direct sum of the identity map on $\Z[G/H]^{\oplus m_{\cH'}(H)}$ for all $H\in (\cH')^{\set}\setminus \{H_{0},H'_{0}\}$, and the map
\begin{equation*}
\Z[G/H_{0}]^{\oplus m_{\cH}(H_{0})}\oplus \Z[G/H'_{0}]\rightarrow \Z[G/H_{0}]^{\oplus m_{\cH'}(H_{0})}\oplus \Z[G/H'_{0}]
\end{equation*}
defined as
\begin{equation*}
    ((x_{i})_{i},y)\mapsto \left((x_{i})_{i\neq i_{0}},\frac{\varphi(H_{0})_{i_{0}}}{(H_{0}:H'_{0})d_{\varphi}(H'_{0})}\Ind_{H_{0}}^{G}\varepsilon_{H_{0}/H'_{0}}^{\circ}(x_{i_{0}})+y\right). 
\end{equation*}
Then it gives a left splitting of the exact sequence in \eqref{eq:cnsp}. Moreover, the diagram
\begin{equation*}
\xymatrix{
\bigoplus_{H\in \cH^{\set}}\Z[G/H]^{\oplus m_{\cH}(H)}\ar[r]^{\Psi\hspace{10pt}}\ar[d]_{(\varphi(H)\cdot \varepsilon_{G/H})_{H}}&\bigoplus_{H'\in (\cH')^{\set}}\Z[G/H']^{\oplus m_{\cH'}(H')}\ar[d]^{(\varphi'(H')\cdot \varepsilon_{G/H'})_{H'}}\\
\Z \ar@{=}[r]& \Z
}
\end{equation*}
is commutative. Hence, we obtain the left isomorphism in \eqref{eq:incd} as desired. 
\end{proof}

\begin{cor}\label{cor:rdmn}
Let $G$ be a finite group, and $\cH$ a multiset of its subgroups. Consider a normalized weight function $\varphi$ on $\cH$, and define a set of subgroups of $G$ as follows: 
\begin{equation*}
    \cH_{\varphi}[1]:=\{H\in \cH^{\set} \mid d_{\varphi}(H)=1\}. 
\end{equation*}
Assume that 
\begin{itemize}
    \item[(a)] $\cH_{\varphi}[1]\neq \emptyset$; and 
    \item[(b)] for each $H\in \cH^{\set}\setminus \cH_{\varphi}[1]$, there is $H'\in \cH_{\varphi}[1]$ such that $d_{\varphi}(H)\in (H:H\cap H')\Z$. 
\end{itemize}
Then the following hold: 
\begin{equation*}
[I_{G/\cH}^{(\varphi)}]=[I_{G/\cH_{\varphi}[1]}],\quad 
[J_{G/\cH}^{(\varphi)}]=[J_{G/\cH_{\varphi}[1]}]. 
\end{equation*}
In particular, the $G$-lattice $J_{G/\cH}^{(\varphi)}$ is quasi-permutation (resp.~quasi-invertible) if and only if $J_{G/\cH_{\varphi}[1]}$ is so. 
\end{cor}

\begin{proof}
It suffices to prove the equality of $[I_{G/\cH}^{(\varphi)}]$. By Corollary \ref{cor:rdrd} (i), we may assume $\cH=\cH^{\set}$. In particular, we have $\varphi=d_{\varphi}$. Moreover, using Lemma \ref{lem:zero}, it suffices to consider the case that $\varphi(H)\neq 0$ holds for all $H\in \cH$. Write $\cH \setminus \cH_{\varphi}[1]=\{H_{1},\ldots,H_{s}\}$. For each $i\in \{1,\ldots,s\}$, take $H_{i}^{\dagger}\in \cH_{\varphi}[1]$ so that 
\begin{equation}\label{eq:hdir}
    d_{\varphi}(H_{i})\in (H_{i}:H_{i}\cap H_{i}^{\dagger})\Z. 
\end{equation}
Note that this is possible by (b). Consider a set $\cH^{\dagger}:=\{H_{i}\cap H_{i}^{\dagger}\mid i\in \{1,\ldots,s\}\}$, and we define a multiset $\widetilde{\cH}$ of subgroups of $G$ as the disjoint union of $\cH$ and $\cH^{\dagger}$. Moreover, let $\widetilde{\varphi}$ be the normalized weight function on $\widetilde{\cH}$ defined as
\begin{equation*}
    \widetilde{\varphi}(H'):=
    \begin{cases}
        (1,\varphi(H'))&\text{if }H'\in \cH\cap \cH^{\dagger};\\
        1&\text{if $H'\in \cH^{\dagger}\setminus \cH$;}\\
        \varphi(H')&\text{if }H'\in \cH\setminus \cH^{\dagger}. 
    \end{cases}
\end{equation*}
Since $\widetilde{\varphi}(H')_{1}=1$ for any $H'\in \cH^{\dagger}$, Lemma \ref{lem:rded} gives an equality
\begin{equation*}
[I_{G/\widetilde{\cH}}^{(\widetilde{\varphi})}]=[I_{G/\cH}^{(\varphi)}]. 
\end{equation*}
Moreover, since $\widetilde{\cH}^{\set}=\cH\cup \cH^{\dagger}$, Corollary \ref{cor:rdrd} (i) implies
\begin{equation*}
[I_{G/\widetilde{\cH}}^{(\widetilde{\varphi})}]=[I_{G/(\cH\cup \cH^{\dagger})}^{(d_{\widetilde{\varphi}})}].  
\end{equation*}
On the other hand, take $H\in (\cH \cup \cH^{\dagger})\setminus \cH'$. Then we have $H\in \cH \setminus \cH_{\varphi}[1]$, and hence $H=H_{i}$ for some $i$. Moreover, one has
\begin{equation*}
d_{\widetilde{\varphi}}(H_{i})=\varphi(H_{i})\in (H_{i}:H_{i}\cap H_{i}^{\dagger})\Z=(H_{i}:H_{i}\cap H_{i}^{\dagger})d_{\widetilde{\varphi}}(H_{i}\cap H_{i}^{\dagger})\Z
\end{equation*}
by \eqref{eq:hdir}. Therefore, we can apply Proposition \ref{prop:rdtr} to $\cH\cup \cH^{\dagger}$, $d_{\widetilde{\varphi}}$ and the inclusion $H_{i}\cap H_{i}^{\dagger}\subset H_{i}=H$. Consequently, we obtain an equality
\begin{equation*}
    [I_{G/(\cH \cup \cH^{\dagger})}^{(d_{\widetilde{\varphi}})}]=[I_{G/\cH'}^{(\varphi')}]. 
\end{equation*}
Here, $\cH':=\cH_{\varphi}[1]\cup \cH^{\dagger}$, and $\varphi'$ is the restriction to $\cH'$ of $d_{\widetilde{\varphi}}$. Then we have $\varphi'(H')=1$ for any $H'\in \cH'$, and hence we obtain an equality
\begin{equation*}
    I_{G/\cH'}^{(\varphi')}=I_{G/\cH'}. 
\end{equation*}
Now, recall that any element $H'$ of $\cH^{\dagger}\setminus \cH_{\varphi}[1]$ satisfies $H'=H_{i}\cap H_{i}^{\dagger}$ for some $i$. Since $H_{i}^{\dagger}$ is an element of $\cH_{\varphi}[1]$, we can apply Lemma \ref{lem:rded} to $\cH'$, $\varphi'$ and the inclusion $H'\subset H_{i}^{\dagger}$. Repeating this argument for all $H'\in \cH^{\dagger}\setminus \cH_{\varphi}[1]$, we get
\begin{equation*}
    [I_{G/\cH'}]=[I_{G/\cH_{\varphi}[1]}]. 
\end{equation*}
Consequently, we obtain the desired assertion. 
\end{proof}

\subsection{Reduction to lattices over smaller groups}

We first describe $I_{G/\cH}^{(\varphi)}$ and $J_{G/\cH}^{(\varphi)}$ as $P$-lattices, where $P$ is a subgroup of $G$. 

\begin{prop}\label{prop:rtmc}
Let $G$ be a finite group, and $P$ a subgroup of $G$. Consider a multiset $\cH$ of subgroups of $G$ and a normalized weight function $\varphi$ on $\cH$. Then there are isomorphisms of $P$-modules
\begin{equation*}
    I_{G/\cH}^{(\varphi)}\cong I_{P/\cH_{P}}^{(\varphi_{P})},\quad J_{G/\cH}^{(\varphi)}\cong J_{P/\cH_{P}}^{(\varphi_{P})}. 
\end{equation*}
Here, $\varphi_{P}$ is defined as follows:
\begin{itemize}
    \item $C_{H}$ is a complete set of representatives of $P\backslash G/H$ in $G$ for each $H \in \cH$; 
    \item $\cH_{P}$ is the multiset of subgroups of $P$ defined as
    \begin{equation*}
        \cH_{P}^{\set}:=\{P\cap gHg^{-1}<P\mid H\in \cH^{\set},g\in C_{H}\}
    \end{equation*}
    and 
    \begin{equation*}
        m_{\cH_{P}}(H'):=\sum_{H\in \cH^{\set},g\in C_{H},P\cap gHg^{-1}=H'}m_{\cH}(H)
    \end{equation*}
    for each $H'\in \cH_{P}^{\set}$; 
    \item $\varphi_{P}$ is the normalized weight function of $\cH_{P}$ which sends $H'\in \cH_{P}^{\set}$ to the unique element of $\Delta_{m_{\cH_{P}}(H')}$ determined by $\varphi(H)$ for all $H\in \cH$ and $g\in C_{H}$ with $H'=P\cap gHg^{-1}$. 
\end{itemize}
In particular, if $J_{G/\cH}^{(\varphi)}$ is quasi-permutation (resp.~quasi-invertible), then $J_{P/\cH_{P}}^{(\varphi_{P})}$ is so. 
\end{prop}

\begin{proof}
This is a consequence of Mackey's decomposition. See \cite[Section 7.3, Proposition 22]{Serre1977}. 
\end{proof}

Next, we describe $(I_{G/\cH}^{(\varphi)})^{N}$ and $(J_{G/\cH}^{(\varphi)})^{[N]}$ for a normal subgroup $N$ of $G$. 

\begin{lem}\label{lem:augd}
Let $G$ be a finite group, and $H$ its subgroup. 
\begin{enumerate}
\item For any subgroup $H'$ of $G$ containing $H$, the diagram
\begin{equation*}
\xymatrix@C=57pt{
\Z[G/H']\ar[r]\ar[d]_{\varepsilon_{G/H'}}&
\Z[G/H]\ar[d]^{\varepsilon_{G/H}}\ar[r]^{\Ind_{H'}^{G}(\varepsilon_{H'/H})}& \Z[G/H']\ar[d]^{\varepsilon_{G/H'}}\\
\Z \ar[r]^{(H':H)}&\Z\ar@{=}[r]& \Z
}
\end{equation*}
is commutative. 
\item Let $N$ be a normal subgroup of $G$. Then the image of the canonical injection
\begin{equation*}
\Z[G/HN]\hookrightarrow \Z[G/H]
\end{equation*}
coincides with $\Z[G/H]^N$. 
\end{enumerate}
\end{lem}

\begin{proof}
This follows from the definitions of the augmentation maps and their induced maps.
\end{proof}

\begin{prop}\label{prop:tkfx}
Let $G$ be a finite group, and $N$ its normal subgroup. Take a multiset $\cH$ of subgroups of $G$ and a normalized weight function $\varphi$ of $\cH$. Then there are isomorphisms
\begin{equation*}
(I_{G/\cH}^{(\varphi)})^{N}\cong I_{(G/N)/\cH^{N}}^{(\overline{\varphi}_{G/N}^{\nor})}, \quad 
(J_{G/\cH}^{(\varphi)})^{[N]}\cong J_{(G/N)/\cH^{N}}^{(\overline{\varphi}_{G/N}^{\nor})}. 
\end{equation*}
Here $\cH^{N}$ and $\overline{\varphi}_{G/N}^{\nor}$ are defined as follows: 
\begin{itemize}
\item $\cH^{N}$ is the multiset of subgroups of $G/N$ defined by
\begin{equation*}
    (\cH^{N})^{\set}=\{HN/N<G/N\mid H\in \cH^{\set}\}
\end{equation*}
and
\begin{equation*}
    m_{\cH^{N}}(\overline{H})=\sum_{H\in \cH^{\set},HN/N=\overline{H}}m_{\cH}(H)
\end{equation*}
for each $\overline{H}\in (\cH^{N})^{\set}$; 
\item the weight function $\overline{\varphi}_{G/N}$ on $\cH^{N}$ that sends $\overline{H}\in (\cH^{N})^{\set}$ to the unique element of $\Delta_{m_{\cH^{N}}}(\overline{H})$ determined by $(HN:H)\varphi(H)$ for all $H\in \cH$ with $HN/N=\overline{H}$; and
\item $\overline{\varphi}_{G/N}^{\nor}$ is as in Lemma \ref{lem:pnor}. 
\end{itemize}
In particular, if $J_{G/\cH}^{(\varphi)}$ is quasi-permutation (resp.~quasi-invertible), then $J_{(G/N)/\cH^{N}}^{(\overline{\varphi}_{G/N}^{\nor})}$ is so.  
\end{prop}

\begin{proof}
By Lemma \ref{lem:augd}, we obtain an isomorphism of $G/N$-lattices $(I_{G/\cH}^{(\varphi)})^{N}\cong I_{(G/N)/\cH^{N}}^{(\overline{\varphi}_{G/N})}$. On the other hand, we have $I_{(G/N)/\cH^{N}}^{(\overline{\varphi}_{G/N})}=I_{(G/N)/\cH^{N}}^{(\overline{\varphi}_{G/N}^{\nor})}$ by Lemma \ref{lem:dvci}. Hence we obtain $(I_{G/\cH}^{(\varphi)})^{N}\cong I_{(G/N)/\cH^{N}}^{(\overline{\varphi}_{G/N}^{\nor})}$ as desired. The isomorphism on $(J_{G/\cH}^{(\varphi)})^{[N]}$ is a consequence of that on $(I_{G/\cH}^{(\varphi)})^{N}$. 
\end{proof}

For a multiset $\cH$ of subgroups of $G$, we denote by $N^{G}(\cH)$ the maximum normal subgroup of $G$ which is contained in $H$ for all $H\in \cH$. Moreover, we simply denote $N^{G}(\cH)$ by $N^{G}(H)$ if $\cH$ consists of a single subgroup $H$. 

\begin{cor}\label{cor:trfx}
Let $G$ be a finite group, and $\cH$ a multiset of subgroups of $G$. Consider a normal subgroup $N$ of $G$ that is contained in $N^{G}(\cH)$. Then there exist isomorphisms of $G$-lattices
\begin{equation*}
    I_{G/\cH}\cong I_{(G/N)/\cH^{N}},\quad J_{G/\cH}\cong J_{(G/N)/\cH^{N}}. 
\end{equation*}
Here we regard $I_{(G/N)/\cH^{N}}$ and $J_{(G/N)/\cH^{N}}$ as $G$-lattices by the natural surjection $G\twoheadrightarrow G/N$. In particular, the $G$-lattice $J_{G/\cH}$ is quasi-permutation (resp.~quasi-invertible) if and only if $J_{(G/N)/\cH^{N}}$ is so as a $G/N$-lattice. 
\end{cor}

\begin{proof}
This follows from Proposition \ref{prop:tkfx} since the actions of $N$ on $I_{G/\cH}$ and $J_{G/\cH}$ are trivial. 
\end{proof}

\section{$p$-groups}\label{sect:pgrp}

For a finite group $G$, we write $\Phi(G)$ for the Frattini subgroup of $G$, that is, the intersection of all maximal subgroups of $G$. Here maximal subgroups mean \emph{proper} subgroups which are maximal with respect to inclusion. 

\begin{prop}[{\cite[Theorem 4.3.2]{Hall1959}}]\label{prop:mxps}
Let $G$ be a $p$-group, where $p$ is a prime number. Then all maximal subgroups of $G$ are normal of index $p$. In particular, the Frattini subgroup $\Phi(G)$ contains the derived subgroup of $G$, and $G/\Phi(G)$ is an elementary $p$-abelian group. 
\end{prop}

\begin{cor}\label{cor:pgfg}
Let $p$ be a prime number, $G$ a $p$-group, and $H$ its proper subgroup. 
\begin{enumerate}
\item The subgroups $\Phi(G)$ and $H$ do not generate $G$. 
\item We further assume $(N_{G}(H):H)=p$. If a subgroup $P$ of $G$ contains $H$ properly, then we have $N_{G}(H)\subset P$. 
\end{enumerate}
\end{cor}

\begin{proof}
(i): By Proposition \ref{prop:mxps}, we can take a maximal subgroup $P$ of $G$ containing $H$. Then it contains $\Phi(G)$ by definition, and hence $\Phi(G)H\subset P\subsetneq G$ as desired. 

(ii): 
Using Proposition \ref{prop:mxps} repeatedly, we obtain that there exists a subgroup $P'$ of $P$ such that $H\subset P'$ and $(P':H)=p$. Moreover, $H$ is normal in $P'$. On the other hand, the assumption $(N_{G}(H):H)=p$ implies that $N_{G}(H)$ is the unique subgroup of $G$ of order $p\cdot \#H$ that contains $H$. Hence, we obtain $N_{G}(H)=P'\subset P$. 
\end{proof}

We use a basic result on $p$-groups as follows:

\begin{lem}\label{lem:nznt}
Let $p$ be a prime number, $G$ a $p$-group, and $N\neq \{1\}$ a normal subgroup of $G$. Then we have $Z(G)\cap N\neq\{1\}$.
\end{lem}

In what follows, for a positive integer $n$, we denote by $D_{n}$ the dihedral group of order $2n$, that is,
\begin{equation*}
    D_{n}=\langle \sigma_{n},\tau_{n}\mid \sigma_{n}^{n}=\tau_{n}^{2}=1,\tau_{n}\sigma_{n}\tau_{n}=\sigma_{n}^{-1} \rangle. 
\end{equation*}

Let $p$ be a prime number. Recall that a finite group $G$ of order $p^{n}$ is said to be \emph{of maximal class} if its nilpotency class is $n-1$. 

\begin{prop}[{\cite[Corollary 1.7]{Berkovich2008}}]\label{prop:clmc}
Let $G$ be a non-abelian $2$-group of maximal class. Then $G$ is isomorphic to one of the following: 
\begin{enumerate}
\item the dihedral group $D_{2^{\nu}}$ of order $2^{\nu+1}$ for $\nu \geq 2$; 
\item the semi-dihedral group $SD_{2^{\nu+1}}$ of order $2^{\nu+1}$ for $\nu \geq 3$; 
\item the generalized quaternion group $Q_{2^{\nu}}$ of order $2^{\nu}$ for $\nu \geq 3$, that is,
\begin{equation*}
    Q_{2^{\nu}}=\langle i_{\nu},j_{\nu}\mid i_{\nu}^{2^{\nu-1}}=1,j_{\nu}^{2}=i_{\nu}^{2^{\nu-2}},j_{\nu}i_{\nu}j_{\nu}^{-1}=i_{\nu}^{-1}\rangle. 
\end{equation*}
\end{enumerate}
\end{prop}

\begin{prop}\label{prop:nmne}
Let $G$ be a non-cyclic $2$-group of order a multiple of $8$ that is not of maximal class. 
\begin{enumerate}
\item There exists an abelian normal subgroup $E$ of $G$ of order $8$. 
\item Under the notation in \emph{(i)}, we further assume $E\cong C_{4}\times C_{2}$. Then $\Phi(E)$ is contained in the center of $G$.
\end{enumerate}
\end{prop}

\begin{proof}
This is explained in \cite[p.~91, Proof of Step 4]{Endo2011}. However, we give a proof for the reader's convenience. 

(i): By \cite[Lemma 1.4]{Berkovich2008}, one can take a non-cyclic normal subgroup $E'$ of $G$ of order $4$ since $G$ is not of maximal class. Furthermore, \cite[Proposition 1.8]{Berkovich2008} implies that the centralizer $Z_{G}(E')$ of $E'$ in $G$ strictly contains $E'$. Hence, there is $g\in Z_{G}(E')$ whose image in $G/E'$ is central of order $2$. Now, let $E$ be the subgroup of $G$ generated by $E'$ and $g$. Then $E$ is abelian, normal in $G$ and has order $8$. 

(ii): This follows from the fact that $\Phi(E)$ is a characteristic subgroup of $E$. 
\end{proof}

The following two lemmas can be obtained by direct computation.

\begin{lem}\label{lem:nodt}
Let $\nu \geq 2$ be an integer. 
\begin{enumerate}
\item For any $m\in \Z$, we have $N_{D_{2^{\nu}}}(\langle \sigma_{2^{\nu}}^{m}\tau_{2^{\nu}}\rangle)=\langle \sigma_{2^{\nu}}^{2^{\nu-1}}, \sigma_{2^{\nu}}^{m}\tau_{2^{\nu}}\rangle$. 
\item For two integers $m$ and $m'$, $\langle \sigma_{2^{\nu}}^{m}\tau_{2^{\nu}}\rangle$ and $\langle \sigma_{2^{\nu}}^{m'}\tau_{2^{\nu}}\rangle$ are conjugate in $D_{2^{\nu}}$ if and only if $m-m'$ is even. 
\item Every non-normal subgroup of order $2$ in $D_{2^{\nu}}$ is of the form $\langle \
\sigma_{2^{\nu}}^{m}\tau_{2^{\nu}}\rangle$ for some integer $m$. 
\end{enumerate}
\end{lem}

\begin{lem}\label{lem:nomx}
Let $\nu \geq 3$ be an integer. 
\begin{enumerate}
\item All non-normal subgroups of order $2$ in $SD_{2^{\nu+1}}$ are conjugate to each other. 
\item There exists a unique subgroup of order $2$ in $Q_{2^{\nu}}$, which is the center of $Q_{2^{\nu}}$. 
\end{enumerate}
\end{lem}

\section{Proof of Theorem \ref{mth1}}\label{sect:pfod}

First, we specify previous results given by Endo--Miyata (\cite{Endo1975}) and Endo (\cite{Endo2001}, \cite{Endo2011}). 

\begin{prop}[{\cite[Theorem 2.3]{Endo1975}}]\label{prop:emqp}
For a finite group $G$, the following are equivalent: 
\begin{enumerate}
    \item the $G$-lattice $J_{G}$ is quasi-permutation; 
    \item there is an isomorphism $G\cong C_{m}\rtimes C_{2^{\nu}}$, where $m$ is an odd positive integer, $\nu \in \Znn$, and the action of $C_{2^{\nu}}$ on $C_{m}$ factors through a homomorphism $C_{2^{\nu}}\rightarrow C_{2}$. 
\end{enumerate}
\end{prop}

\begin{prop}\label{prop:emn1}
Let $p$ be a prime number, $G$ a $p$-group, and $H$ a subgroup of $G$. Then the following are equivalent: 
\begin{enumerate}
\item $J_{G/H}$ is a quasi-permutation $G$-lattice; 
\item $J_{G/H}$ is a quasi-invertible $G$-lattice; 
\item $G/N^{G}(H)$ is cyclic. 
\end{enumerate}
\end{prop}
\begin{proof}
\textbf{Case 1.~$H$ is normal in $G$. }
It suffices to prove (ii) $\Rightarrow$ (iii) $\Rightarrow$ (i). We may assume $H=\{1\}$. We first assume that $J_{G/H}$ is quasi-invertible. Then, \cite[Theorem 1.5]{Endo1975} implies the cyclicity of all the Sylow subgroups of $G$. This implies that $G$ is cyclic, because it is a $p$-group. Hence, (iii) is valid. On the other hand, if $G$ is cyclic, then (i) follows from Proposition \ref{prop:emqp}. 

\textbf{Case 2.~$H$ is not normal in $G$. }
We may assume $N^{G}(H)=\{1\}$. It suffices to prove that $J_{G/H}$ is not quasi-invertible. However, it is the same as \cite[Theorem 2.1]{Endo2011}. 
\end{proof}

\begin{prop}[{\cite[Theorem 1]{Endo2001}}]\label{prop:endo}
Let $p$ be a prime number, and $G$ an elementary $p$-abelian group. Take a reduced set $\cH$ of subgroups of $G$. If $(G:H)=p$ for all $H\in \cH$, then the following are equivalent: 
\begin{enumerate}
    \item $J_{G/\cH}$ is a quasi-permutation $G$-lattice; 
    \item $J_{G/\cH}$ is a quasi-invertible $G$-lattice; 
    \item $\#\cH=1$ or $p=\#\cH=2$. 
\end{enumerate}
\end{prop}

\begin{dfn}\label{dfn:mudf}
For a multiset $\cH$ of subgroups of a finite group $G$, put
\begin{equation*}
\mu(\cH):=\min\{(G:H)\in \Zpn \mid H\in \cH\},\quad M(\cH):=\max\{(G:H)\in \Zpn \mid H\in \cH\}. 
\end{equation*}
\end{dfn}

\begin{lem}\label{lem:quot}
Let $G$ be a finite group, and $\cH$ a reduced set of its subgroups. If $\#\cH \geq 2$, then we have $(G:N^{G}(\cH))>M(\cH)$. 
\end{lem}

\begin{proof}
Take $H_{0}\in \cH$ with $(G:H_{0})=M(\cH)$. If $(G:N^{G}(H_{0}))>M(\cH)$, then the assertion is clear. Otherwise, $H_{0}$ is normal in $G$. Take $H\in \cH \setminus \{H_{0}\}$. Then $N^{G}(H)$ does not contain $H_{0}$ since $\cH$ is reduced. Hence $(G:N^{G}(\{H,H_{0}\}))>M(\cH)$. This implies the desired assertion since $N^{G}(\cH)$ is contained in $N^{G}(\{H,H_{0}\})$. 
\end{proof}

\begin{dfn}\label{dfn:hnor}
For a multiset $\cH$ of subgroups of a finite group $G$, set
\begin{equation*}
    \cH^{\nor}:=\{H\in \cH \mid H \triangleleft G\}. 
\end{equation*}
\end{dfn}

\begin{lem}\label{lem:sglh}
Let $p$ be a prime number, $G$ a $p$-group, and $\cH$ a reduced set of its subgroups. If $G/N^{G}(\cH)$ is cyclic, then $\#\cH=1$ and $\cH^{\nor}=\cH$. 
\end{lem}

\begin{proof}
Take $H,H'\in \cH$. By definition, $H$ and $H'$ contain $N^{G}(\cH)$. Since $G/N^{G}(\cH)$ is cyclic, we have $H\subset H'$ or $H'\subset H$. This implies $H=H'$ since $\cH$ is reduced. Hence, we obtain $\#\cH=1$. The equality $\cH^{\nor}=\cH$ follows from the fact that all subgroups of $G$ containing its derived subgroup are normal in $G$. 
\end{proof}

\begin{dfn}\label{dfn:hctp}
Let $G$ be a finite group, and $P$ its subgroup. For a multiset $\cH$ of subgroups of $G$, set
\begin{equation*}
\cH_{\subset P}:=\{H\in \cH\mid H\subset P\}. 
\end{equation*}
\end{dfn}

\begin{lem}\label{lem:nmlz}
Let $p$ be a prime number, $\nu \geq 3$ a positive integer, and $G$ a finite group of order $p^{\nu}$.
Consider a reduced set $\cH$ of subgroups of $G$ satisfying $\#\cH\geq 2$ and $\mu(\cH)\geq p^{2}$. Take a maximal subgroup $P$ of $G$. Assume that there is $H_{1}\in \cH$ with $(G:H_{1})=\mu(\cH)$ such that all elements of $\cH_{P}^{\red}$ are $G$-conjugate to $H_{1}$. 
If $H_{0}\in \cH^{\nor} \setminus \{H_{1}\}$ has index $\mu(\cH)$ in $G$, then it is contained in $N_{G}(H_{1})$. Moreover, we can take a maximal subgroup $P'$ of $G$ so that $\cH_{\subset P'}$ contains $H_{0}$ and $H_{1}$ (in particular, $\cH_{P'}^{\red}$ contains $H_{0}$). 
\end{lem}

\begin{proof}
We first prove $H_{0}\subset N_{G}(H_{1})$. By assumption, we have $P\cap H_{0}\subset gH_{1}g^{-1}$ for some $g\in G$. In particular, $H_{0}\not\subset P$. Moreover, $P\cap H_{0}\subset H_{1}$ holds since $P$ and $H_{0}$ are normal in $G$. Hence, one has
\begin{equation*}
H_{0}\cap H_{1}=P\cap H_{0}\cap H_{1}=P\cap H_{0}. 
\end{equation*}
Combining this equality with the normality of $H_{0}$ in $G$, we obtain
\begin{equation*}
(H_{0}H_{1}:H_{1})=(H_{0}H_{1}:H_{0})=(H_{1}:H_{0}\cap H_{1})=(H_{1}:P\cap H_{0})=p. 
\end{equation*}
On the other hand, $H_{0}H_{1}$ is a $p$-group since $G$ is so. Consequently, we have 
\begin{equation*}
H_{0}\subset H_{0}H_{1}\subset N_{G}(H_{1})
\end{equation*}
by Proposition \ref{prop:mxps}. 

Secondly, we construct a maximal subgroup $P'$ of $G$ that satisfies $\{H_{0},H_{1}\}\subset \cH_{\subset P'}$. Since $\mu(\cH)\geq p^{2}$, one has
\begin{equation*}
    (G:H_{0}H_{1})=p^{-1}(G:H_{1})=p^{-1}\mu(\cH)\geq p. 
\end{equation*}
Hence, we may take $P'$ so that $H_{0}H_{1}$ is contained. 
\end{proof}

\begin{thm}\label{thm:ndqp}
Let $p$ be an odd prime number, and $G$ a $p$-group. Consider a reduced set $\cH$ of subgroups of $G$. Then the following are equivalent: 
\begin{enumerate}
    \item $J_{G/\cH}$ is a quasi-permutation $G$-lattice; 
    \item $J_{G/\cH}$ is a quasi-invertible $G$-lattice; 
    \item $\#\cH=1$ and $G/N^{G}(\cH)$ is cyclic. 
\end{enumerate}
\end{thm}

\begin{proof}
(i) $\Rightarrow$ (ii) is clear. (iii) $\Rightarrow$ (i) follows from Proposition \ref{prop:emn1}. In the following, we prove (ii) $\Rightarrow$ (iii), which is achieved by giving a proof of the contraposition. We may assume that $\cH$ contains at least two elements. In particular, $G$ is not cyclic according to Lemma \ref{lem:sglh}. It suffices to prove that $J_{G/H}$ is not quasi-invertible if $\#\cH \geq 2$. Write $\#G=p^{\nu}$. We give a proof of the above assertion by induction on $\nu$. If $\nu=2$, the assertion follows from Proposition \ref{prop:endo}. Now suppose $\nu \geq 3$, and the assertion holds for all groups of order $p^{\nu-1}$. If $N^{G}(\cH)\neq \{1\}$, take a subgroup $N$ of order $p$ in $Z(G)\cap N^{G}(\cH)$. Note that $Z(G)\cap N^{G}(\cH)$ is non-trivial by Lemma \ref{lem:nznt}. Then it suffices to prove the assertion for the $G/N$-lattice $J_{(G/N)/\cH^{N}}$. Since $\#(G/N)=p^{\nu-1}$, the assertion follows from the induction hypothesis. Hence, we may further assume $N^{G}(\cH)=\{1\}$. If $M(\cH)=p$, then $G$ is elementary $p$-abelian since $N^{G}(\cH)=\{1\}$. Therefore, the $G$-lattice $J_{G/\cH}$ is not quasi-invertible by Proposition \ref{prop:endo}. Therefore, we can impose $M(\cH)\geq p^{2}$ in the sequel. By the induction hypothesis, it suffices to prove that there is a maximal subgroup $P$ of $G$ that satisfies $\#\cH_{P}^{\red}\geq 2$. 

\textbf{Case 1.~$\cH=\cH^{\nor}$. }

Recall the notations in Definition \ref{dfn:mudf}, that is,
\begin{equation*}
M(\cH):=\max\{(G:H)\in \Zpn \mid H\in \cH\},\quad \mu(\cH):=\min\{(G:H)\in \Zpn \mid H\in \cH\}. 
\end{equation*}

\textbf{Case 1-a.~$\mu(\cH)<M(\cH)$. }
Take $H_{0}\in \cH$ with $(G:H_{0})=M(\cH)$. Pick a maximal subgroup $P$ of $G$ containing $H_{0}$. Then, for any $H\in \cH$ with $(G:H)<M(\cH)$, one has $H_{0}\not\subset P\cap H$ and $H_{0}\not\supset P\cap H$. In particular, $\cH_{P}^{\red}$ contains at least two elements. This completes the proof in this case. 

\textbf{Case 1-b.~$\mu(\cH)=M(\cH)$. }
In this case, the assertion follows from Lemma \ref{lem:nmlz} since $\#\cH \geq 2$ and $\mu(\cH)\geq p^{2}$. 

\textbf{Case 2.~$\cH\setminus \cH^{\nor}$ is non-empty. }

Note that one has $M(\cH)\geq p^{2}$ by Proposition \ref{prop:mxps}. Put
\begin{equation*}
\sigma:=\mu(\cH\setminus \cH^{\nor}). 
\end{equation*}
Take $H\in \cH \setminus \cH^{\nor}$ with $(G:H)=\sigma$, and pick a maximal subgroup $P$ of $G$ containing $N_{G}(H)$. Fix a complete set of representatives $C$ of $P\backslash G/H\cong G/P$ in $G$. Let $H'\in \cH$. If $H'\in \cH^{\nor}$, then $gHg^{-1}\not\subset P\cap H'$ since $H'$ is normal in $G$. On the other hand, if $H'\in \cH \setminus \cH^{\nor}$, then the inclusion $gHg^{-1}\subsetneq P\cap g'H'(g')^{-1}$ does not happen for any $g'\in G$. This follows from $H\in \cH \setminus \cH^{\nor}$ and $(G:H)=\sigma$. Therefore, $\cH_{P}^{\red}$ contains $gHg^{-1}$ for any $g\in C$. In particular, $\cH_{P}^{\red}$ contains at least $p$ elements. Hence, the proof is complete.  
\end{proof}

\begin{rem}
In the proof of Theorem \ref{thm:ndqp}, Case 2 does not use the assumption $\#\cH \geq 2$. In particular, we also obtain an alternative proof of Proposition \ref{prop:emn1} in the case where $H$ is not normal in $G$. 
\end{rem}

\begin{proof}[Proof of Theorem \ref{mth1}]
Put $G:=\Gal(L/k)$ and
\begin{equation*}
\cH:=\{\Gal(L/K_{1}),\ldots,\Gal(L/K_{r})\}. 
\end{equation*}
By definition, $\cH$ is reduced, $\#\cH=r$ and $N^{G}(\cH)=\{1\}$. Moreover, Proposition \ref{prop:coch} gives an isomorphism $X^{*}(T_{\bK/k})\cong J_{G/\cH}$. Hence, the assertion follows from Theorem \ref{thm:ndqp}. 
\end{proof}

\section{Flabby resolutions of particular lattices}\label{sect:flrs}

\subsection{Quasi-permutation lattices}

Here we give some examples of $G$-lattices $J_{G/\cH}^{(\varphi)}$ that are quasi-permutation. For $n\in \Zpn$, we denote by $D_{n}$ the dihedral group of order $2n$, that is, 
\begin{equation*}
D_{n}:=\langle \sigma_{n},\tau_{n} \mid \sigma_{n}^{n}=\tau_{n}^{2}=1,\tau_{n}\sigma_{n}\tau_{n}=\sigma_{n}^{-1}\rangle. 
\end{equation*}

\begin{thm}\label{thm:dnqp}
Let $m$ be a positive integer. Consider a strongly reduced set 
\begin{equation*}
\cH:=\{\langle \tau_{2m} \rangle, \langle \sigma_{2m}\tau_{2m} \rangle \}
\end{equation*}
of subgroups of $D_{2m}$. Then there is an exact sequence of $D_{2m}$-lattices 
\begin{equation*}
    0\rightarrow J_{D_{2m}/\cH}\rightarrow \Z[D_{2m}]\oplus \Z \rightarrow \Z[D_{2m}/\langle \sigma_{2m} \rangle]\rightarrow 0. 
\end{equation*}
In particular, $J_{D_{2m}/\cH}$ is quasi-permutation. 
\end{thm}

Theorem \ref{mth4} follows from Theorem \ref{thm:dnqp} by the same argument as Theorem \ref{mth1}. 

\begin{proof}
Consider three homomorphisms of $D_{2m}$-lattices as follows: 
\begin{gather*}
\iota_{m}\colon \Z[D_{2m}/\langle \sigma_{2m}\rangle]\rightarrow \Z[D_{2m}]\oplus \Z;\,1\mapsto (1+\cdots +\sigma_{2m}^{2m-1},-1);\\
\omega_{m}\colon \Z[D_{2m}]\rightarrow I_{D_{2m}/\cH};\,1\mapsto (1,-1);\\
\psi_{m}\colon \Z \rightarrow I_{D_{2m}/\cH};\,1\mapsto 
(1+\cdots +\sigma_{2m}^{2m-1},-(1+\cdots +\sigma_{2m}^{2m-1})). 
\end{gather*}
It suffices to prove that the sequence of $D_{2m}$-lattices
\begin{equation}\label{eq:dnfl}
0\rightarrow \Z[D_{2m}/\langle \sigma_{2m}\rangle] \xrightarrow{\iota_{m}} \Z[D_{2m}]\oplus \Z \xrightarrow{(\omega_{m},\psi_{m})} I_{D_{2m}/\cH}\rightarrow 0
\end{equation}
is exact. The surjectivity of $(\omega_{m},\psi_{m})$ follows from $\omega_{m}(1-\sigma_{2m}\tau_{2m})=(1-\sigma_{2m},0)$. Inclusion $\Ima(\iota_{m})\subset \Ker(\omega_{m},\psi_{m})$ can be confirmed by direct computation. For the reverse inclusion, pick an element $x=(x_{1},x_{2})$ from $\Ker(\omega_{m},\psi_{m})$. Then we have
\begin{equation*}
x+x_{2}\iota_{m}(1)=(x_{1}+x_{2}(1+\cdots +\sigma_{2m}^{2m-1}),0). 
\end{equation*}
This implies that the first factor of $x+x_{2}\iota_{m}(1)$ is contained in $\Ker(\omega_{m})$. However, since $\Ker(\omega_{m})$ is generated by $(1-\tau_{2m})(1+\cdots+\sigma_{2m}^{2m-1})$ as an abelian group, we see that $x+x_{2}\iota_{m}(1)$ is a multiple of $\iota_{m}(1-\tau_{2m})$ by an integer. Consequently, one has $x\in \Ima(\iota_{m})$ as desired. 
\end{proof}

\begin{rem}
    If $m=1$, then the exact sequence \eqref{eq:dnfl} is given by \cite[p.~30]{Endo2001} (note that the map $\nu$ in \cite[p.30, l.18]{Endo2001} needs to be modified as $\nu(1)=(1+\tau, -(1+\sigma))$). 
\end{rem}

\begin{lem}\label{lem:cfad}
Let $G$ be a finite group, and 
\begin{equation*}
0\rightarrow F \xrightarrow{\iota} R \xrightarrow{\Phi} M \rightarrow 0
\end{equation*}
a coflabby resolution of a $G$-lattice $M$. Consider a homomorphism of $G$-lattices $\psi \colon R'\rightarrow M$, where $R'$ is a permutation $G$-lattice. We denote by $\Phi'$ the sum of $\Phi$ and $\psi$. Then there is an exact sequence
\begin{equation}\label{eq:adex}
0\rightarrow F\oplus R'\xrightarrow{\iota'} R\oplus R'\xrightarrow{\Phi'}M \rightarrow 0
\end{equation}
which satisfies a commutative diagram
\begin{equation*}
\xymatrix@C=35pt{
F\ar[r]^{x \mapsto (x,0)\hspace{10pt}}\ar[d]^{\iota}& F\oplus R'\ar[d]^{\iota'}&
R'\ar[l]_{\hspace{15pt}x\mapsto (0,x)}\ar@{=}[d]\\
R \ar[r]^{x\mapsto (x,0)\hspace{10pt}}& 
R\oplus R' \ar[r]^{\hspace{10pt}\pr_{2}}& R'. 
}
\end{equation*}
In particular, \eqref{eq:adex} is also a coflabby resolution of $M$. 
\end{lem}

\begin{proof}
Consider the commutative diagram
\begin{equation*}
\xymatrix{
0\ar[r]& R \ar[r]\ar[d]^{\Phi}&R\oplus R'\ar[r]\ar[d]^{\Phi'}& R'\ar[r]\ar[d]&0\\
0\ar[r]& M \ar@{=}[r]&M\ar[r]& 0\ar[r]& 0,
}
\end{equation*}
where the horizontal sequences are exact. Applying the snake lemma to this diagram, we get an exact sequence
\begin{equation}\label{eq:cdse}
0\rightarrow F \rightarrow F'\rightarrow R'\rightarrow 0. 
\end{equation}
Here $F'$ is the kernel of $\Phi'$. However, one has $\Ext_{\Z[G]}^{1}(R',F)=0$ by \cite[Lemme 1]{ColliotThelene1977}. Hence \eqref{eq:cdse} splits, and the proof is complete. 
\end{proof}

\begin{lem}\label{lem:lmis}
Consider a commutative diagram of finite free abelian groups
\begin{equation*}
\xymatrix{
0\ar[r]& M_{1}\ar[r] \ar[d]^{f_{1}} & M_{2} \ar[r]\ar[d]^{f_{2}}& M_{3} \ar[r]\ar[d]^{f_{3}}& 0\\
0\ar[r]& M'_{1}\ar[r] & M'_{2} \ar[r] & M'_{3}, &
}
\end{equation*}
where the horizontal sequences are exact. We further assume that 
\begin{itemize}
    \item $\rk_{\Z}(M_{1})=\rk_{\Z}(M'_{1})$; 
    \item $f_{2}$ and $f_{3}$ are injective; and 
    \item the cokernel of $f_{2}$ is torsion-free. 
\end{itemize}
Then the homomorphism $f_{1}$ is an isomorphism. 
\end{lem}

\begin{proof}
By the snake lemma, one has an exact sequence
\begin{equation*}
0\rightarrow \Ker(f_{1})\rightarrow \Ker(f_{2})\rightarrow \Ker(f_{3})\rightarrow \Coker(f_{1})\rightarrow \Coker(f_{2}). 
\end{equation*}
Since $\Ker(f_{2})$ is trivial by assumption, we have $\Ker(f_{1})=0$. On the other hand, since $\Ker(f_{3})=0$ and $\Coker(f_{2})$ is torsion-free, we obtain that $\Coker(f_{1})$ has no torsion. However, the equality $\rk_{\Z}(M_{1})=\rk_{\Z}(M'_{1})$ implies $\rk_{\Z}(\Coker(f_{1}))=0$, and hence $\Coker(f_{1})=0$. This completes the proof. 
\end{proof}

\begin{lem}\label{lem:fgdc}
Let $G$ be a finite group, and $N$ its normal subgroup. Consider a commutative diagram of $G$-lattices as follows: 
\begin{equation*}
\xymatrix{
0\ar[r]& M_{1}\ar[r]\ar[d]^{f_{1}}& M_{2}\ar[r]\ar[d]^{f_{2}}& M_{3}\ar[r]\ar[d]^{f_{3}}&0\\
0\ar[r]& M'_{1}\ar[r]& M'_{2}\ar[r]& M'_{3}\ar[r]&0,  
}
\end{equation*}
where the horizontal sequences are exact. We further assume that 
\begin{itemize}
\item[(a)] $H^{1}(N,\Ker(f_{1}))=0$; and 
\item[(b)] $N$ acts trivially on $M_{3}$.  
\end{itemize}
Then $\Ker(f_{2})$ is generated by $\Ker(f_{1})$ and $\Ker(f_{2})^{N}$. 
\end{lem}

\begin{proof}
We denote by $M_{3}^{\dagger}$ the image of $\Ker(f_{2})$ under the surjection $M_{2}\twoheadrightarrow M_{3}$. Then one has an exact sequence
\begin{equation}\label{eq:krex}
0\rightarrow \Ker(f_{1})\rightarrow \Ker(f_{2})\rightarrow M_{3}^{\dagger}\rightarrow 0. 
\end{equation}
Since $M^{\dagger}_{3}$ is contained in $M_{3}$, the assumption (b) implies that the action of $N$ on $M_{3}^{\dagger}$ is trivial. Taking $N$-fixed parts of \eqref{eq:krex}, we obtain an exact sequence
\begin{equation*}
0\rightarrow \Ker(f_{1})^{N}\rightarrow \Ker(f_{2})^{N}\rightarrow M_{3}^{\dagger}\rightarrow 0. 
\end{equation*}
Here we use (a) for the surjectivity. This implies the desired assertion. 
\end{proof}

\begin{lem}\label{lem:cssp}
Let $G$ be a finite group, $H$ its subgroup, and $N_{1}$ and $N_{2}$ normal subgroups of $G$. Consider homomorphisms of $G$-lattices as follows: 
\begin{gather*}
f_{1}\colon M_{1}\oplus \Z[G/H]^{N_{1}}\rightarrow M'_{1}\oplus \Z[G/H];\\
f_{2}\colon M_{2}\oplus \Z[G/H]^{N_{2}}\rightarrow M'_{1}\oplus \Z[G/H]\oplus M'_{2}. 
\end{gather*}
Let $f$ be the sum of $f_{1}$ and $f_{2}$, and denote by $M' \subset M'_{1}\oplus \Z[G/H]\oplus M'_{2}$ the image of $f$. We further assume that
\begin{itemize}
\item[(1)] $\gcd((HN_{1}N_{2}:HN_{1}),(HN_{1}N_{2}:HN_{2}))=1$; 
\item[(2)] there exist commutative diagrams
\begin{equation*}
\xymatrix@C=35pt{
M_{1}\oplus \Z[G/H]^{N_{1}}\ar[d]^{f_{1}}&
\Z[G/H]^{N_{1}}\ar[l]_{\hspace{15pt}x\mapsto (0,x)}\ar[d]^{x\mapsto x}\\
M'_{1}\oplus \Z[G/H]\ar[r]^{\hspace{10pt}\pr_{2}}&\Z[G/H],
}\quad
\xymatrix@C=30pt{
M_{2}\oplus \Z[G/H]^{N_{2}}\ar[d]^{f_{2}}&
\Z[G/H]^{N_{2}}\ar[l]_{\hspace{15pt}x\mapsto (0,x)}\ar[d]^{x\mapsto x}\\
M'_{1}\oplus \Z[G/H]\oplus M'_{2}\ar[r]^{\hspace{25pt}\pr_{2}}&
\Z[G/H];
}
\end{equation*}
\item[(3)] $M'\cap (M'_{1}\oplus \{0\}\oplus M'_{2})^{N_{1}N_{2}}\subset f_{1}(M_{1}\oplus \{0\})+f_{2}(M_{2}\oplus \{0\})$; and 
\item[(4)] $\rk_{\Z}\Ker(f)=(G:HN_{1}N_{2})$. 
\end{itemize}
Then there is an isomorphism
\begin{equation*}
M' \oplus \Z[G/HN_{1}N_{2}]\cong M_{1}\oplus M_{2}\oplus \Z[G/HN_{1}] \oplus \Z[G/HN_{2}]
\end{equation*}
such that the following diagram is commutative: 
\begin{equation*}
\xymatrix{
M_{1}\oplus M_{2}\ar[r]^{f}\ar[d]^{(x,y)\mapsto (x,y,0,0)}&
M'\ar[d]^{x\mapsto (x,0)}\\
M_{1}\oplus M_{2}\oplus \Z[G/HN_{1}] \oplus \Z[G/HN_{2}]\ar[r]^{\hspace{40pt}\cong}&
M'\oplus \Z[G/HN_{1}N_{2}]. 
}
\end{equation*}
\end{lem}

\begin{proof}
For each $i\in \{1,2\}$, the homomorphism $\Ind_{HN_{i}}^{G}\varepsilon_{HN_{i}/H}^{\circ}$ induces an isomorphism
\begin{equation*}
\varepsilon_{i}\colon \Z[G/HN_{i}]\xrightarrow{\cong}\Z[G/H]^{N_{i}}. 
\end{equation*}
Take two integers $c_{1}$ and $c_{2}$ which satisfy $c_{1}(HN_{1}N_{2}:HN_{1})-c_{2}(HN_{1}N_{2}:HN_{2})=1$. Note that it is possible by (1). We define $\pi$ as the composite
\begin{align*}
(M_{1}\oplus \Z[G/H]^{N_{1}})\oplus (M_{2}\oplus \Z[G/H]^{N_{2}})\xrightarrow{(\varepsilon_{1}^{-1}\circ \pr_{2})\oplus (\varepsilon_{2}^{-1}\circ \pr_{2})}&\,\,\Z[G/HN_{1}]\oplus \Z[G/HN_{2}]\\
\xrightarrow{(c_{1}\Ind_{HN_{1}N_{2}}^{G}\varepsilon_{HN_{1}N_{2}/HN_{1}},\,c_{2}\Ind_{HN_{1}N_{2}}^{G}\varepsilon_{HN_{1}N_{2}/HN_{2}})}&\,\,\Z[G/HN_{1}N_{2}]
\end{align*}
Furthermore, put $E:=\Ker(f)$. 

\begin{claim}
The restriction of $\pi$ to $E$ is an isomorphism. 
\end{claim}

We first prove the surjectivity. Put
\begin{equation*}
y:=\Ind_{HN_{1}N_{2}}^{G}\varepsilon_{HN_{1}N_{2}/H}^{\circ}(1)\in \Z[G/H]^{N_{1}N_{2}}. 
\end{equation*}
Then we have
\begin{equation*}
\pi((0,y),(0,-y))=c_{1}(HN_{1}N_{2}:HN_{1})-c_{2}(HN_{1}N_{2}:HN_{2})=1. 
\end{equation*}
Moreover, (2) implies the following: 
\begin{equation*}
f((0,y),(0,-y))\in M'\cap (M'_{1}\oplus \{0\}\oplus M'_{2})^{N_{1}N_{2}}. 
\end{equation*}
By (3), there exist $x_{1}\in M_{1}$ and $x_{2}\in M_{2}$ so that $f((x_{1},0),(x_{2},0))=f((0,y),(0,-y))$. Then $((-x_{1},y),(-x_{2},-y))$ lies in $E$ and
\begin{equation*}
\pi((-x_{1},y),(-x_{2},-y))=\pi((0,y),(0,-y))=1. 
\end{equation*}

Now, (4) implies that the kernel of the restriction of $\pi$ to $E$ has rank $0$. Hence it must be trivial, and the proof of Claim is complete. \qed

\vspace{5pt}
By Claim, we get a commutative diagram
\begin{equation*}
\xymatrix{
0\ar[r]&E \ar[r]\ar[rd]_{\cong}& (M_{1}\oplus \Z[G/H]^{N_{1}})\oplus (M_{2}\oplus \Z[G/H]^{N_{2}})\ar[r]^{\hspace{95pt}f}\ar[d]^{\pi}& M'\ar[r]&0\\
&&\Z[G/HN_{1}N_{2}],&&
}
\end{equation*}
where the horizontal sequence is exact. In particular, this exact sequence splits, and therefore $f$ and $\pi$ induce an isomorphism
\begin{equation*}
M'\oplus \Z[G/HN_{1}N_{2}] \cong M_{1}\oplus M_{2}\oplus \Z[G/H]^{N_{1}}\oplus \Z[G/H]^{N_{2}}
\end{equation*}
which satisfies the desired property. 
\end{proof}

The following generalizes Theorem \ref{thm:dnqp} in the case that $n$ is a power of $2$.

\begin{thm}\label{thm:nlqp}
Let $m$ be an odd positive integer, and $\nu \in \Zpn$. Consider the finite group
\begin{align*}
G_{m,\nu}:&=C_{m}\times D_{2^{\nu}}\\
&=\langle \rho_{m},\sigma_{2^{\nu}},\tau_{2^{\nu}} \mid \rho_{m}^{m}=\sigma_{2^{\nu}}^{2^{\nu}}=\tau_{2^{\nu}}^{2}=1,\rho_{m}\sigma_{2^{\nu}}=\sigma_{2^{\nu}}\rho_{m},\rho_{m}\tau_{2^{\nu}}=\tau_{2^{\nu}}\rho_{m},\tau_{2^{\nu}}\sigma_{2^{\nu}}\tau_{2^{\nu}}^{-1}=\sigma_{2^{\nu}}^{-1} \rangle, 
\end{align*}
and define $\cH_{m,\nu}:=\{\langle \tau_{2^{\nu}} \rangle,\langle \sigma_{2^{\nu}}\tau_{2^{\nu}} \rangle \}$. Then the $G_{m,\nu}$-lattice $J_{G_{m,\nu}/\cH_{m,\nu}}$ is quasi-permutation. 
\end{thm}

\begin{proof}
In this proof, put $I_{m,\nu}:=I_{G_{m,\nu}/\cH_{m,\nu}}$. By Theorem \ref{thm:dnqp}, we may assume $m>1$. Consider homomorphisms of $G$-lattices as follows: 
\begin{gather*}
\psi_{m,\nu,0}\colon R_{m,\nu,0}:=\Z[G_{m,\nu}/\langle \rho_{m} \rangle]\rightarrow I_{m,\nu};\,1\mapsto \left(\sum_{i=0}^{m-1}\rho_{m}^{i},-\sum_{i=0}^{m-1}\rho_{m}^{i} \right);\\
\psi_{m,\nu,1}\colon R_{m,\nu,1}:=\Z \rightarrow I_{m,\nu};\,1\mapsto \left(\left(\sum_{i=0}^{m-1}\rho_{m}^{i}\right)\left(\sum_{j=0}^{2^{\nu}-1}\sigma_{2^{\nu}}^{j}\right),-\left(\sum_{i=0}^{m-1}\rho_{m}^{i}\right)\left(\sum_{j=0}^{2^{\nu}-1}\sigma_{2^{\nu}}^{j}\right) \right);\\
\psi_{m,\nu,2}\colon R_{m,\nu,2}:=\Z[G_{m,\nu}/\langle \tau_{2^{\nu}} \rangle]\rightarrow I_{m,\nu};\,1\mapsto (1+\rho_{m},-(1+\sigma_{2^{\nu}}^{-1}));\\
\omega_{m,\nu,i}\colon R'_{m,\nu,i}:=\Z[G_{m,\nu}/\langle \sigma_{2^{\nu}}^{2^{\nu-i}}, \sigma_{2^{\nu}} \tau_{2^{\nu}} \rangle]\rightarrow I_{m,\nu};\,1\mapsto \left(0,\left(\sum_{j=0}^{2^{i}-1}\sigma_{2^{\nu}}^{2^{\nu-i}j}\right)(1-\rho_{m}\sigma_{2^{\nu}}^{2^{\nu-i-1}}) \right). 
\end{gather*}
Here $i\in \{0,\ldots,\nu-1\}$. We denote by $\Phi_{m,\nu,\ast}\colon R_{m,\nu,\ast}\rightarrow I_{m,\nu}$ the sum of $\psi_{m,\nu,i}$ for all $i\in \{0,1,2\}$ and $\omega_{m,\nu,0}$. Then there is an isomorphism $\Coker(\Phi_{m,\nu,\ast})\cong (\Z/m\Z)^{\oplus 2^{\nu-1}-1}$, which follows from the following: 
\begin{gather*}
\left(1,\sum_{j=0}^{(m-3)/2}\rho_{m}^{1+2j}(1+\sigma_{2^{\nu}}^{-1})-\sum_{i=0}^{m-1}\rho_{m}^{i}\right)=\psi_{m,\nu,0}(1)-\sum_{j=0}^{(m-3)/2}\psi_{m,\nu,2}(\rho_{m}^{1+2j});\\
(0,1-\rho_{m})=\omega_{m,\nu,0}((1+\rho_{m}^{2}+\cdots+\rho_{m}^{m-1})(1+\rho_{m}\sigma_{2^{\nu}}^{2^{\nu-1}}));\\
(0,1-\sigma_{2^{\nu}}^{2^{\nu-1}})=\omega_{m,\nu,0}(1+\rho_{m}\sigma_{2^{\nu}}^{2^{\nu-1}}+\cdots +(\rho_{m}\sigma_{2^{\nu}}^{2^{\nu-1}})^{m-1});\\
m(0,1-\sigma_{2^{\nu}}^{-1})=\psi_{m,\nu,0}(\tau_{2^{\nu}}-1)+(1-\sigma_{2^{\nu}}^{-1})\sum_{i=1}^{m-1}(0,1-\rho_{m}^{i}). 
\end{gather*}
In particular, $\Phi_{m,1,\ast}$ is surjective. On the other hand, for $\nu \geq 2$, we write $\Phi_{m,\nu,\bullet}\colon R'_{m,\nu,\bullet}\rightarrow I_{m,\nu}$ for the sum of $\omega_{m,\nu,i}$ for all $i\in \{1,\ldots,\nu-1\}$. Moreover, $\Phi_{m,\nu} \colon R_{m,\nu}\rightarrow I_{m,\nu}$ denotes the sum of $\Phi_{m,\nu,\ast}$ and $\Phi_{m,\nu,\bullet}$. Then, we have
\begin{equation*}
2^{i}(0,1-\rho_{m}\sigma_{2^{\nu}}^{2^{\nu-i-1}})=\omega_{m,\nu,i}(1)+\sum_{j=1}^{2^{i}-1}(1-\rho_{m}\sigma_{2^{\nu}}^{2^{\nu-i-1}})(0,1-\sigma_{2^{\nu}}^{2^{\nu-i}j})
\end{equation*}
for any $i\in \{1,\ldots,\nu-1\}$. Hence, by induction, $\Phi_{m,\nu}$ is surjective for all $\nu \geq 2$. Consequently, we obtain an exact sequence of $G$-lattices
\begin{equation*}
0\rightarrow U_{m,\nu} \rightarrow R_{m,\nu} \rightarrow I_{m,\nu} \rightarrow 0. 
\end{equation*}
In the following, we write $C_{m}$ and $Z_{\nu}$ for the subgroups of $G_{m,\nu}$ generated by $\rho_{m}$ and $\sigma_{2^{\nu}}^{2^{\nu-1}}$ respectively. In addition, we define $G_{m,\nu}$-sublattices of $R_{m,\nu}$ as follows: 
\begin{equation*}
R_{m,\nu,\ast}^{-}:=R_{m,\nu,0}\oplus R_{m,\nu,1}\oplus R_{m,\nu,2},\quad R_{m,\nu}^{-}:=R_{m,\nu,\ast}^{-}\oplus R'_{m,\nu,\bullet}, 
\end{equation*}

\begin{claim}
\begin{enumerate}
\item There exist isomorphisms of $G_{m,\nu}$-lattices
\begin{align}
U_{m,\nu} \cap R_{m,\nu,\ast}^{-}&\cong \Z[G_{m,\nu}/\langle \rho_{m},\sigma_{2^{\nu}}\rangle]\oplus R_{m,\nu,2}^{C_{m}},\label{eq:gmnr}\\
U_{m,\nu} \cap R_{m,\nu,\ast}&\cong (U_{m,\nu} \cap R_{m,\nu,\ast}^{-}) \oplus (R'_{m,\nu,0})^{C_{m}}\label{eq:gmna}. 
\end{align}
Moreover, the following diagram is commutative: 
\begin{equation}\label{eq:fcua}
\begin{gathered}
\xymatrix@C=35pt{
U_{m,\nu} \cap R_{m,\nu,\ast}^{-} 
\ar[r]^{x\mapsto (x,0)\hspace{40pt}}\ar[d]^{x\mapsto x}& 
(U_{m,\nu} \cap R_{m,\nu,\ast}^{-})\oplus (R'_{m,\nu,0})^{C_{m}}\ar[d]&
(R'_{m,\nu,0})^{C_{m}}
\ar[l]_{\hspace{45pt}x\mapsto (0,x)}\ar[d]^{x\mapsto x}\\
R_{m,\nu,\ast}^{-}\ar[r]^{x\mapsto (x,0)}& 
R_{m,\nu,\ast}^{-}\oplus R'_{m,\nu,0}\ar[r]^{\pr_{2}}&
R'_{m,\nu,0}. 
}
\end{gathered}
\end{equation}
Here, the central vertical map is induced by \eqref{eq:gmna}. 
\item We have $U_{m,\nu}=(U_{m,\nu}\cap R_{m,\nu,\ast})+U_{m,\nu}^{Z_{\nu}}$. 
\end{enumerate}
\end{claim}

\begin{proof}[Proof of Claim]
(i): By the definition of $\Phi_{m,\nu,\ast}$, we have a commutative diagram
\begin{equation*}
\xymatrix@C=35pt{
0\ar[r]& U_{m,\nu}\cap R_{m,\nu,\ast}^{C_{m}} \ar[r] \ar[d] & R_{m,\nu,\ast}^{C_{m}} \ar[r]^{\Phi_{m,\nu,\ast}}\ar[d]& I_{m,\nu}^{C_{m}} \ar[r]\ar[d]& 0\\
0\ar[r]& U_{m,\nu}\cap R_{m,\nu,\ast} \ar[r] & R_{m,\nu,\ast} \ar[r]^{\Phi_{m,\nu,\ast}} & I_{m,\nu},&
}
\end{equation*}
where the horizontal sequences are exact. Note that the central and the rightmost vertical maps are injective, and the cokernel of the central vertical map is torsion-free. Moreover, since the cokernel of $\Phi_{m,\nu,\ast}$ is torsion, we have 
\begin{align*}
\rk_{\Z}(U_{m,\nu}\cap R_{m,\nu,\ast})&=\rk_{\Z}(R_{m,\nu,\ast})-\rk_{\Z}(I_{m,\nu})\\
&=(2^{\nu+1}(m+1)+1)-(2^{\nu+1}m-1)\\
&=2^{\nu+1}+2. 
\end{align*}
Therefore, the ranks of $U_{m,\nu}\cap R_{m,\nu,\ast}$ and $U_{m,\nu}\cap R_{m,\nu,\ast}^{C_{m}}$ coincide. Hence, Lemma \ref{lem:lmis} implies
\begin{equation*}
U_{m,\nu}\cap R_{m,\nu,\ast}^{C_{m}}=U_{m,\nu}\cap R_{m,\nu,\ast}. 
\end{equation*}
On the other hand, Theorem \ref{thm:dnqp} gives an exact sequence of $G_{m,\nu}$-lattices
\begin{equation*}
0\rightarrow \Z[G_{m,\nu}/\langle \rho_{m},\sigma_{2^{\nu}} \rangle]\rightarrow R_{m,\nu,0}\oplus R_{m,\nu,1} \rightarrow I_{G_{m,\nu}/\cH_{m,\nu}}^{C_{m}}\rightarrow 0. 
\end{equation*}
Hence the assertion follows from Lemma \ref{lem:cfad}. 

(ii): By (i), $U_{m,\nu}\cap R_{m,\nu,\ast}=\Ker(\Phi_{m,\nu,\ast})$ is a permutation $G_{m,\nu}$-lattice. In particular, we have $H^{1}(Z_{\nu},U_{m,\nu}\cap R_{m,\nu,\ast})=0$. Moreover, one has $(R'_{m,\nu,\bullet})^{Z_{\nu}}=R'_{m,\nu,\bullet}$ by definition. Hence, the assertion follows from Lemma \ref{lem:fgdc} for the commutative diagram
\begin{equation*}
\xymatrix{
0\ar[r]& R_{m,\nu,\ast} \ar[r] \ar[d]^{\Phi_{m,\nu,\ast}} & R_{m,\nu} \ar[r]^{\pi'_{m,\nu}}\ar[d]^{\Phi_{m,\nu}}& R'_{m,\nu,\bullet} \ar[r]\ar[d]& 0\\
0\ar[r]& I_{m,\nu} \ar@{=}[r] & I_{m,\nu} \ar[r] & 0 \ar[r]&0, 
}
\end{equation*}
where $\pi'_{m,\nu}$ is the canonical projection $R_{m,\nu}\twoheadrightarrow R'_{m,\nu,\bullet}$. 
\end{proof}

\vspace{6pt}

In the following, we prove that there is an isomorphism of $G_{m,\nu}$-lattices
\begin{equation}\label{eq:fmns}
U_{m,\nu}\oplus S_{m,\nu} \cong (U_{m,\nu}\cap R_{m,\nu,\ast}^{-})\oplus S'_{m,\nu}\oplus S_{m,\nu},
\end{equation}
where
\begin{gather*}
S_{m,\nu}:=\bigoplus_{i=1}^{\nu-1}\Z[G_{m,\nu}/\langle \rho_{m},\sigma_{2^{\nu}}^{2^{\nu-i}},\sigma_{2^{\nu}}\tau_{2^{\nu}}\rangle],\quad
S'_{m,\nu}:=(R'_{m,\nu,0})^{C_{m}}\oplus \left(\bigoplus_{i=1}^{\nu-1}\Z[G_{m,\nu}/\langle \sigma_{2^{\nu}}^{2^{\nu-i}},\sigma_{2^{\nu}}\tau_{2^{\nu}}\rangle]\right),
\end{gather*}
such that the following diagram is commutative: 
\begin{equation*}
\xymatrix{
U_{m,\nu}\cap R_{m,\nu,\ast}^{-}\ar[r]^{x\mapsto x}\ar[d]^{x\mapsto (x,0,0)}&
U_{m,\nu}\ar[d]^{x\mapsto (x,0)}\\
(U_{m,\nu}\cap R_{m,\nu,\ast}^{-})\oplus S'_{m,\nu}\oplus S_{m,\nu}
\ar[r]^{\hspace{40pt}\cong}&
U_{m,\nu}\oplus S_{m,\nu}. 
}
\end{equation*}
This implies the desired assertion, since \eqref{eq:gmnr} implies that $U_{m,\nu}\cap R_{m,\nu,\ast}^{-}$ is permutation. We prove \eqref{eq:fmns} by induction on $\nu$. If $\nu=1$, then the $G_{m,\nu}$-lattice $R'_{m,\nu,\bullet}$ is zero. Hence, the assertion follows from Claim (i). Next, let $\nu \geq 2$, and suppose that the assertion holds for $\nu-1$. By definition, there is a canonical isomorphism $I_{m,\nu-1}\cong I_{m,\nu}^{Z_{\nu}}$. Moreover, the restriction of $\Phi_{m,\nu}$ to $R_{m,\nu}^{-}$ induces a coflabby resolution of $I_{m,\nu}^{Z_{\nu}}$ as follows: 
\begin{equation*}
0\rightarrow U_{m,\nu}\cap (R_{m,\nu}^{-})^{Z_{\nu}} \rightarrow (R_{m,\nu}^{-})^{Z_{\nu}} \rightarrow I_{m,\nu}^{Z_{\nu}}\rightarrow 0. 
\end{equation*}
Consequently, Lemma \ref{lem:cfad} gives an isomorphism
\begin{equation*}
U_{m,\nu}^{Z_{\nu}} \cong (U_{m,\nu}\cap (R_{m,\nu}^{-})^{Z_{\nu}})\oplus (R'_{m,\nu,0})^{Z_{\nu}}
\end{equation*}
and a commutative diagram
\begin{equation}\label{eq:xitc}
\begin{gathered}
\xymatrix@C=35pt{
U_{m,\nu}\cap (R_{m,\nu}^{-})^{Z_{\nu}}
\ar[r]^{x\mapsto (x,0)\hspace{35pt}}\ar[d]^{x\mapsto x}&
(U_{m,\nu}\cap (R_{m,\nu}^{-})^{Z_{\nu}})\oplus (R'_{m,\nu,0})^{Z_{\nu}}\ar[d]&
(R'_{m,\nu,0})^{Z_{\nu}} \ar[l]_{\hspace{55pt}x\mapsto (0,x)}\ar[d]^{x\mapsto x}\\
R_{m,\nu}^{-}\ar[r]^{x\mapsto (x,0)}& R_{m,\nu}^{-}\oplus R'_{m,\nu,0} \ar[r]^{\pr_{2}}&
R'_{m,\nu,0}. 
}
\end{gathered}
\end{equation}
On the other hand, by the induction hypothesis, we obtain an isomorphism of $G_{m,\nu}$-lattices
\begin{equation}\label{eq:fsdd}
(U_{m,\nu}\cap (R_{m,\nu}^{-})^{Z_{\nu}}) \oplus S_{m,\nu-1} \cong (U_{m,\nu}\cap (R_{m,\nu,\ast}^{-})^{Z_{\nu}})\oplus S'_{m,\nu-1} \oplus S_{m,\nu-1}
\end{equation}
such that the following diagram is commutative: 
\begin{equation*}
\xymatrix{
U_{m,\nu}\cap (R_{m,\nu,\ast}^{-})^{Z_{\nu}}\ar[r]^{x\mapsto x}\ar[d]^{x\mapsto (x,0,0)}&
U_{m,\nu}\cap (R_{m,\nu}^{-})^{Z_{\nu}}\ar[d]^{x\mapsto (x,0)}\\
(U_{m,\nu}\cap (R_{m,\nu,\ast}^{-})^{Z_{\nu}})\oplus S'_{m,\nu-1} \oplus S_{m,\nu-1}\ar[r]^{\hspace{20pt}\cong}&
(U_{m,\nu}\cap (R_{m,\nu}^{-})^{Z_{\nu}}) \oplus S_{m,\nu-1}. 
}
\end{equation*}

Now, we denote by $\Xi_{m,\nu,1}$ the central vertical homomorphism in \eqref{eq:fcua}, that is, 
\begin{equation*}
    (U_{m,\nu}\cap R_{m,\nu,\ast}^{-})\oplus (R'_{m,\nu,0})^{C_{m}} \rightarrow R_{m,\nu,\ast}^{-}\oplus R'_{m,\nu,0}
\end{equation*}
On the other hand, set $S_{m,\nu-1}^{\dagger}:=S'_{m,\nu-1}\oplus S_{m,\nu-1}$, and consider the $G$-homomorphism
\begin{equation*}
\Xi_{m,\nu,2}\colon S_{m,\nu-1}^{\dagger}\oplus (R'_{m,\nu,0})^{Z_{\nu}}\rightarrow R_{m,\nu}\oplus S_{m,\nu-1}=R_{m,\nu,\ast}^{-}\oplus R'_{m,\nu,0}\oplus (R'_{m,\nu,\bullet}\oplus S_{m,\nu-1}),
\end{equation*}
induced by \eqref{eq:fsdd} and the central vertical map of \eqref{eq:xitc}. Denote by $\Xi_{m,\nu}$ the sum of $\Xi_{m,\nu,1}$ and $\Xi_{m,\nu,2}$. Then we have
\begin{equation}\label{eq:imxi}
    \Ima(\Xi_{m,\nu})=U_{m,\nu}\oplus S_{m,\nu-1},
\end{equation}
which is a consequence of Claim (ii) and \eqref{eq:fsdd}. In particular, one has an exact sequence
\begin{equation}\label{eq:ffex}
0\rightarrow E_{m,\nu}\rightarrow (U_{m,\nu}\cap R_{m,\nu,\ast}) \oplus S_{m,\nu-1}^{\dagger} \oplus (R'_{m,\nu,0})^{Z_{\nu}} \xrightarrow{\Xi_{m,\nu}} U_{m,\nu}\oplus S_{m,\nu-1}\rightarrow 0. 
\end{equation}

Now, we confirm that $\Xi_{m,\nu,1}$ and $\Xi_{m,\nu,2}$ satisfy the four assumptions in Lemma \ref{lem:cssp}. The condition (1) is not difficult by using the property that $m$ is odd. Moreover, (2) is a consequence of Claim (i) and \eqref{eq:xitc}. On the other hand, \eqref{eq:fsdd} induces the following: 
\begin{align*}
\Xi_{m,\nu,1}(U_{m,\nu}\cap (R_{m,\nu,\ast}^{-})^{Z_{\nu}})+\Xi_{m,\nu,2}(S_{m,\nu-1}^{\dagger})&=(U_{m,\nu}\cap (R_{m,\nu}^{-})^{Z_{\nu}})\oplus S_{m,\nu-1}\\
&=(U_{m,\nu}\oplus S_{m,\nu-1})\cap (R_{m,\nu}^{-}\oplus S_{m,\nu-1})^{Z_{\nu}}. 
\end{align*}
Hence, (3) is valid by \eqref{eq:imxi}. Finally, \eqref{eq:ffex} implies $\rk_{\Z}(\Ker(\Xi_{m,\nu}))=2^{\nu-1}$, that is, (4) holds. Therefore, we can apply Lemma \ref{lem:cssp}. In particular, we obtain an isomorphism of $G_{m,\nu}$-lattices
\begin{equation}\label{eq:aflm}
\begin{gathered}
U_{m,\nu}\oplus S_{m,\nu-1}\oplus \Z[G_{m,\nu}/H_{\nu}C_{m}Z_{\nu}]
\\ \cong (U_{m,\nu}\cap R_{m,\nu,\ast}^{-})\oplus S_{m,\nu-1}^{\dagger}\oplus \Z[G_{m,\nu}/H_{\nu}C_{m}]\oplus \Z[G_{m,\nu}/H_{\nu}Z_{\nu}]
\end{gathered}
\end{equation}
where $H_{\nu}:=\langle \sigma_{2^{\nu}}\tau_{2^{\nu}}\rangle$, which induces the restriction of $\Xi_{m,\nu}$ to $(U_{m,\nu}\cap R_{m,\nu,\ast}^{-})\oplus S_{m,\nu-1}^{\dagger}$. On the other hand, by the definitions of $S_{m,\nu}$ and $S'_{m,\nu}$, there exist isomorphisms
\begin{gather*}
S_{m,\nu-1}\oplus \Z[G_{m,\nu}/H_{\nu}C_{m}Z_{\nu}]=S_{m,\nu},\\
S_{m,\nu-1}^{\dagger}\oplus \Z[G_{m,\nu}/H_{\nu}C_{m}]\oplus \Z[G_{m,\nu}/H_{\nu}Z_{\nu}]
\cong S'_{m,\nu}\oplus S_{m,\nu}.
\end{gather*}
Consequently, \eqref{eq:aflm} can be written as \eqref{eq:fmns}. Hence, the proof is complete. 
\end{proof}

\subsection{Non-quasi-invertible lattices}

We prove that some $J_{G/\cH}$ are not quasi-invertible. The proofs are based on \cite[\S\S 3--4]{Endo2001} and \cite[Lemma 2.2]{Endo2011}. For a $G$-lattice $M$, let 
\begin{equation*}
M_{2}:=M\otimes_{\Z}\Z_{2},
\end{equation*}
where $\Z_{2}$ is the ring of $2$-adic integers.

The following will be used in the next proposition. 

\begin{lem}\label{lem:v4qp}
Let $G:=(C_{2})^{2}=\langle \sigma,\tau \mid \sigma^{2}=\tau^{2}=1,\sigma\tau=\tau \sigma \rangle$, $\cH:=\{\{1\},G\}$, and $\varphi$ the weight function on $\cH$ defined as $\varphi(\{1\})=1$ and $\varphi(G)=2$. Then there is an exact sequence of $G$-lattices
\begin{equation*}
0\rightarrow \Z^{\oplus 2}\rightarrow \Z[G/\langle \sigma \rangle]\oplus \Z[G/\langle \tau \rangle]\oplus \Z[G/\langle \sigma \tau \rangle] \rightarrow I_{G/\cH}^{(\varphi)}\rightarrow 0. 
\end{equation*}
\end{lem}

\begin{proof}
Consider homomorphisms as follows: 
\begin{align*}
f_{1}\colon \Z[G/\langle \sigma \rangle]\rightarrow I_{G/\cH}^{(\varphi)};&\, 1\mapsto (1+\sigma,-1);\\
f_{2}\colon \Z[G/\langle \tau \rangle]\rightarrow I_{G/\cH}^{(\varphi)};&\, 1\mapsto (1+\tau,-1);\\
f_{3}\colon \Z[G/\langle \sigma \tau \rangle]\rightarrow I_{G/\cH}^{(\varphi)};&\, 1\mapsto (1+\sigma \tau,-1). 
\end{align*}
We set
\begin{equation*}
f\colon \Z[G/\langle \sigma \rangle]\oplus \Z[G/\langle \tau \rangle]\oplus \Z[G/\langle \sigma \tau \rangle]\rightarrow I_{G/\cH}^{(\varphi)};\,(x_{1},x_{2},x_{3})\mapsto f_{1}(x_{1})+f_{2}(x_{2})+f_{3}(x_{3}). 
\end{equation*}
Then $f$ is surjective, which follows from the following: 
\begin{equation*}
    (1-\sigma,0)=f_{2}(1)-f_{3}(\tau),\quad 
    (1-\tau,0)=f_{1}(1)-f_{3}(\sigma). 
\end{equation*}
Moreover, the kernel of $f$ is generated by $(1+\tau,-(1+\sigma),0)$ and $(0,1+\sigma,-(1+\sigma))$. These elements are fixed under $G$, and hence the assertion holds. 
\end{proof}

\begin{prop}\label{prop:apt1}
Let $G:=(C_{2})^{3}$ and $\cH:=\{C_{2}\times \{1\}\times \{1\},\{1\}\times C_{2}\times C_{2} \}$. Then the $G$-lattice $J_{G/\cH}$ is not quasi-invertible. 
\end{prop}

\begin{proof}
Write
\begin{equation*}
  G=\langle \rho,\sigma,\tau \mid \rho^{2}=\sigma^{2}=\tau^{2}=1,\rho \sigma=\sigma\rho,\sigma \tau=\tau\sigma, \tau \rho=\rho\tau \rangle.  
\end{equation*}
We may assume $\cH:=\{H_{1},H_{2}\}$, where $H_{1}=\langle \rho \rangle$ and $H_{2}=\langle \sigma,\tau \rangle$. Let $I=I_{G/\cH}$ and $J=J_{G/\cH}$. Regard $I$ as a $G$-submodule of $\Z[G/H_{1}]\oplus \Z[G/H_{2}]$, which induces an exact sequence
\begin{equation}\label{eq:2mvr}
    0\rightarrow I \rightarrow \Z[G/H_{1}]\oplus \Z[G/H_{2}] \xrightarrow{(\varepsilon_{G/H_{1}},\,\varepsilon_{G/H_{2}})} \Z \rightarrow 0. 
\end{equation}
Put $\overline{G}:=G/H_{1} \cong (C_{2})^{2}$. By definition, there is an isomorphism of $\overline{G}$-lattices
\begin{equation*}
    I^{H_{1}}\cong I_{\overline{G}/\overline{\cH}}^{(\varphi)}.
\end{equation*}
Here $\overline{\cH}:=\{\{1\},\overline{G}\}$ and $\varphi$ is the weight function on $\overline{\cH}$ defined as $\varphi(\{1\})=1$ and $\varphi(\overline{G})=2$. Hence, by Lemma \ref{lem:v4qp}, we obtain a coflabby resolution
\begin{equation*}
    0\rightarrow \Z^{\oplus 2}\rightarrow \Z[G/\langle \rho, \sigma \rangle]\oplus \Z[G/\langle \rho,\tau \rangle]\oplus \Z[G/\langle \rho,\sigma \tau \rangle] \xrightarrow{f'} I^{\langle \rho \rangle}\rightarrow 0. 
\end{equation*}
On the other hand, we have the following for any $g\in \{\sigma,\tau,\sigma\tau \}$: 
\begin{equation*}
    I^{\langle g\rangle}=I^{\langle \rho,g\rangle}+\langle (0,1-\rho) \rangle_{\Z},\quad
    I^{\langle \rho g\rangle}=I^{\langle \rho,g\rangle}. 
\end{equation*}
Moreover, $I$ is generated by $(1,-1)$ as a $G$-lattice. Consider homomorphisms of $G$-lattices
\begin{align*}
f_{1}\colon \Z[G]\rightarrow I&;\,1\mapsto (1,-1);\\
f_{2}\colon \Z[G/\langle \sigma, \tau \rangle]\rightarrow I&;\,1\mapsto  (0,1-\rho). 
\end{align*}
Then $f'$, $f_{1}$ and $f_{2}$ induce a coflabby resolution of $I$: 
\begin{equation*}
0\rightarrow U \rightarrow R \xrightarrow{f} I \rightarrow 0. 
\end{equation*}

\begin{claim}
    The exponent of $\Hhat^{0}(G,U)$ is a divisor of $4$. 
\end{claim}

\begin{proof}[Proof of Claim]
Consider an exact sequence
\begin{equation*}
    \Hhat^{-1}(G,I)\rightarrow \Hhat^{0}(G,U)\rightarrow \Hhat^{0}(G,R). 
\end{equation*}
Then we have $\Hhat^{-1}(G,I)=0$, which follows from the surjectivity of the horizontal homomorphisms of the commutative diagram
\begin{equation*}
\xymatrix@C=80pt{
    \Hhat^{-2}(H_{1},\Z)\oplus \Hhat^{-2}(H_{2},\Z)\ar[r]^{\hspace{35pt}(\Cor_{G/H_{1}},\,\Cor_{G/H_{2}})} \ar[d]^{\cong}& \Hhat^{-2}(G,\Z) \ar[d]^{\cong}\\
    H_{1}\oplus H_{2}\ar[r]& G. 
}
\end{equation*}
Moreover, there is an isomorphism $\Hhat^{0}(G,R)\cong (\Z/4\Z)^{\oplus 4}$. This implies that the exponent of $\Hhat^{0}(G,U)$ is a divisor of $4$.
\end{proof}

\vspace{6pt}
Now suppose that $J$ is quasi-invertible, that is, $U^{\circ}$ is invertible. Then $U$ is also invertible. Moreover, by the same argument as \cite[Lemma 2.2]{Endo2011}, $U_{2}$ is a permutation $\Z_{2}[G]$-lattice. Hence $U_{2}$ contains $\Z_{2}$ as a direct summand of $\Z_{2}[G]$-modules since $\rk_{\Z}(U)=11$. In particular, $\Hhat^{0}(G,U)$ has exponent $8$, which contradicts Claim. Therefore, the $G$-lattice $J$ is not quasi-invertible as desired. 
\end{proof}

\begin{prop}\label{prop:apt2}
Let $G:=C_{4}\times C_{2}$ and $\cH:=\{C_{4}\times \{1\},\{1\}\times C_{2}\}$. Then the $G$-lattice $J_{G/\cH}$ is not quasi-invertible. 
\end{prop}

\begin{proof}
Write
\begin{equation*}
    G=\langle \sigma,\tau\mid \sigma^{4}=\tau^{2}=1,\sigma\tau=\tau\sigma\rangle, 
\end{equation*}
and put $H_{1}:=\langle \sigma\rangle$ and $H_{2}:=\langle \tau \rangle$. We may assume $\cH=\{H_{1},H_{2}\}$. Let $I=I_{G/\cH}$ and $J=J_{G/\cH}$. We regard $I$ as a sublattice of $\Z[G/H_{1}] \oplus \Z[G/H_{2}]$. Then one has an exact sequence
\begin{equation}\label{eq:42mn}
    0\rightarrow I \rightarrow \Z[G/H_{1}]\oplus \Z[G/H_{2}] \xrightarrow{(\varepsilon_{G/H_{1}},\,\varepsilon_{G/H_{2}})} \Z \rightarrow 0. 
\end{equation}
Moreover, we have
\begin{equation*}
    I=\langle (1-\tau,0),(0,1-\sigma),(0,\sigma(1-\sigma)),(0,\sigma^{2}(1-\sigma)),(1,-1) \rangle_{\Z}. 
\end{equation*}
In particular, the following hold: 
\begin{gather*}
    I^{\langle \sigma^{2} \rangle}=\langle (1-\tau,0), (1+\tau,-(1+\sigma^{2})),(1+\tau,-\sigma(1+\sigma^{2}))\rangle_{\Z},\\
    I^{\langle \tau\rangle}=\langle (0,1-\sigma),(0,\sigma(1-\sigma)),(0,\sigma^{2}(1-\sigma)),(1+\tau,-(1+\sigma^{2}))\rangle_{\Z},\\
    I^{\langle \sigma^{2}\tau \rangle}=I^{\langle \sigma^{2},\tau \rangle}=\langle (1+\tau,-(1+\sigma^{2})),(1+\tau,-\sigma(1+\sigma^{2}))\rangle_{\Z}=\Z[G/\langle \sigma^{2},\tau \rangle](1+\tau,-(1+\sigma^{2})),\\
    I^{\langle \sigma \rangle}=\langle (1-\tau,0),(2(1+\tau),-(1+\sigma+\sigma^{2}+\sigma^{3})) \rangle_{\Z},\\
    I^{\langle \sigma\tau\rangle}=I^{G}=\langle (2(1+\tau),-(1+\sigma+\sigma^{2}+\sigma^{3}))\rangle_{\Z}.
\end{gather*}
Now, consider homomorphisms of $G$-lattices as follows: 
\begin{align*}
f_{1}\colon \Z[G]\rightarrow I&;\,1\mapsto (1,-1),\\
f_{2}\colon \Z[G/\langle \tau \rangle]\rightarrow I&;\,1\mapsto (0,1-\sigma),\\
f_{3}\colon \Z[G/\langle \sigma^{2},\tau \rangle]\rightarrow I&;\,1\mapsto (1+\tau,-(1+\sigma^{2})),\\
f_{4}\colon \Z[G/\langle \sigma \rangle]\rightarrow I&;\,1\mapsto  (1-\tau,0). 
\end{align*}
Then the sum of these maps induces a coflabby resolution
\begin{equation*}
0\rightarrow U \rightarrow R \xrightarrow{f} I \rightarrow 0
\end{equation*}
of $I$ which admits an isomorphism
\begin{equation}\label{eq:stfx}
U^{\langle \sigma^{2}\tau \rangle}\cong \Z[G/\langle \sigma^{2}\tau \rangle]\oplus \Z[G/\langle \sigma^{2},\tau \rangle]\oplus \Z. 
\end{equation}

Now, suppose that $J$ is quasi-invertible. Since $\rk_{\Z}(U)=11$ and $\rk_{\Z}(U^{G})=3$, there exist a subgroup $H$ of $G$ of order $4$ and an isomorphism
\begin{equation*}
    U_{2}\cong \Z_{2}[G]\oplus \Z_{2}[G/H]\oplus \Z_{2}. 
\end{equation*}
In particular, we obtain an isomorphism
\begin{equation*}
    \Hhat^{0}(G,U)\cong \Z/4\Z \oplus \Z/8\Z. 
\end{equation*}
However, since $\Hhat^{0}(G,R)\cong \Z/2\Z\oplus (\Z/4\Z)^{\oplus 2}$, the same argument as Claim in Proposition \ref{prop:apt1} implies that the exponent of $\Hhat^{0}(G,U)$ is a divisor of $4$. Hence, we obtain a contradiction, and the proof is complete. 
\end{proof}

\section{Proof of Theorem \ref{mth2}}\label{sect:pfev}

We give a proof of Theorem \ref{mth2} by dividing into four steps.

\subsection{First step: The case for groups of order $8$}\label{pf21}

\begin{prop}\label{prop:fixd}
Let $G$ be a $p$-group, where $p$ is a prime number. Consider a reduced set $\cH$ of subgroups of $G$ which contains a normal subgroup $H_{0}$ of $G$ of index $\mu(\cH)$ such that $G/H_{0}$ is not cyclic. Then the $G$-lattice $J_{G/\cH}$ is not quasi-invertible. 
\end{prop}

\begin{proof}
Let $\varphi$ be the weight function on $\cH$ that takes the value $1$ for all $H\in \cH$. By Proposition \ref{prop:tkfx}, one has an isomorphism
\begin{equation*}
J_{G/\cH}^{[H_{0}]}\cong J_{(G/H_{0})/\cH^{H_{0}}}^{(\overline{\varphi}_{G/H_{0}})}. 
\end{equation*}
Note that $\overline{\varphi}_{G/H_{0}}$ is normalized since $\overline{\varphi}_{G/H_{0}}(\{1\})=1$. Moreover, since $(G:H_{0})=\mu(H)$, all the factors of $\overline{\varphi}_{G/H_{0}}(\overline{H})$ are divisible by $(\pi^{-1}(\overline{H}):H_{0})$ for any $\overline{H}\in \cH^{H_{0}}$. Here $\pi$ denotes the natural surjection from $G$ onto $G/H_{0}$. Therefore Corollary \ref{cor:rdmn} implies the existence of permutation $G$-lattices $R_{1}$ and $R_{2}$ and an isomorphism of $G$-lattices
\begin{equation*}
J_{(G/H_{0})/\cH^{H_{0}}}^{(\overline{\varphi}_{G/H_{0}})}\oplus R_{1}\cong J_{G/H_{0}}\oplus R_{2}. 
\end{equation*}
However, the $G$-lattice $J_{G/H_{0}}$ is not quasi-invertible, which is a consequence of Proposition \ref{prop:emn1}. Hence the assertion follows from Proposition \ref{prop:tkfx}. 
\end{proof}

\begin{lem}\label{lem:nocy}
Let $p$ be a prime number, and $G$ a $p$-group that admits its non-cyclic maximal subgroup. Consider a reduced set $\cH$ of subgroups of $G$. If $\cH$ contains all maximal subgroups that are not cyclic, then $N^{G}(\cH)$ contains $\Phi(G)$. 
\end{lem}

\begin{proof}
If all maximal subgroups of $G$ are not cyclic, then the assertion follows from the definition of $\Phi(G)$. From now on, assume that $G$ has a cyclic maximal subgroup. Pick $H\in \cH$ which is contained in no non-cyclic maximal subgroup of $G$. It suffices to prove that $H$ is maximal in $G$. Take a maximal subgroup $P$ of $G$ containing $H$. Then the assumption on $\cH$ implies that $P$ is cyclic. On the other hand, Proposition \ref{prop:mxps} implies that $P/\Phi(G)$ is elementary $p$-abelian. Combining this with the cyclicity of $P$, we obtain that $\Phi(G)$ contains the unique subgroup of $P$ of index $p$. Hence, we must have $H=P$ as desired, which follows from $H\not\subset \Phi(G)$. 
\end{proof}

\begin{prop}\label{prop:c2nr}
Let $G:=(C_{2})^{3}$, and $\cH$ be a reduced set of its subgroups with $\#\cH \geq 2$ and $N^{G}(\cH)=\{1\}$. Then the $G$-lattice $J_{G/\cH}$ is not quasi-invertible. 
\end{prop}

\begin{proof}
For $i\in \{2,4\}$, put
\begin{equation*}
r_{i}:=\#\{H\in \cH \mid (G:H)=i\}. 
\end{equation*}
By Propositions \ref{prop:endo} and \ref{prop:fixd}, we may assume that $r_{2}$ and $r_{4}$ are greater than $0$. 

\textbf{Case 1.~$r_{2}=r_{4}=1$. }

In this case, the assertion follows from Proposition \ref{prop:apt1}. 

\textbf{Case 2.~$r_{4}\geq 2$. }

Pick $H_{0},H_{1},H_{2}\in \cH$ which satisfy $(G:H_{0})=2$ and $(G:H_{1})=(G:H_{2})=4$. Put $P:=H_{1}H_{2}$, which is isomorphic to $(C_{2})^{2}$. Then one has $(P:P\cap H_{0})=2$ and 
\begin{equation*}
(P\cap H_{0})\cap H_{1}=(P\cap H_{0})\cap H_{2}=\{1\}. 
\end{equation*}
In particular, we have $\cH_{P}^{\red}=\{P\cap H_{0},H_{1},H_{2}\}$. On the other hand, Proposition \ref{prop:endo} implies that the $P$-lattice $J_{P/\cH_{P}^{\red}}$ is not quasi-invertible. Therefore, the assertion follows from Proposition \ref{prop:rtmc}. 

\textbf{Case 3.~$r_{2}\geq 2$. }

Write $G=\langle \rho,\sigma,\tau \mid \rho^{2}=\sigma^{2}=\tau^{2}=1,\rho\sigma=\sigma\rho,\sigma\tau=\tau\sigma,\tau\rho=\rho\tau\rangle$. Take $H_{1},H_{2},H_{3}\in \cH$ with $(G:H_{1})=(G:H_{2})=2$ and $(G:H_{3})=4$. We may assume
\begin{equation*}
H_{1}=\langle \rho,\sigma \rangle,\quad H_{2}=\langle \sigma,\tau \rangle,\quad H_{3}=\langle \rho\tau\rangle. 
\end{equation*}
Now put $P:=\langle \rho\sigma,\sigma\tau \rangle$, which is isomorphic to $(C_{2})^{2}$. Note that $P$ is not contained in $\cH$ since $H_{3}\subset P$ and $\cH$ is reduced. Then we have
\begin{equation*}
P\cap H_{1}=\langle \rho\sigma \rangle,\quad P\cap H_{2}=\langle \sigma\tau \rangle,\quad P\cap H_{3}=H_{3}=\langle \rho\tau \rangle. 
\end{equation*}
Therefore, the assertion follows from the same argument as Case 2. 
\end{proof}

\begin{prop}\label{prop:c4nr}
Let $G:=C_{4}\times C_{2}$, and $\cH$ a reduced set of its subgroups with $\#\cH \geq 2$ and $N^{G}(\cH)=\{1\}$. Then the $G$-lattice $J_{G/\cH}$ is not quasi-invertible. 
\end{prop}

\begin{proof}
By Proposition \ref{prop:endo}, we may assume $M(\cH)=4$. If $\mu(\cH)=2$, then we have $\Phi(G) \notin \cH$. On the other hand, if $\mu(\cH)=4$ and $\Phi(G)\in \cH$, then Proposition \ref{prop:fixd} implies that the $G$-lattice $J_{G/\cH}$ is not quasi-invertible. Therefore, we may assume that $\cH$ contains a subgroup $H_{0}$ of index $4$ such that $G/H_{0}$ is cyclic. Write $G=\langle \sigma,\tau\mid \sigma^{4}=\tau^{2}=1,\sigma\tau=\tau\sigma \rangle$, and set
\begin{equation*}
N:=\langle \sigma^{2},\tau \rangle. 
\end{equation*}
Since $N$ is the unique non-cyclic maximal subgroup of $G$, it must not be contained in $\cH$ by Lemma \ref{lem:nocy}. Consequently, we may consider the case where $\cH$ satisfies at least one of (i)--(iv) as follows: 
\begin{enumerate}
\item $\cH=\{\langle \sigma \rangle,\langle \tau \rangle\}$; 
\item $\cH=\{\langle \sigma \rangle,\langle \sigma \tau \rangle,\langle \tau \rangle\}$; 
\item $\cH=\{\langle \tau\rangle,\langle \sigma^{2}\tau\rangle \}$; 
\item $\{\langle \tau \rangle,\langle \sigma^{2}\tau \rangle \}\subsetneq \cH \subset \{\langle \tau \rangle,\langle \sigma^{2}\tau \rangle,\langle \sigma \rangle,\langle \sigma \tau \rangle,\langle \sigma^{2} \rangle \}$. 
\end{enumerate}
In the following, we prove the assertion case-by-case. 

\textbf{Case (i): }In this case, the assertion follows from Proposition \ref{prop:apt2}. 

\textbf{Case (ii): }It suffices to prove that there is an isomorphism
\begin{equation}\label{eq:leei}
\Sha_{\omega}^{2}(G,J_{G/\cH})\cong \Z/2\Z. 
\end{equation}
In fact, if it holds, then Proposition \ref{prop:rrs2} gives the desired assertion. Put
\begin{equation*}
H_{0}:=\langle \sigma \rangle,\quad H_{1}:=\langle \sigma \tau \rangle,\quad H_{2}:=\langle \tau \rangle. 
\end{equation*}
By global class field theory, we can take a finite abelian extension $L/\Q$ with group $G$. Let $\bK:=\prod_{i=0}^{2}K_{i}$, where $K_{i}$ is the $H_{i}$-fixed subfield of $L$ for each $i\in \{0,1,2\}$. Then, combining Proposition \ref{prop:coch} with \cite[Lemma 3.3]{BayerFluckiger2020}, we get an isomorphism
\begin{equation*}
    \Sha_{\omega}^{2}(G,J_{G/\cH})\cong \Sha_{\omega}^{2}(\Q,X^{*}(T_{\bK/\Q})). 
\end{equation*}
Here, the right-hand side is defined as follows: 
\begin{equation*}
    \Sha_{\omega}^{2}(\Q,X^{*}(T_{\bK/\Q})):=\{x\in H^{2}(\Q,X^{*}(T_{\bK/\Q}))\mid \Res_{\Q_{v}/\Q}(x)=0\text{ for almost all places $v$ of $\Q$}\}. 
\end{equation*}
Therefore, we can apply the theory of \cite{Lee2022}. 
First, note that the numbering of elements of $\cH$ satisfies the assumption in \cite[p.~8, l.~6]{Lee2022}. Put $\cI:=\{1,2\}$. Since $H_{0}H_{1}=H_{0}H_{2}=G$, we have $U_{0}=\cI$. Furthermore, one has
\begin{equation*}
n_{l}(\cI):=\#(\cI/\sim_{l})=
\begin{cases}
1&\text{if }l=0,\\
2&\text{if }l>0. 
\end{cases}
\end{equation*}
In particular, we obtain 
\begin{itemize}
\item $L(\cI)=0$ (see \cite[p.~17, Section 5, l.~17--18]{Lee2022}); and 
\item $n_{l+1}(c)=1$ for every $l>0$ and $c\in \cI/\sim_{l}$. 
\end{itemize}
On the other hand, since $H_{0}$ contains $H_{1}\cap N=\Phi(G)$, the integer $f_{\cI}^{\omega}$ in \cite[p.~18, l.10--12]{Lee2022} must be $1$. Here we use the fact that $N$ is the unique subgroup of index $2$ containing $H_{2}$. Therefore \cite[Corollary 6.3 (1)]{Lee2022} gives an isomorphism
\begin{equation*}
\Sha_{\omega}^{2}(\Q,X^{*}(T_{K/\Q}))\cong (\Z/2\Z^{f_{\cI}^{\omega}})^{\oplus n_{1}(\cI)-1}=\Z/2\Z, 
\end{equation*}
which concludes the proof of \eqref{eq:leei}.

\textbf{Case (iii): }This is the same as \cite[Lemma 2.2]{Endo2011}. 

\textbf{Case (iv): }By assumption, we have 
$\cH_{N}^{\red}=\{\langle \sigma^{2}\rangle,\langle \tau\rangle,\langle \sigma^{2}\tau\rangle\}$. Hence, the assertion follows from the same argument as in Case 2 in the proof of Proposition \ref{prop:c2nr}. 
\end{proof}

\begin{prop}\label{prop:d4nr}
Let $G:=D_{4}$, and $\cH$ be a reduced set of its subgroups satisfying $\#\cH\geq 2$ and $N^{G}(\cH)=\{1\}$. We further assume that $\cH^{\nor}$ is non-empty. Then the $G$-lattice $J_{G/\cH}$ is not quasi-invertible. 
\end{prop}

\begin{proof}
Using the same argument as the proof of Proposition \ref{prop:c4nr}, we may assume $M(\cH)=4$ and $\langle \sigma_{4}^{2}\rangle \notin \cH$. Hence, it suffices to consider the case where $\cH$ contains $\langle \tau_{4} \rangle$. Combining this with the non-emptiness of $\cH^{\nor}$, we obtain that $\cH$ contains a subgroup $H_{0}$ of index $2$. Then $H_{0}$ coincides with $\langle \sigma_{4} \rangle$ or $\langle \sigma_{4}^{2},\sigma_{4}\tau_{4} \rangle$ since $\cH$ is reduced. Now, put $P:=\langle \sigma_{4}^{2},\tau_{4} \rangle$, which is isomorphic to $(C_{2})^{2}$. Then we have 
\begin{equation*}
    \cH_{P}^{\red}=\{\langle \sigma_{4}^{2}\rangle,\langle \tau_{4} \rangle,\langle \sigma_{4}^{2}\tau_{4} \rangle\}. 
\end{equation*}
Therefore, the same argument as Case 2 in the proof of Proposition \ref{prop:c2nr} gives the desired assertion. 
\end{proof}

In summary, we obtain the following. 

\begin{thm}\label{thm:ind1}
Let $G$ be a $2$-group of order $8$, and $\cH$ a reduced set of its subgroups satisfying $\#\cH \geq 2$ and $N^{G}(\cH)=\{1\}$. We further assume that $\cH^{\nor}$ is non-empty. Then the $G$-lattice $J_{G/\cH}$ is not quasi-invertible. 
\end{thm}

\begin{proof}
By the assumption on $\cH$ and the classification of groups of order $8$, we obtain that $G$ is isomorphic to $(C_{2})^{3}$, $C_{4}\times C_{2}$ or $D_{4}$. Hence, the assertion follows from Propositions \ref{prop:c2nr}, \ref{prop:c4nr} and \ref{prop:d4nr}. 
\end{proof}

\subsection{Second step: The case consisting of normal subgroups with indices at most $4$}\label{ssec:tgs2}

Here we generalize Theorem \ref{thm:ind1}, which will be needed in the final step. 

\begin{prop}\label{prop:anm2}
Let $G$ be a $2$-group, and $\cH$ a reduced set of its subgroups satisfying $\mu(\cH)=2$, $M(\cH)=4$ and $\cH=\cH^{\nor}$. Then the $G$-lattice $J_{G/H}$ is not quasi-invertible. 
\end{prop}

\begin{proof}
Take $H_{1},H_{2}\in \cH$ satisfying $(G:H_{1})=2$ and $(G:H_{2})=4$. Put $N:=H_{1}\cap H_{2}$, which has index $8$ in $G$. Write $\varphi$ for the weight function on $\cH$ that takes the value $1$ for every $H\in \cH$. Then Proposition \ref{prop:tkfx} gives an isomorphism
\begin{equation*}
J_{G/\cH}^{[N]}\cong J_{(G/N)/\cH^{N}}^{(\overline{\varphi}_{G/N})}. 
\end{equation*}
Note that $\overline{\varphi}_{G/N}$ is normalized since $\overline{\varphi}_{G/N}(H_{1}/N)=\overline{\varphi}_{G/N}(H_{2}/N)=1$. Let $\cH_{\overline{\varphi}_{G/N}}^{N}[1]$ be as in Corollary \ref{cor:rdmn}, that is,
\begin{equation*}
    \cH_{\overline{\varphi}_{G/N}}^{N}[1]:=\{H'\in \cH^{N} \mid d_{\overline{\varphi}_{G/N}}(H')=1\}. 
\end{equation*}
Denote by $H'_{i}$ the image of $H_{i}$ in $G/N$ for each $i\in \{1,2\}$. If $H'\in \cH^{N}\setminus \cH_{\overline{\varphi}_{G/N}}^{N}[1]$ satisfies $(G/N:H')=2$, then $M(\cH)=4$ gives an equality
\begin{equation*}
\overline{\varphi}_{G/N}(H')=(\underbrace{2,\ldots,2}_{m_{\cH^{N}}(H')}). 
\end{equation*}
Hence, we have $d_{\overline{\varphi}_{G/N}}(H')\in \{2,4\}\subset (H'_{1}:H'_{1}\cap H')\Z$. On the other hand, if $G/N\in \cH^{N}$, then $\overline{\varphi}_{G/N}(G/N)\in \Delta_{m_{\cH^{N}}(G/N)}$ is a sequence consisting of $2$ and $4$. Consequently, one has $d_{\overline{\varphi}_{G/N}}(G/N)\in \{2,4\}\subset (G/N:H'_{1})\Z$. Therefore, we can apply Corollary \ref{cor:rdmn}, and we obtain an equality
\begin{equation*}
[J_{(G/N)/\cH^{N}}^{(\overline{\varphi}_{G/N})}]=[J_{(G/N)/\cH_{\overline{\varphi}_{G/N}}^{N}[1]}]. 
\end{equation*}
However, the $G$-lattice $J_{(G/N)/\cH_{\overline{\varphi}_{G/N}^{N}}[1]}$ is not quasi-invertible by Theorem \ref{thm:ind1}. Combining this result with Proposition \ref{prop:tkfx}, we obtain the desired assertion. 
\end{proof}

\begin{lem}\label{lem:cyst}
Let $p$ be a prime number, and $G$ a $p$-group. Fix $m\in \Zpn$. Consider a multiset $\cH$ of subgroups of $G$ satisfying $N^{G}(\cH)=\{1\}$. Suppose 
\begin{enumerate}
    \item $H\triangleleft G$ and $G/H\cong C_{p^{m}}$ for every $H\in \cH$; and
    \item every maximal subgroup of $G$ contains an element of $\cH$. 
\end{enumerate}
Then the group $G$ is isomorphic to a finite direct product of copies of $C_{p^{m}}$. 
\end{lem}

\begin{proof}
We first note that (i) implies that $G$ is abelian of exponent a divisor of $p^{m}$. Assume that the conclusion fails. Then there is an element $g\in G$ whose order is smaller than $p^{m}$ so that $g\neq h^{p}$ for all $h\in G$. Take a maximal subgroup $P$ of $G$ which does not contain $g$ and $H\in \cH_{\subset P}$. Then the image $\overline{g}$ of $g$ in $G/H$ must be a generator. However, we have $\overline{g}^{p^{m-1}}=1$ since $g^{p^{m-1}}=1$. This contradicts the cyclicity of $G/H$, and hence the proof is complete. 
\end{proof}

\begin{thm}\label{thm:ind2}
Let $G$ be a $2$-group, and $\cH$ a reduced set of its subgroups satisfying $N^{G}(\cH)=\{1\}$, $\#\cH \geq 2$, $M(\cH)=4$ and $\cH^{\nor}=\cH$. Then the $G$-lattice $J_{G/H}$ is not quasi-invertible. 
\end{thm}

\begin{proof}
By Proposition \ref{prop:anm2}, we may assume $\mu(\cH)=4$. 

\textbf{Step 1. }Here, we prove the assertion in the case where all maximal subgroups of $G$ contain a member of $\cH$. By Proposition \ref{prop:fixd}, we may further assume that $G/H$ is cyclic for any $H\in \cH$. Then Lemma \ref{lem:cyst} gives an isomorphism $G \cong (C_{4})^{m}$ for some $m\in \Zpn$. Take a maximal subgroup $P$ of $G$, which is possible by assumption. Then it is clear that $\mu(\cH_{P}^{\red})=2$. 

In the following, we shall give $H_{0}\in \cH$ satisfying $P\cap H_{0}\in \cH_{P}^{\red}$ and $(P:P\cap H_{0})=4$. Fix elements $g_{1},\ldots,g_{m}$ of $G$ so that
\begin{equation*}
G=\langle g_{1},\ldots,g_{m}\mid g_{1}^{4}=\cdots =g_{m}^{4}=1,g_{i}g_{j}=g_{j}g_{i}\,(i\neq j)\rangle
\end{equation*}
and 
\begin{equation*}
    P=\langle g_{1}^{2},g_{2},\ldots,g_{m}\rangle. 
\end{equation*}
Pick $H'\in \cH_{P}^{\set}$ with $(P:H')=2$. Then there exist $a_{2},\ldots,a_{m}\in \{0,2\}$ such that
\begin{equation*}
    H'=\langle g_{1}^{a_{2}}g_{2},\ldots,g_{1}^{a_{m}}g_{m}\rangle. 
\end{equation*}
Now, let
\begin{equation*}
    P':=\langle g_{1}^{2},g_{1}g_{2},g_{3},\ldots,g_{m}\rangle,
\end{equation*}
and pick $H_{0}\in \cH_{\subset P'}$. We prove that this $H_{0}$ satisfies the desired properties. There exist $b_{2},\ldots,b_{m}\in \{0,2\}$ such that
\begin{equation*}
    H_{0}=\langle g_{1}^{b_{2}+1}g_{2},g_{1}^{b_{3}}g_{3}\ldots,g_{1}^{b_{m}}g_{m}\rangle. 
\end{equation*}
In particular, one has an equality
\begin{equation*}
    P\cap H_{0}=\langle g_{1}^{2}g_{2}^{2},g_{1}^{b_{3}}g_{3}\ldots,g_{1}^{b_{m}}g_{m}\rangle. 
\end{equation*}
Hence we obtain $P/(P\cap H_{0})\cong G/H_{0}\cong C_{4}$. Furthermore, this isomorphism implies that
\begin{equation*}
    H'_{0}:=\langle g_{1}^{2},g_{2}^{2},g_{3}\ldots,g_{m}\rangle
\end{equation*}
is the unique subgroup of $P$ of index $2$ containing $P\cap H_{0}$. Hence, if $P\cap H_{0}$ does not lie in $\cH_{P}^{\red}$, then $\mu(\cH)=4$ implies $H'_{0}\in \cH$. However, it is a contradiction since $G/H'_{0}$ is not cyclic. Therefore, we obtain $P\cap H_{0}\in \cH_{P}^{\red}$ as desired.

As above, we know that $\mu(\cH_{P}^{\red})=2$ and $M(\cH_{P}^{\red})=4$. Then Proposition \ref{prop:anm2} implies that the $P$-lattice $J_{P/\cH_{P}^{\red}}$ is not quasi-invertible. Hence, the assertion follows from Proposition \ref{prop:rtmc}. 

\textbf{Step 2. }Let $\#G=2^{\nu+1}$, where $\nu \geq 2$ is an integer. Here we give a proof of the assertion in general by induction on $\nu$. If $\nu=2$, then the claim is contained in Theorem \ref{thm:ind1}. In the following, suppose $\nu \geq 3$ and that the assertion holds for $\nu-1$. By Step 1, we may further assume that $\cH_{\subset P}$ is empty for some maximal subgroup $P$ of $G$. Then we have $\mu(\cH_{P})=M(\cH_{P})=4$. In particular, we have $\cH_{P}^{\set}=\cH_{P}^{\red}$ and $N^{P}(\cH_{P}^{\red})=P\cap N^{G}(\cH)=\{1\}$. On the other hand, since $\nu \geq 3$ implies $\#P=2^{\nu}>4$, we must have $\#\cH_{P}^{\red}\geq 2$. 
Now, by the induction hypothesis, the $P$-lattice $J_{P/\cH_{P}^{\red}}$ is not quasi-invertible. Therefore, the assertion follows from Proposition \ref{prop:rtmc}. 
\end{proof}

\subsection{Third step: The case admitting normal factors}

\begin{thm}\label{thm:ndq2}
Let $G$ be a $2$-group, and $\cH$ a reduced set of subgroups of $G$. We further assume that $\cH^{\nor}$ is non-empty. Then the following are equivalent: 
\begin{enumerate}
\item $J_{G/\cH}$ is a quasi-permutation $G$-lattice; 
\item $J_{G/\cH}$ is a quasi-invertible $G$-lattice; 
\item $G$ and $\cH$ satisfy one of the following: 
\begin{itemize}
    \item[(iii-1)] $\#\cH=1$ and $G/N^{G}(\cH)$ is cyclic; or
    \item[(iii-2)] $\#\cH=2$ and $G/N^{G}(\cH)\cong (C_{2})^{2}$. 
\end{itemize}
\end{enumerate}
\end{thm}

\begin{proof}
We may assume that $\cH$ contains at least two elements. In particular, $G$ is not cyclic according to Lemma \ref{lem:sglh}. Moreover, we may assume $M(\cH)>2$, which is a consequence of Proposition \ref{prop:endo}. Write $\#G=2^{\nu+1}$, which satisfies $\nu \geq 2$. It suffices to prove that $J_{G/H}$ is not quasi-invertible if $(G:N^{G}(\cH))\geq 8$ or $\#\cH \geq 3$. 

We give a proof of the above assertion by induction on $\nu$. If $\nu=2$, the claim follows from Theorem \ref{thm:ind1}. Now suppose $\nu \geq 3$ and the assertion holds for $\nu-1$. If $N^{G}(\cH)\neq \{1\}$, then the same argument as Theorem \ref{thm:ndqp} implies that the assertion follows from the induction hypothesis. Hence, we may further assume $N^{G}(\cH)=\{1\}$. By the induction hypothesis, it suffices to prove that there is a maximal subgroup $P$ of $G$ which satisfies $(P:N^{P}(\cH_{P}^{\red}))\geq 8$ or $\#\cH_{P}^{\red}\geq 3$. 

\textbf{Case 1.~$\cH=\cH^{\nor}$. }

By Theorem \ref{thm:ind2}, we may assume $M(\cH)\geq 8$. Take $H_{0}\in \cH$ with $(G:H_{0})=M(\cH)$. Using the same argument as Case 1 in the proof of Theorem \ref{thm:ndqp}, there is a maximal subgroup $P$ of $G$ satisfying $H_{0}\subset P$ and $\#\cH_{P}^{\red}\geq 2$. Then we have $M(\cH_{P}^{\red})\geq M(\cH)/2\geq 4$. Hence, the inequality $(P:N^{P}(\cH_{P}^{\red}))\geq 8$ follows from Lemma \ref{lem:quot}. 

\textbf{Case 2.~$\cH^{\nor}$ and $\cH\setminus \cH^{\nor}$ are non-empty. }

Note that one has $M(\cH)\geq 4$ by Proposition \ref{prop:mxps}. Consider positive integers as follows: 
\begin{equation}\label{eq:sgd2}
\Sigma:=\min\{(G:H)\in \Zpn \mid H\in \cH^{\nor}\},\quad 
\sigma:=\min\{(G:H)\in \Zpn \mid H\in \cH\setminus \cH^{\nor}\}. 
\end{equation}
Take $H_{0}\in \cH^{\nor}$ with $(G:H_{0})=\Sigma$ and $H_{1}\in \cH \setminus \cH^{\nor}$ with $(G:H_{1})=\sigma$. 

\textbf{Case 2-a.~$\Sigma<\sigma$. }
Pick a maximal subgroup $P$ of $G$ containing $N_{G}(H_{1})$. Then we have $gH_{1}g^{-1}\not\subset P\cap H_{0}\subsetneq P$ and $gH_{1}g^{-1}\not\supset P\cap H_{0}$ for any $g\in G$. Hence $\cH_{P}^{\red}$ has at least $3$ elements. 

\textbf{Case 2-b.~$\Sigma>\sigma$. }
Take a maximal subgroup $P$ of $G$ containing $H_{0}$. Then one has $H_{0}\not\subset P\cap H_{1}\subsetneq P$ and $H_{0}\not\supset P\cap H_{1}$. Hence, there exists an element of $\cH_{P}^{\red}$ that contains $P\cap H_{1}$. In particular, one has $\#\cH_{P}^{\red}\geq 2$. On the other hand, $\Sigma>\sigma$ implies $\Sigma \geq 8$. Therefore, we obtain the desired inequality
\begin{equation*}
(P:N^{P}(\cH_{P}^{\red}))\geq 2(P:H_{0})=\Sigma \geq 8. 
\end{equation*}

\textbf{Case 2-c.~$\Sigma=\sigma$. }
Let $P$ be a maximal subgroup of $G$ containing $N_{G}(H_{1})$. It suffices to prove that $\cH_{P}^{\red}$ is not a subset of $\{gH_{1}g^{-1}\mid g\in G\}$. This implies $\#\cH_{P}^{\red}\geq 3$ since $H_{1}$ is not normal in $G$. Assume not, that is, $\cH_{P}^{\red}$ is contained in $\{gH_{1}g^{-1}\mid g\in G\}$. Since $H_{0}$ is normal of index $\Sigma=\sigma=\mu(\cH)$, it is contained in $N_{G}(H_{1})$ by Lemma \ref{lem:nmlz}. In particular, $H_{0}$ is an element of $\cH_{P}^{\red}$, which contradicts the assumption. Hence the proof is complete. 
\end{proof}

\subsection{Final step: General case}

\begin{lem}\label{lem:abcj}
Let $G$ be a finite group, $H$ a subgroup of $G$, and $E$ an abelian normal subgroup of $G$. For $g_{1},g_{2}\in G$, assume $Eg_{1}H=Eg_{2}H$ in $E\backslash G/H$. Then we have $E\cap g_{1}Hg_{1}^{-1}=E\cap g_{2}Hg_{2}^{-1}$. 
\end{lem}

\begin{proof}
By assumption, we have $g_{2}=xg_{1}h$ for some $x\in E$ and $h\in H$. This implies an equality
\begin{equation*}
E\cap g_{2}Hg_{2}^{-1}=x(E\cap g_{1}Hg_{1}^{-1})x^{-1}. 
\end{equation*}
Now, the right-hand side equals $E\cap g_{1}Hg_{1}^{-1}$ since $E$ is abelian. Hence the proof is complete. 
\end{proof}

\begin{prop}\label{prop:mxmu}
Let $G$ be a group of order $2^{\nu+1}$ with $\nu \geq 2$. Consider a strongly reduced set $\cH$ of subgroups of $G$ satisfying $\#\cH \geq 2$, $\cH^{\nor}=\emptyset$ and $\mu(\cH)=\#G/2$. Then the following are equivalent: 
\begin{enumerate}
\item $J_{G/\cH}$ is a quasi-permutation $G$-lattice; 
\item $J_{G/\cH}$ is a quasi-invertible $G$-lattice; 
\item $G \cong D_{2^{\nu}}$ and $\cH=\{\langle \sigma_{2^{\nu}}^{2m}\tau_{2^{\nu}} \rangle, \langle \sigma_{2^{\nu}}^{2m'+1}\tau_{2^{\nu}} \rangle\}$ for some integers $m$ and $m'$. 
\end{enumerate}
\end{prop}

\begin{proof}
If $G\cong D_{2^{\nu}}$, then the assumption on $\cH$ implies $\cH=\{\langle \sigma_{2^{\nu}}^{2m}\tau_{2^{\nu}}\rangle,\langle \sigma_{2^{\nu}}^{2m'+1}\tau_{2^{\nu}}\rangle\}$ for some integers $m$ and $m'$. Hence the $G$-lattice $J_{G/\cH}$ is quasi-permutation by Theorem \ref{thm:dnqp}. Otherwise, by Proposition \ref{prop:clmc} and Lemma \ref{lem:nomx}, the group $G$ is not of maximal class. This implies that we can take a non-cyclic abelian normal subgroup $E$ of $G$ of order $8$, which is a consequence of Proposition \ref{prop:nmne}. 

\textbf{Case 1.~$\mu(\cH_{E}^{\red})=8$. }

In this case, we have $\cH_{E}^{\red}=\{\{1\}\}$ and $E/N^{E}(\cH_{E}^{\red})$ is not cyclic. Then Proposition \ref{prop:emn1} gives the desired assertion. 

\textbf{Case 2.~$\mu(\cH_{E}^{\red})=4$ and $E\cong (C_{2})^{3}$. }

By assumption, there is an isomorphism 
\begin{equation*}
E/N^{E}(\cH_{E}^{\red})\cong
\begin{cases}
    (C_{2})^{2}&\text{if }\#\cH_{E}^{\red}=1;\\
    (C_{2})^{3}&\text{if }\#\cH_{E}^{\red}\geq 2. 
\end{cases}
\end{equation*}
Then the assertion follows from Theorem \ref{thm:ndq2}. 

\textbf{Case 3.~$\mu(\cH_{E}^{\red})=4$ and $E\cong C_{4}\times C_{2}$. }

Write $E=\langle \sigma,\tau \mid \sigma^{4}=\tau^{2}=1,\sigma\tau=\tau\sigma\rangle$. By Proposition \ref{prop:nmne} (ii), $\Phi(E)$ is contained in the center of $G$. Pick $H'\in \cH_{E}^{\red}$. Then, $\mu(\cH)=\#G/2$ implies $H'=g_{1}Hg_{1}^{-1}$ for some $H\in \cH$ and $g_{1}\in G$. Moreover, $H'$ coincides with $\langle \tau \rangle$ or $\langle \sigma^{2}\tau \rangle$ since it is not normal in $G$. We may assume $H'=\langle \tau \rangle$. On the other hand, since $H'$ is not normal in $G$, the elements $\tau$ and $\sigma^{2}\tau$ are conjugate in $G$. In particular, there is $g_{2}\in G$ such that $g_{2}Hg_{2}^{-1}=\langle \sigma^{2}\tau\rangle$. Then one has $Eg_{1}H \neq Eg_{2}H$ in $E\backslash G/H$, which is a consequence of Lemma \ref{lem:abcj}. Therefore, $\cH_{E}^{\red}$ contains $\langle \tau \rangle$ and $\langle \sigma^{2}\tau \rangle$. In particular, we have $\#\cH_{E}^{\red}\geq 2$ and $N^{E}(\cH_{E}^{\red})=\{1\}$. Now, the assertion follows from Theorem \ref{thm:ndq2}. 
\end{proof}

\begin{lem}\label{lem:mxgt}
Let $G$ be a $2$-group, and $\cH$ a strongly reduced set of its subgroups. Assume
\begin{itemize}
\item $N^{G}(\cH)=\{1\}$; 
\item $\mu(\cH)=M(\cH)$; and
\item $(H:N^{G}(H))=(N_{G}(H):H)=2$ and $N^{G}(H)\neq \{1\}$ for any $H\in \cH$. 
\end{itemize}
Then there is a maximal subgroup $P$ of $G$ such that $\#\cH_{P}^{\srd}\geq 3$. 
\end{lem}

\begin{proof}
By assumption, there exist $H,H'\in \cH$ such that $N^{G}(H)\neq N^{G}(H')$. 

\textbf{Case 1.~$N^{G}(H)H'=G$. }

Let $P$ be a maximal subgroup of $G$ containing $N_{G}(H)$. Fix $g_{0}\in G \setminus P$. We may assume that $\cH_{P}^{\srd}$ contains $H$ and $g_{0}Hg_{0}^{-1}$. On the other hand, we have 
\begin{equation*}
    (G:N^{G}(H)\cdot (P\cap H'))\leq 2
\end{equation*}
since $N^{G}(H)H'=G$. Combining this inequality with the inclusion $N^{G}(H)\cdot (P\cap H')\subset P$, we obtain an equality
\begin{equation*}
N^{G}(H)\cdot (P\cap H')=P. 
\end{equation*}
In particular, $P\cap H'$ is not contained in $H$ or $g_{0}Hg_{0}^{-1}$. If $P\cap gH'g^{-1}\in \cH_{P}^{\srd}$ for some $g\in G$, then $\#\cH_{P}^{\srd}\geq 3$ is clear. Otherwise, there exist $g_{1},g_{2}\in G$ and $H_{0}\in \cH$ such that $P\cap g_{1}H'g_{1}^{-1}\subsetneq P\cap g_{2}H_{0}g_{2}^{-1}\in \cH_{P}^{\srd}$. Here, we use the normality of $P$ in $G$, which follows from Proposition \ref{prop:mxps}. Since $\#(P\cap H')=\#H'/2=\#H_{0}/2$, we have $P\cap g_{2}H_{0}g_{2}^{-1}=g_{2}H_{0}g_{2}^{-1}$, and in particular $g_{2}H_{0}g_{2}^{-1}\in \cH_{P}^{\srd}$. Therefore, one has $\#\cH_{P}^{\srd}\geq 3$, which completes the proof in this case. 

\textbf{Case 2.~$N^{G}(H)H'\neq G$. }

Let $P$ be a maximal subgroup of $G$ containing $N^{G}(H)H'$. Since $(N_{G}(H'):H')=2$, Corollary \ref{cor:pgfg} implies $N_{G}(H')\subset P$. If $H\subset P$, then the same argument as above yields $N_{G}(H)\subset P$. This implies $g_{0}Hg_{0}^{-1},g_{1}H'g_{1}^{-1},gHg^{-1}\in \cH_{P}^{\srd}$ for some $g_{0},g_{1}\in P$ and $g\in G\setminus P$. In particular, we obtain $\#\cH_{P}^{\srd}\geq 3$. On the other hand, if $H\not\subset P$, then one has $P\cap H=N^{G}(H)$ since $P$ contains $N^{G}(H)$. Therefore, the same argument as in Case 1 implies the desired assertion. 
\end{proof}

\begin{lem}\label{lem:qphc}
Let $G$ be a $2$-group, and $\cH$ a strongly reduced set of subgroups of $G$. Then the following are equivalent: 
\begin{enumerate}
\item $\cH=\{H,H'\}$ and there exist an integer $\nu \geq 2$ and an isomorphism $G/N^{G}(\cH) \cong D_{2^{\nu}}$ that induces $H/N^{G}(\cH)\cong \langle \sigma_{2^{\nu}}^{2m}\tau_{2^{\nu}} \rangle$ and $H'/{N^{G}(\cH)}\cong \langle \sigma_{2^{\nu}}^{2m'+1}\tau_{2^{\nu}} \rangle$ for some $m,m'\in \Zpn$;
\item $\cH$ satisfies all the conditions as follows: 
\begin{itemize}
\item $\#\cH=2$;
\item $(G:N^{G}(\cH))=2\mu(\cH)$; 
\item $G/N^{G}(\cH)$ is a $2$-group of maximal class; and
\item $(H:N^{G}(H))=(N_{G}(H):H)=2$ for all $H\in \cH$. 
\end{itemize}
\end{enumerate}
\end{lem}

\begin{proof}
(i) $\Rightarrow$ (ii) is clear by Lemma \ref{lem:nodt} (i). In the following, we prove the reverse implication. We may assume $N^{G}(\cH)=\{1\}$. Then $G$ is not abelian and of maximal class. Hence Proposition \ref{prop:clmc} yields one of the following: 
\begin{itemize}
\item[(a)] $G\cong D_{2^{\nu}}$ for some $\nu \geq 2$; 
\item[(b)] $G\cong SD_{2^{\nu+1}}$ for some $\nu \geq 3$; or
\item[(c)] $G\cong Q_{2^{\nu}}$ for some $\nu\geq 3$. 
\end{itemize}
Moreover, since $\#\cH=2$ and $(H:N^{G}(H))=2$ for all $H\in \cH$, only (a) is valid by Lemma \ref{lem:nomx}. Hence, the assertion follows from Lemma \ref{lem:nodt} (ii), (iii). 
\end{proof}

\begin{thm}\label{thm:tnqi}
Let $G$ be a $2$-group, and $\cH$ a strongly reduced set of its subgroups. Then, the following are equivalent:
\begin{enumerate}
\item $J_{G/\cH}$ is a quasi-permutation $G$-lattice; 
\item $J_{G/\cH}$ is a quasi-invertible $G$-lattice; 
\item $G$ and $\cH$ satisfy one of the following: 
\begin{itemize}
    \item[(iii-I)] $\#\cH=1$ and $G/N^{G}(\cH)$ is cyclic; or
    \item[(iii-I\hspace{-1pt}I)] $\cH=\{H,H'\}$ and there exist $\nu \in \Zpn$ and an isomorphism $G/N^{G}(\cH) \cong D_{2^{\nu}}$ that induces $H/N^{G}(\cH)\cong \langle \sigma_{2^{\nu}}^{2m}\tau_{2^{\nu}} \rangle$ and $H'/{N^{G}(\cH)}\cong \langle \sigma_{2^{\nu}}^{2m'+1}\tau_{2^{\nu}} \rangle$ for some $m,m'\in \Z$. 
\end{itemize}
\end{enumerate}
\end{thm}

\begin{proof}
If $\#\cH=1$, then the assertion follows from Proposition \ref{prop:emn1}. Hence, we may assume $\#\cH\geq 2$. 

(i) $\Rightarrow$ (ii): This is clear. 

(iii) $\Rightarrow$ (i): If (iii-I) is valid, then Proposition \ref{prop:emn1} gives the desired assertion. On the other hand, if (iii-I\hspace{-1pt}I) is satisfied, then the assertion follows from Theorem \ref{thm:dnqp}. 

(ii) $\Rightarrow$ (iii): It suffices to prove that the $G$-lattice $J_{G/\cH}$ is not quasi-invertible if (iii-I\hspace{-1pt}I) fails. Write $\#G=2^{\nu+1}$, which satisfies $\nu \geq 1$. We prove the assertion by induction on $\nu$. If $\nu \in \{1,2\}$, then the failure of (iii-I\hspace{-1pt}I) implies that $\cH^{\nor}$ is non-empty. Hence, the assertion follows from Theorem \ref{thm:ndq2}. In the following, suppose $\nu \geq 3$ and that the assertion holds for $\nu-1$. We may assume that $\cH^{\nor}$ is empty, which is a consequence of Theorem \ref{thm:ndq2}. Moreover, it suffices to discuss the case $N^{G}(\cH)=\{1\}$. In addition, we only need a consideration for $\mu(\cH)<\#G/2$, which is a consequence of Proposition \ref{prop:mxmu}. 

\textbf{Case 1.~$\mu(\cH)<M(\cH)$. }

Take $H_{0}\in \cH$ with $(G:H_{0})=M(\cH)$, and pick a maximal subgroup $P$ of $G$ containing $N_{G}(H_{0})$. Fix $g\in G\setminus P$. Then $P\cap H$ does not contain $H_{0}$ or $gH_{0}g^{-1}$ for any $H\in \cH$ with $(G:H)<M(\cH)$. This implies $\#\cH_{P}^{\srd}\geq 3$, and hence the $P$-lattice $J_{P/\cH_{P}^{\srd}}$ is not quasi-invertible by the induction hypothesis and Lemma \ref{lem:qphc}. Therefore the assertion follows from Proposition \ref{prop:rtmc}. 

\textbf{Case 2.~$(N_{G}(H):H)\geq 4$ for some $H\in \cH$. }

Let $P$ be a maximal subgroup of $G$ containing $N_{G}(H)$. Pick $g\in G\setminus P$. Then $H$ and $gHg^{-1}$ lie in $\cH_{P}^{\srd}$. Moreover, we have $N_{P}(H)=N_{G}(H)$ since $P$ contains $N_{G}(H)$. This implies $(N_{P}(H):H)\geq 4$, and hence the $P$-lattice $J_{P/\cH_{P}^{\srd}}$ is not quasi-invertible by the induction hypothesis and Lemma \ref{lem:qphc}. Hence the assertion follows from Proposition \ref{prop:rtmc}. 

\textbf{Case 3.~$(H:N^{G}(H))\geq 4$ for some $H\in \cH$. }

By assumption, we have $(G:H)\geq 8$. Take a maximal subgroup $P$ of $G$ containing $\Phi(G)H$. Then it also contains $N_{G}(H)$ by Corollary \ref{cor:pgfg} (ii). Moreover, we have $N_{G}(H)\neq P$ since $(G:H)\geq 8$. Now, pick $g\in G\setminus P$. Then $H$ and $gHg^{-1}$ are non-normal subgroups of $P$ that are not conjugate to each other. This is a consequence of the inclusion $N_{G}(H)\subsetneq P$ and the inequality $(G:H)\geq 8$. Moreover, they are contained in $\cH_{P}^{\red}$. In particular, we have $\#\cH_{P}^{\red}\geq 2$. If $\#\cH_{P}^{\srd}\geq 3$, then the $P$-lattice $J_{P/\cH_{P}^{\srd}}$ is not quasi-invertible by the induction hypothesis. In the following, suppose $\#\cH_{P}^{\srd}=2$, which implies $\cH_{P}^{\srd}=\{H,gHg^{-1}\}$ for some $g\in G\setminus P$. Hence, we have $N^{P}(\cH_{P}^{\srd})=N^{G}(H)$, and in particular
\begin{equation*}
    (P:N^{P}(\cH_{P}^{\srd}))=(P:H)(H:N^{G}(H))\geq 4\mu(\cH_{P}^{\srd}). 
\end{equation*}
Now, the $P$-lattice $J_{P/\cH_{P}^{\srd}}$ is not quasi-invertible by the induction hypothesis and Lemma \ref{lem:qphc}. Consequently, the $G$-lattice $J_{G/\cH}$ is not quasi-invertible by Proposition \ref{prop:rtmc}. 

\textbf{Case 4.~$\mu(\cH)=M(\cH)$ and $(H:N^{G}(H))=(N_{G}(H):H)=2$ for any $H\in \cH$. }

Since we assume $\mu(\cH)<\#G/2$, we have $N^{G}(H)\neq \{1\}$ for all $H\in \cH$. Hence, Lemma \ref{lem:mxgt} yields that there is a maximal subgroup $P$ of $G$ satisfying $\#\cH_{P}^{\srd}\geq 3$. Then the $P$-lattice $J_{P/\cH_{P}^{\srd}}$ is not quasi-invertible, which is a consequence of the induction hypothesis and Lemma \ref{lem:qphc}. Now, Proposition \ref{prop:rtmc} gives the desired assertion. 
\end{proof}

\begin{proof}[Proof of Theorem \ref{mth2}]
Put $G:=\Gal(L/k)$, and
\begin{equation*}
\cH:=\{\Gal(L/K_{1}),\ldots,\Gal(L/K_{r})\}. 
\end{equation*}
Then, we have $N^{G}(\cH)=\{1\}$. Moreover, Proposition \ref{prop:coch} gives an isomorphism $X^{*}(T_{\bK/k})\cong J_{G/\cH}$. Hence, Proposition \ref{prop:rtiv} implies that $T_{\bK/k}$ is stably (resp.~retract) rational over $k$ if and only if $J_{G/\cH}$ is a quasi-permutation (resp.~quasi-invertible) $G$-lattice. Moreover, by Proposition \ref{prop:rdsr}, the $G$-lattice $J_{G/\cH}$ is quasi-permutation (resp.~quasi-invertible) if and only if $J_{G/\cH^{\srd}}$ is so. On the other hand, we have the following:
\begin{itemize}
\item (iii-a) holds if and only if $\cH=\{\{1\}\}$ and $G$ is cyclic; and
\item (iii-b) holds if and only if $G\cong D_{2^{\nu}}$ with $\nu \in \Zpn$ and $\cH^{\srd}=\{\langle \sigma_{2^{\nu}}^{2m}\tau_{2^{\nu}}\rangle,\langle \sigma_{2^{\nu}}^{2m'+1}\tau_{2^{\nu}}\rangle\}$ for some $m,m'\in \Z$. 
\end{itemize}
Consequently, the assertion follows from Theorem \ref{thm:tnqi}. 
\end{proof}

\section{Proof of Theorem \ref{mth3}}\label{sect:nilp}

\begin{dfn}
Let $G$ be a finite group, and consider multisets $\cH$ and $\cH'$ of subgroups of $G$. We write $\cH'<\cH$ if
\begin{itemize}
\item $\cH'$ is a reduced set; and
\item there is an injection $\Xi \colon \cH' \hookrightarrow \cH^{\set}$ such that $\Xi(H')\subset H'$ for every $H'\in \cH'$. 
\end{itemize}
\end{dfn}

We can confirm the following without difficulty. 

\begin{lem}\label{lem:msin}
Let $G$ be a finite group, and $\cH$ a multiset of its subgroups. 
\begin{enumerate}
    \item We have $\cH^{\red}<\cH$. 
    \item Let $\cH'$ and $\cH''$ be reduced sets of subgroups of $G$ such that $\cH''<\cH'<\cH$. Then, we have $\cH''<\cH$. 
\end{enumerate}
\end{lem}

\begin{lem}\label{lem:psdp}
Let $G=G_{p}\times G'$, where $G_{p}$ is a $p$-group, and $G'$ is a finite group of order coprime to $p$. Consider a multiset $\cH$ of subgroups of $G$. Then, there is a reduced set $\cH_{0}$ of subgroups of $G$ such that
\begin{enumerate}
\item $\cH_{0}<\cH$; 
\item $(\cH_{0})_{G_{p}}^{\set}=\cH_{G_{p}}^{\red}$; and
\item $[J_{G/\cH_{0}}]=[J_{G/\cH}]$. 
\end{enumerate}
\end{lem}

\begin{proof}
By Corollary \ref{cor:rdrd} (ii) and Lemma \ref{lem:msin}, we may assume that $\cH$ is reduced. 
Note that all elements of $\cH_{G_{p}}$ are of the form $G_{p}\cap H$ for some $H\in \cH$. Now, pick $H_{0}\in \cH$ with $G_{p}\cap H_{0}\notin \cH_{G_{p}}^{\red}$. Take $H_{1}\in \cH$ so that $G_{p}\cap H_{0}\subset N_{1}:=G_{p}\cap H_{1}\in \cH_{G_{p}}^{\red}$. Then we have $H_{0}N_{1}\notin \cH$ since $\cH$ is reduced. Moreover, $H_{1}\cap H_{0}N_{1}$ is strictly contained in $H_{0}N_{1}$, which follows from $H_{0} \not\subset H_{1}$. Now we consider the multiset $\widetilde{\cH}$ of subgroups of $G$ which satisfies 
\begin{itemize}
\item $\widetilde{\cH}^{\set}=\cH \cup \{H_{0}N_{1},H_{1}\cap H_{0}N_{1}\}$; 
\item $m_{\widetilde{\cH}}(H_{0}N_{1})=2$; and
\item $m_{\widetilde{\cH}}(H)=1$ for each $H\in \widetilde{\cH}^{\set}\setminus \{H_{0}N_{1}\}$. 
\end{itemize}
Moreover, we write $\varphi$ for the weight function on $\widetilde{\cH}$ defined as follows: 
\begin{equation*}
\varphi(H)=
\begin{cases}
((H_{0}N_{1}:H_{0}),(H_{0}N_{1}:H_{1}\cap H_{0}N_{1}))&\text{if $H=H_{0}N_{1}$ and $\#H_{0}>\#(H_{1}\cap H_{0}N_{1})$;}\\
((H_{0}N_{1}:H_{1}\cap H_{0}N_{1}),(H_{0}N_{1}:H_{0}))&\text{if $H=H_{0}N_{1}$ and $\#H_{0}\leq \#(H_{1}\cap H_{0}N_{1})$; }\\
1&\text{otherwise. }
\end{cases}
\end{equation*}
Then Proposition \ref{prop:rdtr} gives an isomorphism of $G$-lattices
\begin{equation*}
J_{G/\widetilde{\cH}}^{(\varphi)} \cong 
\begin{cases}
J_{G/\cH}\oplus \Z[G/H_{0}N_{1}]^{\oplus 2}&\text{if }H_{1}\subset H_{0}N_{1};\\
J_{G/\cH}\oplus \Z[G/H_{0}N_{1}]^{\oplus 2}\oplus \Z[G/(H_{1}\cap H_{0}N_{1})]&\text{if }H_{1}\not\subset H_{0}N_{1}. 
\end{cases}
\end{equation*}
On the other hand, since $(H_{0}N_{1}:H_{0})$ is a power of $p$ and $(H_{0}N_{1}:H_{1}\cap H_{0}N_{1})$ is coprime to $p$, we obtain that $d_{\varphi}$ is the constant weight function of $\widetilde{\cH}^{\set}$ that takes value $1$. Therefore, if we set $\cH^{\dagger}:=\widetilde{\cH}^{\red}$, then Corollary \ref{cor:rdrd} gives an equality
\begin{equation*}
[J_{G/\widetilde{\cH}}^{(\varphi)}]=[J_{G/{\cH}^{\dagger}}]. 
\end{equation*}
In summary, one has
\begin{equation*}
[J_{G/\cH}]=[J_{G/\cH^{\dagger}}]. 
\end{equation*}
Note that the inclusions $H_{0} \subset H_{0}N_{1}$ and $H_{1}\cap H_{0}N_{1}\subset H_{0}N_{1}$ imply that $H_{0}$ and $H_{1}\cap H_{0}N_{1}$ are not contained in ${\cH}^{\dagger}$. In particular, we have ${\cH}^{\dagger}\setminus \{H_{0}N_{1}\}\subset \cH$. On the other hand, if $H_{0}N_{1}\subsetneq H'$ for some $H'\in \cH^{\dagger}$, then we have $H'\in \cH$ and $H_{0}\subsetneq H'$. This is absurd since $\cH$ is reduced, and hence $H_{0}N_{1}$ is contained in ${\cH}^{\dagger}$. Now, let
\begin{equation*}
\Xi \colon {\cH}^{\dagger}\rightarrow \cH;\,H\mapsto 
\begin{cases}
    H_{0}&\text{if }H=H_{0}N_{1}; \\
    H&\text{otherwise. }
\end{cases}
\end{equation*}
Then, it is injective and $\Xi(H)\subset H$ for any $H\in \widetilde{\cH}^{\red}$. Therefore, we obtain $\widetilde{\cH}^{\red}<\cH$. On the other hand, we have $({\cH}^{\dagger})_{G_{p}}^{\red}=\cH_{G_{p}}^{\red}$ and
\begin{equation*}
\#\{H\in \cH^{\dagger} \mid G_{p}\cap H\notin \cH_{G_{p}}^{\red}\}<
\#\{H\in \cH \mid G_{p}\cap H\notin \cH_{G_{p}}^{\red}\}.
\end{equation*}
Applying the above procedure for all $H_{0}\in \cH$ with $G_{p}\cap H_{0}\notin \cH_{G_{p}}^{\red}$, we obtain a reduced set of subgroups $\cH_{0}$ of $G$ satisfying (i), (ii) and (iii). 
\end{proof}

\begin{thm}\label{thm:mnnp}
Let $G$ be a finite nilpotent group, and $\cH$ a multiset of its subgroups. Then the following are equivalent: 
\begin{enumerate}
\item $J_{G/\cH}$ is a quasi-permutation $G$-lattice; 
\item $J_{G/\cH}$ is a quasi-invertible $G$-lattice; 
\item $[J_{G/\cH}]=[J_{G/\cH'}]$, where $\cH'$ is a reduced set of subgroups of $G$ with $\cH'<\cH$ that satisfies one of the following: 
\begin{itemize}
\item[(iii-$\alpha$)] $\cH'=\{N\}$, $N\triangleleft G$ and $G/N$ is cyclic; or
\item[(iii-$\beta$)] $\cH'=\{H,H'\}$ and there exist $m\in \Zpn \setminus 2\Z$, $\nu\in \Zpn$ and an isomorphism $G/N^{G}(\cH')\cong C_{m}\times D_{2^{\nu}}$ that induces $H/N^{G}(\cH')\cong \langle (1,\tau_{2^{\nu}})\rangle$ and $H'/N^{G}(\cH')\cong \langle (1,\sigma_{2^{\nu}}\tau_{2^{\nu}})\rangle$. 
\end{itemize}
\end{enumerate}
\end{thm}

\begin{proof}
(i) $\Rightarrow$ (ii): This is clear.

(ii) $\Rightarrow$ (iii): Since $G$ is nilpotent, we have $G=\prod_{p\mid \#G}G_{p}$, where $G_{p}$ is a $p$-Sylow subgroup of $G$. Hence, applying Lemma \ref{lem:psdp} for all prime divisors of $\#G$, we obtain a reduced set $\cH''$ of subgroups of $G$ that satisfies
\begin{itemize}
\item $\cH''<\cH$;
\item $(\cH''_{G_{p}})^{\set}=\cH_{G_{p}}^{\red}$ for all $p$; and
\item $[J_{G/\cH''}]=[J_{G/\cH}]$. 
\end{itemize}
Now, let $p$ be an odd prime divisor of $\#G$. Since $J_{G/\cH}$ is quasi-invertible as a $G_{p}$-lattice, Theorem \ref{thm:ndqp} implies that $(\cH''_{G_{p}})^{\set}$ consists of a single normal subgroup $N_{p}$ of $G_{p}$ with $G_{p}/N_{p}$ cyclic. Since $\prod_{2<p\mid \#G}N_{p}$ is normal in $G$, we may assume that $G_{p}$ is cyclic and $N_{p}=\{1\}$ for any odd prime divisor $p$ of $\#G$. In particular, we have
\begin{equation*}
    G\cong C_{m}\times G_{2}
\end{equation*}
for some odd positive integer $m$. Moreover, $\#H$ is a power of $2$ for any $H\in \cH''$. Now, set
\begin{equation*}
    \cH':=(\cH'')^{\srd}. 
\end{equation*}
Then $(\cH'_{G_{2}})^{\set}$ is also strongly reduced.  

\textbf{Case 1.~$\#\cH'=1$. }

In this case, Theorem \ref{thm:tnqi} implies that $(\cH'_{G_{2}})^{\set}$ satisfies (iii-I), that is, $(\cH'_{G_{2}})^{\set}$ consists of a normal subgroup $N$ of $G_{2}$ with $G_{2}/N$ cyclic. Moreover, we have $N\triangleleft G$ and $G/N$ is cyclic. Therefore, $\cH'$ satisfies (iii-$\alpha$). 

\textbf{Case 2.~$\#\cH' \geq 2$. }

In this case, Theorem \ref{thm:tnqi} asserts that (iii-I\hspace{-1pt}I) is valid for $(\cH'_{G_{2}})^{\set}$, that is, 
\begin{itemize}
\item $G_{2}/N^{G_{2}}(\cH') \cong D_{2^{\nu}}$ for some $\nu \in \Zpn$; 
\item $\cH'=\{H,H'\}$; 
\item $H/N^{G_{2}}(\cH')\cong \langle \sigma_{2^{\nu}}^{2m}\tau_{2^{\nu}} \rangle$ and $H'/{N^{G_{2}}(\cH')}\cong \langle \sigma_{2^{\nu}}^{2m'+1}\tau_{2^{\nu}} \rangle$ for some $m,m'\in \Z$. 
\end{itemize}
Since $N^{G_{2}}(\cH')$ is normal in $G$, we have $N^{G}(\cH')=N^{G_{2}}(\cH')$ and $G/N^{G}(\cH')\cong C_{m}\times D_{2^{\nu}}$. Consequently, $\cH'$ satisfies (iii-$\beta$). 

(iii) $\Rightarrow$ (i): By assumption, we may assume $\cH=\cH'$. Moreover, using Corollary \ref{cor:trfx}, it suffices to consider the case $N^{G}(\cH)=\{1\}$. If (iii-$\alpha$) holds, then the assertion follows from Proposition \ref{prop:emqp}. On the other hand, if (iii-$\beta$) holds, the assertion is a consequence of Theorem \ref{thm:nlqp}. 
\end{proof}

Now, we begin with the proof of Theorem \ref{mth3} using Theorem \ref{thm:mnnp}. Let $k$ be a field, and $\bK=\prod_{i=1}^{r}K_{i}$ a finite {\'e}tale algebra over $k$. A \emph{subalgebra} of $\bK$ refers to a finite product $\prod_{j=1}^{s}K'_{i_{j}}$, where $1\leq i_{1}<\cdots<i_{s}\leq r$ are integers and $K'_{i_{j}}$ is an intermediate field of $K_{i_{j}}/k$ for each $j\in \{1,\ldots,s\}$. 

\begin{proof}[Proof of Theorem \ref{mth3}]
It suffices to prove (ii) $\Rightarrow$ (iii) $\Rightarrow$ (i). Put $G:=\Gal(L/k)$ and 
\begin{equation*}
    \cH:=\{\Gal(L/K_{i})<G\mid i\in \{1,\ldots,r\}\}
\end{equation*}

(ii) $\Rightarrow$ (iii): Combining the assumption with Proposition \ref{prop:coch}, we obtain that the $G$-lattice $J_{G/\cH}$ is quasi-invertible. Hence, there exists a reduced set $\cH'$ of subgroups of $G$ with $\cH'<\cH$ and $[J_{G/\cH'}]=[J_{G/\cH}]$ that satisfies (iii-$\alpha$) or (iii-$\beta$) in Theorem \ref{thm:mnnp}. If (iii-$\alpha$) holds, denote by $\bK'$ the intermediate field of $L/k$ corresponding to $N$. Then, it is a subalgebra of $\bK$ satisfying (iii-a) in Theorem \ref{mth3}. On the other hand, if (iii-$\beta$) is true, write $\cH'=\{H,H'\}$, and denote by $K'_{1}$ and $K'_{2}$ the intermediate fields of $L/k$ corresponding to $H$ and $H'$, respectively. Let $\bK':=K'_{1}\times K'_{2}$. Then it is a subalgebra of $\bK$ satisfying (iii-b) in Theorem \ref{mth3}. Now, pick an isomorphism of $G$-lattices $J_{G/\cH}\oplus R\cong J_{G/\cH'}\oplus R'$, where $R$ and $R'$ are permutation $G$-lattices. Then there exist quasi-trivial $k$-tori $S$ and $S'$ with $X^{*}(S)\cong R$ and $X^{*}(S')\cong R'$, and we have a desired isomorphism $T_{\bK/k}\times S\cong T_{\bK'/k}\times S'$. 

(iii) $\Rightarrow$ (i): We may assume $\bK=\bK'$. If $\bK$ satisfies (iii-a), then (iii-$\alpha$) in Theorem \ref{thm:mnnp} is valid for $\cH$. On the other hand, if $\bK$ satisfies (iii-b), then (iii-$\beta$) in Theorem \ref{thm:mnnp} holds for $\cH$. In either case, Theorem \ref{thm:mnnp} implies that $J_{G/\cH}$ is quasi-permutation. Therefore, the assertion follows from Proposition \ref{prop:coch}. 
\end{proof}

\section{Proof of Theorem \ref{mth5}}\label{sect:pfmn}

In this section, we assume that $k$ is a global field. For a finite {\'e}tale algebra $\bK$ over $k$, let
\begin{equation*}
    \Sha(\bK/k):=(\N_{\bK/k}(\A_{\bK}^{\times})\cap k^{\times})/\N_{\bK/k}(\bK^{\times}). 
\end{equation*}
Here, $\A_{\bK}^{\times}$ is the product of the id{\`e}le groups of all the factors of $\bK$. We say that the \emph{multinorm principle holds for $\bK/k$} if
\begin{equation*}
    \Sha(\bK/k)=1. 
\end{equation*}

One can rephrase $\Sha(\bK/k)$ by means of the associated multinorm one torus. For an algebraic torus $T$ over $k$, the Tate--Shafarevich group of $T$ is defined as follows: 
\begin{equation*}
    \Sha^{1}(k,T):= \Ker \left(H^1(k,T)\xrightarrow{(\Res_{k_{v}/k})_{v}} \prod_{v} H^1(k_v,T)\right). 
\end{equation*}
Here, $v$ runs through all places of $k$, and $k_{v}$ denotes the completion of $k$ at $v$ (see also \cite[Section 6.3]{PR94}). 

\begin{lem}[{\cite[p.~7, (2.5)]{Liang2024}}]\label{lem:gnon}
Let $k$ be a global field, and $\bK$ a finite {\'e}tale algebra over $k$. Then, there is an isomorphism
\begin{equation*}
    \Sha(\bK/k)\cong \Sha^{1}(k,T_{\bK/k}). 
\end{equation*}
\end{lem}

Note that Lemma \ref{lem:gnon} is a generalization of Ono's theorem, which is given in \cite[p.~70]{Ono1963}. 

\begin{prop}[{\cite[p.~1213, Theorem 5]{Voskresenskii1969}}]\label{prop:vskr}
Let $k$ be a global field, and $T$ an algebraic torus over $k$. Then, there is an exact sequence
\begin{equation*}
0\rightarrow A_{k}(T) \rightarrow H^1(k,\Pic(\overline{X}))^{\vee}\rightarrow \Sha^{1}(k,T)\rightarrow 0. 
\end{equation*}
Here, $H^{1}(k,\Pic(\overline{X}))^{\vee}$ and $A_{k}(T)$ are defined as follows: 
\begin{itemize}
    \item $\overline{X}:=X\otimes_{k}k^{\sep}$, where $X$ is a smooth compactification of $T$ over $k$; 
    \item $H^{1}(k,\Pic(\overline{X}))^{\vee}$ is the Pontryagin dual of $H^{1}(k,\Pic(\overline{X}))$; 
    \item $A_{k}(T):=(\prod_{v}T(k_v))/\overline{T(k)}$ is the defect of the weak approximation of $T$.
\end{itemize}
\end{prop}

\begin{proof}[Proof of Theorem \ref{mth5}]
By Theorem \ref{mth3}, $T_{\bK/k}$ is stably rational over $k$. In particular, if we set $G:=\Gal(K_{1}K_{2}/k)$, then Proposition \ref{prop:rtiv} (i) implies that the $G$-lattice $X^{*}(T_{\bK/k})$ is quasi-permutation. Hence, we have $H^{1}(k,\Pic(\overline{X}))=0$ by Corollary \ref{cor:crvs}. Combining this result with Proposition \ref{prop:vskr}, we obtain $\Sha^{1}(k,T_{\bK/k})=0$. Now, the assertion follows from Lemma \ref{lem:gnon}. 
\end{proof}


\begin{thebibliography}{HHLY24}
\bibitem[BLP19]{BayerFluckiger2019}
E.~Bayer-Fluckiger, T.-Y.~Lee, R.~Parimala, \emph{Hasse principles for multinorm equations}, Adv.~Math.~\textbf{356} (2019), 106818. 
\bibitem[BP20]{BayerFluckiger2020}
E.~Bayer-Fluckiger, R.~Parimala, \emph{On unramified Brauer groups of torsors over tori}, Documenta Math.~\textbf{25} (2020), 1263--1284. 
\bibitem[Ber08]{Berkovich2008}
Y.~Berkovich, \emph{Groups of prime power order, I}, de Gruyter, Berlin, 2008.
\bibitem[CS77]{ColliotThelene1977}
J.-L.~Colliot-Th{\'e}l{\`e}ne, J.-J.~Sansuc, \emph{La $R$-{\'e}quivalence sur les tores}, Ann.~sci.~de l'{\'E}.~N.~S.~$4^{e}$ s{\'e}rie, \textbf{10} no.~2 (1977), 175--229. 
\bibitem[CS87]{ColliotThelene1987}
J.-L.~Colliot-Th{\'e}l{\`e}ne, J.-J.~Sansuc, \emph{Principal homogeneous spaces under flasque tori: applications}, J.~Algebra \textbf{106} (1987), no.~1, 148--205. 
\bibitem[CHS05]{ColliotThelene2005}
J.-L.~Colliot-Th{\'e}l{\`e}ne, D.~Harari, A.~N.~Skorobogatov, \emph{Compactification {\'e}quivariante d’un tore (d’apr{\`e}s Brylinski et K{\"u}nnemann)}, Expo.~Math.~\textbf{23} (2005), 161--170.
\bibitem[CK00]{CortellaKunyavskii2000}
A.~Cortella, B.~Kunyavskii, \emph{Rationality problem for generic tori in simple groups}, J.~Alg.~\textbf{225} (2000), 771--793.
\bibitem[DW14]{Demarche2014}
C.~Demarche, D.~Wei, \emph{Hasse principle and weak approximation for multinorm equations}, Israel J.~Math.~\textbf{202} (2014), 275--293. 
\bibitem[EM73]{Endo1973}
S.~Endo, T.~Miyata, \emph{Invariants of finite abelian groups}, J.~Math.~Soc.~Japan \textbf{25} (1973), No.~1, 7--26. 
\bibitem[EM75]{Endo1975}
S.~Endo, T.~Miyata, \emph{On a classification of the function fields of algebraic tori}, Nagoya Math.~J.~\textbf{56} (1975), 85--104. 
\bibitem[End01]{Endo2001}
S.~Endo, \emph{On the rationality of algebraic tori of norm type}, J.~Alg.~\textbf{235} (2001), 27--35. 
\bibitem[End11]{Endo2011}
S.~Endo, \emph{The rationality problem for norm one tori}, Nagoya Math.~J.~\textbf{202} (2011), 83--106. 
\bibitem[Flo]{Florence}
M.~Florence, \emph{Non rationality of some norm-one tori}, preprint (2006).
\bibitem[Hal59]{Hall1959}
M.~Hall, \emph{The theory of groups}, The Macmillan Company, New York, 1959. 
\bibitem[Has31]{Hasse1931}
H.~Hasse, \emph{Beweis eines Satzes und Wiederlegung einer Vermutung {\"u}ber das allgemeine Normenrestsymbol}, Nachr.~Ges.~Wiss.~G{\"o}tt., Math.-Phys.~Kl.~(1931), 64--69.  
\bibitem[HHY20]{Hasegawa2020}
S.~Hasegawa, A.~Hoshi, A.~Yamasaki, \emph{Rationality problem for norm one tori in small dimensions}, Math.~Comp.~\textbf{89} (2020), 923--940. 
\bibitem[Hir64]{Hironaka1964}
H.~Hironaka, \emph{Resolution of singularities of an algebraic variety over a field of characteristic zero. I, II}, Ann.~Math., (2) \textbf{79} (1964), 109--203; 205--326. 
\bibitem[HY17]{Hoshi2017}
A.~Hoshi, A.~Yamasaki, \emph{Rationality problem for algebraic tori}, Mem.~Amer.~Math.~Soc.~\textbf{248} (2017), no.~1176, v+215 pp. 
\bibitem[HY21]{Hoshi2021}
A.~Hoshi, A.~Yamasaki, \emph{Rationality problem for norm one tori}, Israel J.~Math.~\textbf{241} (2021), 849-867. 
\bibitem[HY24]{Hoshi2024}
A.~Hoshi, A.~Yamasaki, \emph{Rationality problem for norm one tori for dihedral extensions}, J.~Alg.~\textbf{640} (2024), 368--384. 
\bibitem[HY]{Hoshi}
A.~Hoshi, A.~Yamasaki, \emph{Rationality problem for norm one tori for $A_{5}$ and $\mathrm{PSL}_{2}(\mathbb{F}_{8})$-extensions}, preprint, arXiv:2309.16187, 2023. 
\bibitem[HHLY24]{Huang2024}
J.-H.~Huang, F.-Y.~Hung, P.-X.~Liang, C.-F.~Yu, \emph{Computing Tate-Shafarevich groups of multinorm one tori of Kummer type}, Taiwanese J.~Math.~\textbf{28} (2024), no.~4, 671--694.  
\bibitem[H{\"u}r84]{Hurlimann1984}
W.~H{\"u}rlimann, \emph{On algebraic tori of norm type}, Comment.~Math.~Helv., \textbf{59}, no.~4 (1984), 539--549. 
\bibitem[Kan12]{Kang2012}
M.~Kang, \emph{Retract rational fields}, J.~Algebra, \textbf{349} (2012), 22--37.
\bibitem[Kun90]{Kunyavskii1990}
B.~E.~Kunyavskii, \emph{Three-dimensional algebraic tori}, Selecta Math.~Soviet.~\textbf{9} (1990), 1--21. 
\bibitem[LeB95]{LeBruyn95}
L.~Le Bruyn, \emph{Generic norm one tori}, Nieuw Arch.~Wisk.~(4), \textbf{13} (1995), 401--407. 
\bibitem[Lee22]{Lee2022}
T.-Y.~Lee, \emph{The Tate-Shafarevich groups of multinorm-one tori}, J.~Pure Appl.~Alg., \textbf{226} (2022), 106906. 
\bibitem[LL00]{LemireLorenz2000}
N.~Lemire, M.~Lorenz, \emph{On certain lattices associated with generic division algebras}, J.~Group Theory, \textbf{3} (2000), 385--405.
\bibitem[Len74]{Lenstra1974}
H.~W.~Lenstra Jr., \emph{Rational functions invariant under a finite abelian group}, Invent.~Math.~\textbf{25} (1974), 299--325. 
\bibitem[LOY24]{Liang2024}
P.-X.~Liang, Y.~Oki, C.-F.~Yu, \emph{Cohomological properties of multinorm-one tori}, Results Math., \textbf{79} (2024), article number 277. 
\bibitem[Lor05]{Lorenz2005}
M.~Lorenz, \emph{Multiplicative Invariant Theory}, Encyclopaedia of Mathematical Sciences, vol.~135, Springer, 2005. 
\bibitem[Mer17]{Merkurjev2017}
A.~S.~Merkurjev, \emph{Invariants of algebraic groups and retract rationality of classifying spaces}, Algebraic groups: structure and actions, 277--294, Proc.~Sympos.~Pure Math., 94, Amer.~Math.~Soc., Providence, RI, 2017.
\bibitem[Mil17]{Milne2017}
J.~S.~Milne, \emph{Algebraic groups: The theory of group schemes of finite type over a field}, Cambridge Studies in Advanced Mathematics 170, Cambridge University Press, Cambridge, 2017. 
\bibitem[Ono63]{Ono1963}
T.~Ono, \emph{On Tamagawa numbers of algebraic tori}, Ann.~Math. (2) \textbf{78} (1963), 47--73. 
\bibitem[PR94]{PR94} 
V. P. Platonov, A. Rapinchuk, {\it Algebraic groups and number theory}, 
Translated from the 1991 Russian original by Rachel Rowen, 
Pure and applied mathematics, 139, Academic Press, 1994. 
\bibitem[Sal84]{Saltman1984}
D.~J.~Saltman, \emph{Retract rational fields and cyclic Galois extensions}, Israel J.~Math., \textbf{47} (1984), 165--215. 
\bibitem[San81]{Sansuc1981}
J.-J.~Sansuc, \emph{Groupe de Brauer et arithm{\'e}tique des groupes alg{\'e}briques lin{\'e}aires sur un corps de nombres}, (French) J.~Reine Angew.~Math., \textbf{327} (1981), 12--80. 
\bibitem[Ser77]{Serre1977}
J.-P.~Serre, \emph{Linear representations of finite groups}, Springer-Verlag, New York-Heidelberg, 1977, translated from the second French edition by L.~L.~Scott, Grad.~Texts in Math., Vol.~42. 
\bibitem[Vos67]{Voskresenskii1967}
V.~E.~Voskresenskii, \emph{On two-dimensional algebraic tori II}, Izv.~Akad.~Nauk SSSR Ser.~Mat.~\textbf{31} (1967), 711--716; translation in Math.~USSR-Izv.~\textbf{1} (1967), 691--696. 
\bibitem[Vos69]{Voskresenskii1969}
V.~E.~Voskresenskii, \emph{The birational equivalence of linear algebraic groups}, Dokl. Akad.~Nauk SSSR ~\textbf{188} (1969), 978--981; translation in Soviet Math. Dokl.~\textbf{10} (1969) 1212–1215.
\bibitem[Vos98]{Voskresenskii1998}
V.~E.~Voskresenskii, \emph{Algebraic groups and their birational invariants}, translated from the Russian manuscript by B.~Kunyavskii, Trans.~Math.~Monographs, 179, Amer.~Math.~Soc.~Providence, RI, 1998. 
\end{thebibliography}
\end{document}